\documentclass[10pt]{article}

\usepackage{amsfonts}
\usepackage{amssymb}
\usepackage{amsthm}
\usepackage{amsmath}
\usepackage{bm}
\usepackage{amscd} 
\usepackage{graphicx}
\usepackage{mathrsfs}
\usepackage{relsize}
\usepackage{amsrefs}
\usepackage{comment}
\usepackage[all]{xy}
\usepackage{MnSymbol}

\usepackage{fullpage}
\usepackage{hyperref}

\numberwithin{equation}{section} \DeclareMathSizes{2}{10}{12}{13}
\newtheorem{theorem}{Theorem}[section]
\newtheorem{lemma}[theorem]{Lemma}

\newtheorem{prop}[theorem]{Proposition}
\newtheorem{corollary}[theorem]{Corollary}
\newtheorem{remark}[theorem]{Remark}
\newtheorem{definition}[theorem]{Definition}

\numberwithin{equation}{section}

\title{Cycles over DGH-semicategories and pairings in categorical Hopf-cyclic cohomology}

\author{Mamta Balodi\footnote{Email: mamta.balodi@gmail.com} \footnote{MB was supported by 
SERB
Fellowship PDF/2017/000229}  $\qquad$ Abhishek Banerjee \footnote{Email: abhishekbanerjee1313@gmail.com} \footnote{AB was partially supported by SERB Matrics fellowship MTR/2017/000112}}

\date{}

\begin{document}

\maketitle

\centerline{\emph{Department of Mathematics, Indian Institute of Science, Bangalore - 560012, India.}}

\begin{abstract}
Let $H$ be a Hopf algebra and let $\mathcal D_H$ be a Hopf-module category.  We describe the cocycles and coboundaries for the Hopf cyclic cohomology
of $\mathcal D_H$, which correspond respectively to categorified cycles and vanishing cycles over $\mathcal D_H$. An important role in our work is played
by semicategories, which are categories that may not contain identity maps.   In particular, a cycle over $\mathcal D_H$ consists of a differential graded $H$-module semicategory equipped with a trace on endomorphism groups satisfying some conditions. Using a pairing on cycles, we obtain
a pairing $HC^p(\mathcal{C}) \otimes HC^q(\mathcal{C}') \longrightarrow HC^{p+q}(\mathcal{C} \otimes \mathcal{C}')$ on cyclic cohomology groups for small $k$-linear categories $\mathcal C$ and $\mathcal C'$. 

\end{abstract}

\medskip
{\bf MSC(2010) Subject Classification:} 16T05,  18E05, 58B34 


\section{Introduction}

In  \cite{CM0}, \cite{CM1}, \cite{CM2}, Connes and Moscovici introduced Hopf-cyclic cohomology as a generalization of Lie algebra cohomology adapted to Noncommutative Geometry. Given a Hopf algebra $H$ that is equipped with a modular pair in involution $(\delta,\sigma)$ and acts on an algebra $A$, they constructed a characteristic map
\begin{equation}\label{qe1.1}
\gamma^\bullet: HC^\bullet_{(\delta,\sigma)}(H)\longrightarrow HC^\bullet(A)
\end{equation}
taking values in the cyclic cohomology $HC^\bullet(A)$ of $A$. Both Hochschild homology and the cyclic theory have  since been studied extensively in several categorical contexts (see, for instance, \cite{BoSt}, \cite{BoSt1}, \cite{KL},  \cite{KoSh}, \cite{LV}, \cite{Low}, \cite{carthy}). The purpose of this paper is to categorify the formalism of cycles, traces and vanishing cycles
of Connes \cite{C2} in the context of Hopf cyclic cohomology.

\smallskip
Let $k$ be a field. A Hopf-module category consists of a $k$-linear category $\mathcal D_H$ with $H$ acting on its morphism spaces in such a way that the composition on $\mathcal D_H$
is well-behaved with respect to the coproduct on $H$. This notion was introduced by Cibils and Solotar in \cite{CiSo}, where they constructed a Morita equivalence connecting the Galois coverings of a category to its smash extensions via a Hopf algebra. A small Hopf-module category may be treated as a ``Hopf module algebra with several objects,'' in the same way
as a small preadditive category plays the role of a ``ring with several objects'' in the sense of Mitchell \cite{Mit1}. In fact, the replacement of rings by small preadditive categories
has been widely studied in the literature (see, for instance, \cite{AB}, \cite{EV}, \cite{LV2}, \cite{Low2}, \cite{Xu1}, \cite{Xu2}). Further, cyclic modules associated to Hopf-module categories
have been studied by Kaygun and Khalkhali \cite{kk}, while the Hochschild-Mitchell cohomology of a Hopf-comodule category has been studied by Herscovich and Solotar \cite{HS}. This paper is also part of our larger program of studying Hopf-module categories as objects of independent interest, begun in \cite{BBR1}, \cite{BBR2}.

\smallskip
For a Hopf-module category $\mathcal D_H$,  we describe in this paper  the cocycles and coboundaries that determine  its Hopf cyclic cohomology groups  by extending Connes' original construction of cyclic cohomology from \cite{C1} and \cite{C2} in terms of cycles and closed graded traces on differential graded algebras. An important role in our paper is played by ``semicategories,'' which are categories that may not contain identity maps. This notion, introduced by Mitchell \cite{Mit}, is precisely what we need in order to categorify 
non unital algebras. We work with the Hopf cyclic cohomology groups  $HC^\bullet_H(\mathcal D_H,M)$ having coefficients in $M$, where $M$
is a stable anti-Yetter Drinfeld module in the sense of \cite{hkrs}. Accordingly, we interpret the cocycles $Z^\bullet_H(\mathcal D_H,M)$  and the coboundaries $B^\bullet_H(\mathcal D_H,M)$ as characters of differential graded $H$-module semicategories equipped with closed graded traces with coefficients in $M$. We therefore feel that the present article is the first step towards ``categorification'' of the Noncommutative Differential Geometry of Connes \cite{C2}. We continue this program by developing categorifications of Fredholm modules and associated Chern characters in \cite{BB}.

\smallskip
We now describe the paper in more detail. In Section 2, we introduce the notion of a $\delta$-invariant $\sigma$-trace on a left $H$-category $\mathcal{D}_H$ (see Definition 
\ref{Df2.2w}), where $(\delta,\sigma)$ is a modular pair in involution for the Hopf algebra $H$. Given such a trace, we  prove (see Theorem \ref{charmap}) that there is a characteristic map 
\begin{equation}
\gamma^\bullet: HC^\bullet_{(\delta,\sigma)}(H)\longrightarrow HC^\bullet(\mathcal D_H)
\end{equation}
from the Hopf-cyclic cohomology $HC^\bullet_{(\delta,\sigma)}(H)$ taking values in the ordinary cyclic cohomology $HC^\bullet(\mathcal D_H)$ of the category $\mathcal{D}_H$. When $\mathcal{D}_H$ is a left $H$-category with a single object, i.e., an $H$-module algebra $A$, this recovers the   characteristic map  $HC^\bullet_{(\delta,\sigma)}(H)\longrightarrow HC^\bullet(A)$ in \eqref{qe1.1}.

\smallskip
In Section 3, we study the $\delta$-invariant $\sigma$-traces on $\mathcal{D}_H$ in more detail. For this, we first take the Hopf-cyclic cohomology $HC^\bullet_H(\mathcal{D}_H,M)$ of an $H$-category $\mathcal{D}_H$ with coefficients in an SAYD module $M$.  We then show that $\delta$-invariant $\sigma$-traces on $\mathcal{D}_H$ are precisely the 0-cocycles, i.e., the elements of $Z_H^0(\mathcal{D}_H,{^\sigma}{k}_\delta)=HC_H^0(\mathcal{D}_H,{^\sigma}{k}_\delta)$ (see Proposition \ref{bijection}). Since a $\delta$-invariant $\sigma$-trace on $\mathcal{D}_H$ induces a map $\gamma^\bullet: HC^\bullet_{(\delta,\sigma)}(H) \longrightarrow HC^\bullet(\mathcal D_H)$, we obtain a pairing
\begin{equation}\label{1.4gh}
HC^n_{(\delta,\sigma)}(H) \otimes HC_H^0(\mathcal{D}_H,{^\sigma}{k}_\delta) \longrightarrow HC^n(\mathcal{D}_H)
\end{equation}

\smallskip
The pairing in \eqref{1.4gh} suggests that we  look for a similar pairing when the Hopf algebra $H$ is replaced by an $H$-module coalgebra $C$ and the SAYD module ${^\sigma}{k}_\delta$ is replaced by an arbitrary SAYD module $M$. In \cite{hkrs}, it was conjectured that there
is a general pairing between the Hopf-cyclic
cohomology of a module coalgebra and the Hopf-cyclic cohomology of a module algebra, a fact that was proved later by Khalkhali and Rangipour in \cite{kr}. In Theorem \ref{thm8.15}, we use methods similar to Rangipour \cite{Rangipu} to construct a pairing
\begin{equation}\label{rpi2xy}
HC^q_H(C,M) \otimes HC^p_H(\mathcal{D}_H,M) \longrightarrow HC^{p+q}(\mathcal{D}_H)
\end{equation} for $p$, $q\geq 0$. For related work  on pairings
and Hopf-cyclic cohomology, we refer the reader to \cite{AK}, \cite{MB}, \cite{hkrs2}, \cite{Hask}, \cite{Hask1}, \cite{kyg}, \cite{kyg1}, \cite{Rangipu}.

\smallskip
In Section 5, we provide  a description of the space $Z_H^\bullet(\mathcal{D}_H,M)$ of cocycles, for which we extend the formalism of Connes \cite{C2}. We first describe in detail the construction of the universal differential graded (DG)-semicategory associated to a small $k$-linear category.  We then consider DGH-semicategories which may be treated as  differential graded (not necessarily unital) $H$-module algebras with several objects.

\smallskip
Since  $\delta$-invariant $\sigma$-traces on $\mathcal{D}_H$ are precisely the 0-cocycles in the Hopf-cyclic cohomology $HC^\bullet_H(\mathcal{D}_H,{^\sigma}{k}_\delta)$, we are motivated
to define  more generally the $n$-dimensional closed graded $(H,M)$-traces on a DGH-semicategory $\mathcal{S}_H$ (see, Definition \ref{gradedtrace}). We then introduce  cycles $(\mathcal{S}_H,\hat{\partial}_H,\hat{\mathscr{T}}^H)$ over the $H$-category $\mathcal D_H$ using which we provide a description of $Z_H^\bullet(\mathcal{D}_H,M)$ in Theorem \ref{charcycl}. This result is an $H$-linear categorical version of Connes'  \cite[Proposition 1, p. 98]{C2}. We show that an element $\phi \in Z_H^n(\mathcal{D}_H,M)$ if and only if it is the character of an $n$-dimensional cycle over $\mathcal{D}_H$.
It also follows from Theorem \ref{charcycl} that there is a one to one correspondence between  $Z^n_H(\mathcal{D}_H,M)$ and the collection of $n$-dimensional closed graded $(H,M)$-traces on the universal DGH-semicategory $\Omega(\mathcal D_H)$ associated to $\mathcal{D}_H$. We then proceed to obtain a description of the space $B^\bullet_H(\mathcal{D}_H,M)$ of coboundaries. 

\smallskip 
In Section 6, we show that the Hopf-cyclic cohomology of an $H$-category $\mathcal D_H$ is the same as that of its linearization $\mathcal{D}_H \otimes M_r(k)$ by the matrix
ring $M_r(k)$. For this, we first construct a para-cyclic module $C_\bullet(\mathcal{D}_H,M)=\{M \otimes CN_n(\mathcal{D}_H)\}_{n \geq 0}$ using the cyclic nerve 
$CN_\bullet(\mathcal D_H)$. We also consider inclusion and trace maps
\begin{equation*}
({inc}_1,M):M \otimes  CN_n(\mathcal{D}_H)  \longrightarrow M \otimes CN_n\left(\mathcal{D}_H  \otimes M_r(k)\right) 
\end{equation*}
\begin{equation*}
tr^M:M \otimes CN_n\left(\mathcal{D}_H \otimes M_r(k)\right) \longrightarrow M \otimes CN_n(\mathcal{D}_H)  
\end{equation*}
Then, we show  in Proposition \ref{homotopy} that the induced morphisms 
\begin{equation*}
C_\bullet(inc_1,M)^{hoc}:C_\bullet(\mathcal{D}_H,M)^{hoc} \longrightarrow C_\bullet\left(\mathcal{D}_H \otimes M_r(k),M\right)^{hoc}
\end{equation*}
\begin{equation*}
C_\bullet({tr^M})^{hoc}:C_\bullet\left(\mathcal{D}_H \otimes M_r(k),M\right)^{hoc} \longrightarrow  C_\bullet(\mathcal{D}_H,M)^{hoc}
\end{equation*}
between the underlying Hochschild complexes are homotopy inverses of each other. Applying the functor $Hom_H(-,k)$, we show in Proposition \ref{moritainvcyc} that there are mutually inverse isomorphisms  of Hopf-cyclic cohomologies:
 \begin{equation}\label{1.6p}
  \xy (10,0)*{HC^\bullet_H(\mathcal{D}_H,M)}; (40,3)*{HC^{\bullet}_H(tr^M)}; (40,-3.3)*{HC^{\bullet}_H(inc_1,M)};  {\ar
(25,1)*{}; (60,1)*{}}; {\ar (60,-1)*{};(25,-1)*{}};
\endxy \quad HC^\bullet_H\left(\mathcal{D}_H \otimes M_r(k),M\right)
\end{equation} 
In Section 7, we provide a description of $B^\bullet_H(\mathcal{D}_H,M)$. Throughout, we take $k=\mathbb C$. We consider families $\eta$ of automorphisms $\eta=\{\eta(X)\in Aut_{\mathcal D_H}(X)\}_{X\in Ob(\mathcal D_H)}$ such that
\begin{equation*}
h(\eta(X))=\varepsilon(h)\eta(X) \qquad\forall\textrm{ }h\in H,X\in Ob(\mathcal D_H)
\end{equation*} We show that these families form a group, which we denote by $\mathbb U_H(\mathcal D_H)$. Further, we show that the inner automorphism
of $\mathcal D_H$ induced by conjugating with an element $\eta\in \mathbb U_H(\mathcal D_H)$, i.e., the functor which fixes objects and takes any morphism $f\in Hom_{\mathcal D_H}(X,Y)$ to $\eta(Y)\circ f\circ \eta(X)^{-1}$, induces the identity functor on $HC^\bullet(\mathcal D_H,M)$. Using this and the isomorphisms in \eqref{1.6p}, we obtain in Proposition \ref{vanishing} a set of sufficient conditions for the Hopf cyclic cohomology of an $H$-category to be zero. 

\smallskip
We say that a cycle $(\mathcal{S}_H,\hat{\partial}_H,\hat{\mathscr{T}}^H)$ is vanishing if $\mathcal{S}_H^0$ is an $H$-category and $\mathcal S_H^0$  satisfies the  assumptions in Proposition \ref{vanishing}. We describe the elements of $B^\bullet_H(\mathcal{D}_H,M)$ in Proposition \ref{cob1} as the characters of vanishing cycles over $\mathcal{D}_H$. Finally, in Theorem \ref{Thmfin}, we use categorified cycles and vanishing cycles to construct a pairing 
\begin{equation*}
HC^p(\mathcal{C}) \otimes HC^q(\mathcal{C}') \longrightarrow HC^{p+q}(\mathcal{C} \otimes \mathcal{C}')
\end{equation*} 
for $k$-linear small categories $\mathcal C$ and $\mathcal C'$.

\smallskip
\textbf{Notations:}
Throughout the paper, $H$ is a Hopf algebra over the field  $k$ of characteristic zero, with comultiplication $\Delta$, counit $\varepsilon$ and bijective antipode $S$. We will use  Sweedler's notation for the coproduct $\Delta(h)= h_1 \otimes h_2$ and for a left $H$-coaction $\rho:M \longrightarrow H \otimes M$, $\rho(m)= m_{(-1)} \otimes m_{(0)}$ (with the summation sign suppressed). The small cyclic category introduced by Connes in \cite{C1} will be denoted by $\Lambda$. The Hochschild differential will always  be denoted by $b$. 

\section{Categorified characteristic in Hopf-cyclic cohomology}\label{nerve}
It is well known that a ring can be identified with a preadditive category having a single object (see, for instance \cite{Sten}). Accordingly, any small preadditive category  may be treated as a ring with several objects in the sense of Mitchell (see \cite{Mit1,Mit2}). We now recall the notion of an $H$-category, introduced by Cibils and Solotar \cite{CiSo},  which may be considered as an ``$H$-module algebra with several objects."

\begin{definition}\label{defH-cat}
Let $H$ be a Hopf algebra over a field $k$. A $k$-linear category $\mathcal{D}_H$ is said to be a left $H$-module category if
\begin{itemize}
\item[(i)] $Hom_{\mathcal{D}_H}(X,Y)$ is a left $H$-module for all $X,Y \in Ob(\mathcal{D}_H)$ 
\item[(ii)] $h(\text{id}_X)=\varepsilon(h)\text{id}_X$ for all $X \in Ob(\mathcal{D}_H)$ and $h \in H$
\item[(iii)] the composition map is a morphism of $H$-modules, i.e., 
$$h(gf)=(h_1g)(h_2f)$$
for any $h \in H$, $f \in Hom_{\mathcal{D}_H}(X,Y)$ and  $g \in Hom_{\mathcal{D}_H}(Y,Z)$. 
\end{itemize} A small left $H$-module category will be called a left $H$-category.
\end{definition}

We now let $H$ be a Hopf algebra with a modular pair in involution $(\delta,\sigma)$ (see \cite{CM2}) and let $\mathcal{D}_H$ be a  left $H$-category. In this section, we introduce the notion of a $\delta$-invariant $\sigma$-trace on the category $\mathcal{D}_H$. Using this trace, we then construct a characteristic map from the cyclic cohomology of the Hopf algebra $H$ to that of the category $\mathcal{D}_H$. We first recall from \cite{CM1}, \cite{CM2} the cyclic cohomology of a Hopf algebra.

\smallskip
Let $\delta \in H^*$ be a character and $\sigma \in H$ be a group-like element (i.e., $\Delta(\sigma)=\sigma\otimes \sigma$ and $\varepsilon(\sigma)=1$) such that $\delta(\sigma)=1$. Then, the pair $(\delta,\sigma)$ is said to be a modular pair on $H$. The character $\delta$ determines a $\delta$-twisted antipode ${S}_\delta:H \longrightarrow H$ given by
$${S}_\delta(h)=\delta(h_1)S(h_2) \qquad \forall h \in H$$
The pair $(\delta,\sigma)$ is said to be a modular pair in involution if ${S}_\delta^2(h)=\sigma h\sigma^{-1}$ for all $h \in H$. 

\smallskip
Given a modular pair $(\delta,\sigma)$ in involution for a Hopf algebra $H$, one can associate a $\Lambda$-module $C^\bullet \left(H_{(\delta,\sigma)}\right)$  by setting $C^n(H):=H^{\otimes n}$, $\forall n \geq 1$ and 
$C^0(H)=k$. For $n>1$, the face maps $\delta_i:C^{n-1}(H) \longrightarrow C^n(H)$ are given as follows:
\begin{equation*}
\begin{array}{lll}
\delta_i(h^1 \otimes \ldots \otimes h^{n-1})&=& \begin{cases}1 \otimes h^1 \otimes \ldots \otimes h^{n-1} \quad~~~~~~~~~~~~~~~~~~ i=0\\
h^1 \otimes \ldots \otimes \Delta(h^i) \otimes \ldots \otimes h^{n-1} \quad~~~~~ 1 \leq i \leq n-1\\
h^1 \otimes \ldots \otimes h^{n-1} \otimes \sigma \quad~~~~~~~~~~~~~~~~~~ i=n
\end{cases}
\end{array}
\end{equation*}For $n=1$, we have $\delta_0(1)=1$ and $\delta_1(1)=\sigma$.
For $n>0$, the degeneracy maps $\sigma_i:C^{n+1}(H) \longrightarrow C^n(H)$ for $0 \leq i \leq n$ are given by
\begin{equation*}
\sigma_i(h^1 \otimes \ldots \otimes h^{n+1})= h^1 \otimes \ldots \otimes \varepsilon(h^{i+1}) \otimes \ldots \otimes h^{n+1}\\
\end{equation*} For $n=0$, we have  $\sigma_0(h)=\varepsilon(h)$. The cyclic operator $\tau_n:C^n(H) \longrightarrow C^n(H)$ is given by
\begin{equation*}
\tau_n(h^1 \otimes \ldots \otimes h^{n})= {S}_\delta(h^1)\cdot (h^2 \otimes \ldots \otimes h^{n} \otimes \sigma)
\end{equation*}
The cyclic cohomology determined by the $\Lambda$-module $C^\bullet \left(H_{(\delta,\sigma)}\right)$ is said to be the cyclic cohomology of the Hopf algebra $H$ with respect to the modular pair $(\delta,\sigma)$ and will be denoted by $HC^\bullet_{(\delta,\sigma)}(H)$.

\smallskip
We now recall the  cyclic cohomology of a small $k$-linear category due to McCarthy \cite{carthy}. 
The  additive cyclic nerve of a small $k$-linear category $\mathcal{C}$ is defined to be the cyclic module determined by
\begin{equation}\label{cnerf}CN_n(\mathcal{C}):=\bigoplus Hom_\mathcal{C}(X_1,X_0) \otimes Hom_\mathcal{C}(X_2,X_1) \otimes \ldots \otimes Hom_\mathcal{C}(X_0,X_n)
\end{equation}
where the direct sum runs over all $(X_0,X_1,\ldots,X_n) \in Ob(\mathcal{C})^{n+1}$. The structure maps are given by
\begin{equation*}
\begin{array}{lll}
d_i(f^0 \otimes \ldots \otimes f^{n})&= \begin{cases}
f^0 \otimes f^1 \otimes \ldots \otimes f^if^{i+1} \otimes \ldots \otimes f^{n}& 0 \leq i \leq n-1\\
f^nf^0 \otimes f^1 \otimes \ldots \otimes f^{n-1}& i=n\\
\end{cases}\\
\vspace{.2cm}
s_i(f^0 \otimes \ldots \otimes f^{n})&= \begin{cases} f^0 \otimes f^1 \otimes \ldots f^i \otimes id_{X_{i+1}}\otimes f^{i+1} \otimes \ldots \otimes f^{n}& 0 \leq i \leq n-1\\
f^0 \otimes f^1 \otimes \ldots \otimes f^{n} \otimes id_{X_{0}}& i=n\\
\end{cases}\\ 
\vspace{.2cm}
t_n(f^0 \otimes \ldots \otimes f^{n})&=f^n \otimes f^0 \otimes \ldots \otimes f^{n-1}
\end{array}
\end{equation*}
for any $f^0 \otimes \ldots \otimes f^{n} \in Hom_\mathcal{C}(X_1,X_0) \otimes Hom_\mathcal{C}(X_2,X_1) \otimes \ldots \otimes Hom_\mathcal{C}(X_0,X_n)$.
The cyclic cohomology groups of $\mathcal C$ are determined by the $k$-spaces $CN^n(\mathcal{C}):=Hom_k(CN_n(\mathcal{C}),k)$ with the structure maps 
\begin{align*}
\delta_i: &~CN^{n-1}(\mathcal{C}) \longrightarrow CN^n(\mathcal{C}), \quad  \phi \mapsto \phi\circ d_i\\
\sigma_i: &~CN^{n+1}(\mathcal{C}) \longrightarrow CN^n(\mathcal{C}), \quad \psi \mapsto \psi \circ s_i\\
\tau_n: &~CN^{n}(\mathcal{C}) \longrightarrow CN^n(\mathcal{C}), \quad~~~ \varphi \mapsto \varphi \circ t_n.
\end{align*}
We will use the notation $CN^\bullet(\mathcal{C}):=\{CN^n(\mathcal{C})\}_{n \geq 0}$ to denote the cocyclic module associated to the category $\mathcal{C}$ and $HC^\bullet(\mathcal{C})$ for the corresponding cyclic cohomology.

\smallskip
We now introduce the notion of $\delta$-invariant $\sigma$-traces on an $H$-category.
\begin{definition}\label{Df2.2w}
Let $(\delta,\sigma)$ be a modular pair for a Hopf algebra $H$ and let $\mathcal{D}_H$ be a left $H$-category. Suppose that we have a collection $T^H:=\{T_X^H:Hom_{\mathcal{D}_H}(X,X) \longrightarrow k\}_{X \in Ob(\mathcal{D}_H)}$ of $k$-linear maps. Then, we say that the collection $T^H$ is a $\sigma$-trace on $\mathcal{D}_H$ if 
\begin{equation}\label{tracecomm}
T_{X}^H(g\circ f)=T_{Y}^H\left(f\circ (\sigma g)\right)
\end{equation}
for any $f \in Hom_{\mathcal{D}_H}(X,Y)$ and $g \in Hom_{\mathcal{D}_H}(Y,X)$. Moreover, we say that the $\sigma$-trace $T^H$ is $\delta$-invariant under the action of $H$ if
\begin{equation}\label{1.2d}
T_X^H(hf')=\delta(h)T_X^H(f') \qquad \forall h \in H, 
\end{equation} for all $f' \in Hom_{\mathcal{D}_H}(X,X)$. 
\end{definition}

\begin{lemma}
Let $T^H$ be a $\sigma$-trace on a left $H$-category $\mathcal{D}_H$. Then, $T^H$ is $\delta$-invariant under the action of $H$ iff the following holds:
\begin{equation}
T_{X}^H((hg)\circ f)=T_{X}^H\left(g\circ (S_\delta(h)f)\right)  \label{invariance}
\end{equation}
for any $h \in H$, $f \in Hom_{\mathcal{D}_H}(X,Y)$ and $g \in Hom_{\mathcal{D}_H}(Y,X)$.
\end{lemma}
\begin{proof}
Let $T^H$ be $\delta$-invariant. Then, we have
\begin{equation*}
T_{X}^H\left(g(S_\delta(h)f)\right)=T_{X}^H\left(g(\delta(h_1)S(h_2)f)\right)=T_{X}^H\left(h_1\left(g(S(h_2)f)\right)\right)=T_{X}^H\left((h_1g)(h_2S(h_3)f)\right)=T_{X}^H((hg)f)
\end{equation*} 
Conversely, suppose that the collection $T^H$ satisfies \eqref{invariance}. Then, for any $f' \in Hom_{\mathcal{D}_H}(X,X)$, we have
\begin{equation*}
T_X^H(hf')=T_X^H\left(f'(S_\delta(h)id_X)\right)=T_X^H\left(\delta(h_1)f'(\varepsilon(S(h_2))id_X)\right)=\delta(h)T_X^H(f')
\end{equation*}
\end{proof}

We are now ready to prove that there is a characteristic map from the  cyclic cohomology of the Hopf algebra $H$ taking values in the  cyclic cohomology of the $H$-category $\mathcal{D}_H$.

\begin{theorem}\label{charmap}
Let $H$ be a Hopf algebra with a modular pair in involution $(\delta,\sigma)$. Let $\mathcal{D}_H$ be a left $H$-category and let $T^H$ be a $\delta$-invariant $\sigma$-trace on $\mathcal{D}_H$. Then, we have a characteristic map $\gamma^\bullet:C^\bullet \left(H_{(\delta,\sigma)}\right) \longrightarrow CN^\bullet(\mathcal{D}_H)$ of $\Lambda$-modules given by
\begin{equation*}
\big(\gamma^n(h^1 \otimes \ldots \otimes h^n)\big)(f^0 \otimes \ldots \otimes f^{n}):=T_{X_0}^H\big(f^0(h^1f^1)\ldots (h^nf^n)\big)
\end{equation*}
for $h^1 \otimes \ldots \otimes h^n \in H^{\otimes n}$ and $f^0 \otimes \ldots \otimes f^{n} \in Hom_{\mathcal{D}_H}(X_1,X_0) \otimes Hom_{\mathcal{D}_H}(X_2,X_1) \otimes \ldots \otimes Hom_{\mathcal{D}_H}(X_0,X_n)$. This induces a homomorphism in cyclic cohomology
\begin{equation*}
\gamma^\bullet:HC^\bullet_{(\delta,\sigma)}(H) \longrightarrow HC^\bullet(\mathcal{D}_H)
\end{equation*}
\end{theorem}
\begin{proof}
We need to show that the following identities hold:
\begin{align}
\gamma^{n}\delta_i&=\delta_i\gamma^{n-1} \quad 0 \leq i \leq n \label{1}\\ 
\gamma^n \sigma_i &=\sigma_i\gamma^{n+1} \quad 0 \leq i \leq n \label{2}\\ 
\gamma^n\tau_n &= \tau_n\gamma^n \label{3}
\end{align}
We first verify \eqref{1}. The case $i=0$ is straightforward. For $i=n$, we have
\begin{align*}
\big(\gamma^{n}(\delta_n(h^1 \otimes \ldots \otimes h^{n-1}))\big)(f^0 \otimes \ldots \otimes f^{n})&=\big(\gamma^{n}(h^1 \otimes \ldots \otimes h^{n-1} \otimes \sigma)\big)(f^0 \otimes \ldots \otimes f^{n})\\
&=T_{X_0}^H\big(f^0(h^1f^1)\ldots (h^{n-1}f^{n-1})(\sigma f^n)\big)\\
&=T_{X_n}^H\big(f^nf^0(h^1f^1)\ldots (h^{n-1}f^{n-1})\big) \quad ~~~~(\text{by}~ \eqref{tracecomm})\\
&= \big(\gamma^{n-1}(h^1 \otimes \ldots \otimes h^{n-1})\big)\big((f^nf^0)\otimes f^1\otimes \ldots \otimes f^{n-1}\big)\\
&=\big(\delta_n (\gamma^{n-1}(h^1 \otimes \ldots \otimes h^{n-1}))\big)(f^0 \otimes \ldots \otimes f^{n})
\end{align*}
For $1 \leq i \leq n-1$, we have
\begin{align*}
\big(\gamma^{n}(\delta_i(h^1 \otimes \ldots \otimes h^{n-1}))\big)(f^0 \otimes \ldots \otimes f^{n})&= \big(\gamma^{n}(h^1 \otimes \ldots \otimes h^i_1 \otimes h^i_2 \otimes \ldots \otimes h^{n-1})\big)(f^0 \otimes \ldots \otimes f^{n})\\
&= T_{X_0}^H\big(f^0(h^1f^1)\ldots (h^i_1f^i)(h^i_2f^{i+1}) \ldots (h^{n-1}f^{n})\big)\\
&=T^H_{X_0}(f^0(h^1f^1)\ldots (h^i(f^if^{i+1})) \ldots (h^{n-1}f^{n}))\\
&=\big(\delta_i (\gamma^{n-1}(h^1 \otimes \ldots \otimes h^{n-1}))\big)(f^0 \otimes \ldots \otimes f^{n})
\end{align*} 
Next we verify \eqref{2}. For $0 \leq i \leq n-1$, we have
\begin{align*}
\big(\gamma^{n}(\sigma_i(h^1 \otimes \ldots \otimes h^{n+1}))\big)(f^0 \otimes \ldots \otimes f^{n})&=\big(\gamma^{n}(h^1 \otimes \ldots h^i \otimes \varepsilon(h^{i+1}) \otimes h^{i+2} \otimes \ldots \otimes h^{n+1})\big)(f^0 \otimes \ldots \otimes f^{n})\\
&=T_{X_0}^H\big(f^0(h^1f^1)\ldots(\varepsilon(h^{i+1})id_{X_{i+1}}) \ldots (h^{n+1}f^{n})\big)\\
&=T_{X_0}^H\big(f^0(h^1f^1)\ldots(h^{i+1}id_{X_{i+1}}) \ldots (h^{n+1}f^{n})\big)\\
&=\big(\sigma_i (\gamma^{n+1}(h^1 \otimes \ldots \otimes h^{n+1}))\big)(f^0 \otimes \ldots \otimes f^{n})
\end{align*}
It now remains to verify \eqref{3}. For that, we have
\begin{align*}
\big(\gamma^n (\tau_n(h^1 \otimes \ldots \otimes h^n))\big)(f^0 \otimes \ldots \otimes f^{n})&=\Big(\gamma^n \big(\delta(h^1_1)S(h^1_2)(h^2 \otimes \ldots \otimes h^n\otimes \sigma)\big)\Big)(f^0 \otimes \ldots \otimes f^{n})\\
&=\Big(\gamma^n\big(\delta(h^1_1)\left(S(h^1_{n+1})h^2 \otimes \ldots \otimes S(h^1_3)h^n \otimes S(h^1_2)\sigma\right)\big)\Big)(f^0 \otimes \ldots \otimes f^{n})\\
&=T_{X_0}^H\Big(\delta(h^1_1)f^0 \big(S(h^1_{n+1})h^2f^1\big)\ldots \big(S(h^1_3)h^nf^{n-1}\big)\big(S(h^1_2)\sigma f^{n}\big)\Big)\\
&=T_{X_0}^H\Big(\delta(h^1_1)f^0 \big[S(h^1_2)\big((h^2f^1)\ldots (h^nf^{n-1})(\sigma f^{n})\big)\big]\Big)\\
&=T_{X_0}^H\Big(f^0 \big[S_\delta(h^1)\big((h^2f^1)\ldots (h^nf^{n-1})(\sigma f^{n})\big)\big]\Big)\\
&=T_{X_0}^H\big( (h^1f^0)(h^2f^1)\ldots (h^nf^{n-1})(\sigma f^{n})\big) \quad~~~~~~~~~ (\text{by}~ \eqref{invariance})\\
&=T_{X_n}^H\big(f^n (h^1f^0)(h^2f^1)\ldots (h^nf^{n-1})\big) \quad~~~~~~~~~~~~~ (\text{by}~ \eqref{tracecomm})\\
&=\big(\tau^n (\gamma_n(h^1 \otimes \ldots \otimes h^n))\big)(f^0 \otimes \ldots \otimes f^{n})
\end{align*}
This completes the proof.
\end{proof}

\section{Invariant traces as Hopf-cyclic cocycles}\label{Hopfcyclic}
We continue with $H$ being a Hopf algebra equipped with a modular pair $(\delta,\sigma)$ in involution. The key to the construction of the characteristic map 
$\gamma^\bullet:HC^\bullet_{(\delta,\sigma)}(H) \longrightarrow HC^\bullet(\mathcal{D}_H)$ in Theorem \ref{charmap} is the $\delta$-invariant
$\sigma$-trace $T^H$ on the left $H$-category $\mathcal D_H$.  In this section, we will describe such traces as Hopf-cyclic cocycles for $\mathcal D_H$
taking values in a certain SAYD  module.

\smallskip
Since $\mathcal D_H$ is a left $H$-category, it is clear from the definition in \eqref{cnerf} that each $\{CN_n(\mathcal{D}_H)\}_{n\geq 0}$ is a left $H$-module via the diagonal action of $H$.

\begin{lemma} 
Let $M$ be a right $H$-module. For each $n\geq 0$, $M \otimes CN_n(\mathcal{D}_H)$ is a right $H$-module with action determined by
$$(m \otimes f^0 \otimes \ldots \otimes f^{n})h:= mh_1 \otimes S(h_2)(f^0 \otimes \ldots \otimes f^{n})$$
for any $m \in M$, $f^0 \otimes \ldots \otimes f^{n} \in CN_n(\mathcal D_H)$ and $h \in H$.
\end{lemma}
\begin{proof}
This follows from the fact that the antipode $S$ is an anti-algebra homomorphism and $\Delta(1_H)=1_H \otimes 1_H$.
\end{proof}

We now recall the notion of a SAYD module from \cite[Definition 2.1]{hkrs2}.
\begin{definition}
Let $H$ be a Hopf algebra with a bijective antipode $S$. A $k$-vector space $M$ is said to be a right-left anti-Yetter-Drinfeld module over $H$ if $M$ is a right $H$-module and a left $H$-comodule such that
\begin{equation}\label{SAYDcondi}
\rho(mh)=(mh)_{(-1)} \otimes (mh)_{(0)}= S(h_3)m_{(-1)}h_1 \otimes m_{(0)}h_2
\end{equation}
for all $m \in M$ and $h \in H$, where $\rho: M \longrightarrow H \otimes M,~ m \mapsto m_{(-1)} \otimes m_{(0)}$ is the coaction. Moreover, $M$ is said to be stable if $m_{(0)}m_{(-1)}=m$.
\end{definition}

We now take the Hopf-cyclic cohomology $HC^\bullet_H(\mathcal{D}_H,M)$ of an $H$-category $\mathcal{D}_H$ with coefficients in a stable anti-Yetter-Drinfeld (SAYD) module $M$ (see also \cite{kk}). This generalizes the construction of the Hopf-cyclic cohomology for $H$-module algebras with coefficients in a SAYD module (see \cite{hkrs}).  For each $n \geq 0$, we set 
\begin{equation*}
\begin{array}{ll}
C^n(\mathcal{D}_H,M):=Hom_k(M \otimes CN_n(\mathcal{D}_H),k)\\
C^n_H(\mathcal{D}_H,M):=Hom_H(M \otimes CN_n(\mathcal{D}_H),k)
\end{array}
\end{equation*}
where $k$ is considered as a right $H$-module via the counit. It is clear from the definition that an element in $C^n_H(\mathcal{D}_H,M)$ is a $k$-linear map $\phi:M \otimes CN_n(\mathcal{D}_H) \longrightarrow k$ satisfying
\begin{equation}\label{new3.2}
\phi\left(mh_1 \otimes S(h_2)(f^0 \otimes \ldots \otimes f^{n})\right)=\varepsilon(h)\phi(m \otimes f^0 \otimes \ldots \otimes f^{n})
\end{equation}

\begin{lemma}\label{lem2.3f}
Let $M$ be a right-left SAYD module over $H$ and let $\phi \in C^n_H(\mathcal{D}_H,M)$. Then,
\begin{equation*}
\phi\left(m \otimes h(f^0 \otimes f^1 \otimes \ldots \otimes f^n)\right)=\phi(mh \otimes f^0 \otimes f^1 \otimes \ldots \otimes f^n)
\end{equation*}
for any $m \in M$ and $f^0 \otimes f^1 \otimes \ldots \otimes f^n \in  CN_n(\mathcal D_H)$.
\end{lemma}
\begin{proof}
Using the stability of $M$, we have
\begin{equation*}
\begin{array}{lll}
\phi(mh \otimes f^0 \otimes \ldots \otimes f^n)&= \phi\left((mh)_{(0)} (mh)_{(-1)} \otimes f^0 \otimes \ldots \otimes f^n\right)\\
&= \phi\left((m_{(0)}h_2)((S(h_3)m_{(-1)}h_1) \otimes f^0 \otimes \ldots \otimes f^n\right)& (\text{by}~ \eqref{SAYDcondi})\\
&= \phi\left((m_{(0)}m_{(-1)}h_1 \otimes \varepsilon(h_2)(f^0 \otimes \ldots \otimes f^n)\right)\\
&= \phi\left(mh_1 \otimes S(h_2)h_3(f^0 \otimes \ldots \otimes f^n)\right)\\
&= \varepsilon(h_1)\phi\left(m \otimes h_2(f^0 \otimes \ldots \otimes f^n)\right)& (\text{by}~ \eqref{new3.2})\\
&= \phi\left(m \otimes h(f^0 \otimes \ldots \otimes f^n)\right)
\end{array}
\end{equation*}
\end{proof}

Using the stability of $M$, we also observe that 
\begin{equation}\label{2obsv}m_{(0)}S^{-1}(m_{(-1)})=m_{(0)(0)}m_{(0)(-1)}S^{-1}(m_{(-1)})=m_{(0)}m_{(-1)2}S^{-1}(m_{(-1)1})=m_{(0)}\varepsilon(m_{(-1)})=m
\end
{equation}

We now recall that a (co)simplicial module is said to be para-(co)cyclic if all the relations for a (co)cyclic module are satisfied except $\tau_n^{n+1}=id$ (see, for instance \cite{kk}).

\begin{prop}\label{prop2.3} Let $\mathcal{D}_H$ be a left $H$-category and let $M$ be a right-left SAYD module over $H$. Then, 

\smallskip
(1) we have a para-cocyclic module $C^\bullet(\mathcal{D}_H,M):=\{C^n(\mathcal{D}_H,M)\}_{n \geq 0}$ with the following structure maps  
\begin{align*}
(\delta_i\phi)(m \otimes f^0 \otimes \ldots \otimes f^{n})&= \begin{cases}
\phi(m \otimes f^0 \otimes \ldots \otimes f^if^{i+1} \otimes \ldots \otimes f^{n}) \quad~~~~~~~~~~~~~~~~ 0 \leq i \leq n-1\\
\phi\big(m_{(0)} \otimes \big(S^{-1}(m_{(-1)})f^n\big)f^0 \otimes \ldots \otimes f^{n-1}\big) \quad~~~~~~~~~~~~i=n\\
\end{cases}\\
\vspace{.2cm}
(\sigma_i\psi)(m \otimes f^0 \otimes \ldots \otimes f^{n})&= \begin{cases} \psi(m \otimes f^0 \otimes \ldots \otimes f^i \otimes id_{X_{i+1}}\otimes f^{i+1} \otimes \ldots \otimes f^{n}) \quad~~~~~ 0 \leq i \leq n-1\\
\psi(m \otimes f^0 \otimes \ldots \otimes f^{n} \otimes id_{X_{0}}) \quad~~~~~~~~~~~~~~~~~~~~~~~~~~ i=n\\
\end{cases}\\ 
\vspace{.2cm}
(\tau_n\varphi)(m \otimes f^0 \otimes \ldots \otimes f^{n})&=\varphi\big(m_{(0)} \otimes S^{-1}(m_{(-1)})f^n \otimes f^0 \otimes \ldots \otimes f^{n-1}\big)
\end{align*}
for any $\phi \in C^{n-1}(\mathcal{D}_H,M)$, $\psi \in C^{n+1}(\mathcal{D}_H,M)$, $\varphi \in C^{n}(\mathcal{D}_H,M)$, $m \in M$ and $ f^0 \otimes \ldots \otimes f^n \in Hom_{\mathcal{D}_H}(X_1,X_0) \otimes Hom_{\mathcal{D}_H}(X_2,X_1) \otimes \ldots \otimes Hom_{\mathcal{D}_H}(X_0,X_n)$.

\smallskip
(2) by restricting to right $H$-linear morphisms $C^n_H(\mathcal{D}_H,M)=Hom_H(M \otimes CN_n(\mathcal{D}_H),k),$ we obtain a cocyclic module $C^\bullet_H(\mathcal{D}_H,M):=\{C^n_H(\mathcal{D}_H,M)\}_{n \geq 0}$. 
\end{prop}
\begin{proof}
(1) It is easy to verify that $\delta_i$ and $\sigma_i$ for $0 \leq i \leq n$ define a cosimplicial structure on $C^\bullet(\mathcal{D}_H,M)$. Therefore, it remains to check that the following identities hold:
\begin{equation*}
\begin{array}{ll}
\tau_n\delta_i&=\delta_{i-1}\tau_{n-1} \qquad 1 \leq i \leq n\\
\tau_n\delta_0&=\delta_n\\
\tau_n\sigma_i&=\sigma_{i-1}\tau_{n+1} \qquad 1 \leq i \leq n\\
\tau_n\sigma_0&=\sigma_n\tau^2_{n+1}
\end{array}
\end{equation*}
For $1 \leq i < n$, we have
\begin{equation*}
\begin{array}{ll}
\left(\tau_n\delta_i(\phi)\right)(m \otimes f^0 \otimes \ldots \otimes f^{n})&=\delta_i(\phi)\left(m_{(0)} \otimes S^{-1}(m_{(-1)})f^n \otimes f^0 \otimes \ldots \otimes f^{n-1}\right)\\
&= \phi\left(m_{(0)} \otimes S^{-1}(m_{(-1)})f^n \otimes f^0 \otimes \ldots \otimes f^{i-1}f^{i} \otimes \ldots \otimes f^{n-1}\right)\\
&=\left(\delta_{i-1}\tau_{n-1}(\phi)\right)(m \otimes f^0 \otimes \ldots \otimes f^{n})
\end{array}
\end{equation*}
and
\begin{align*}
\left(\tau_n\delta_n(\phi)\right)(m \otimes f^0 \otimes \ldots \otimes f^{n}))&=\delta_n(\phi)\left(m_{(0)} \otimes S^{-1}(m_{(-1)})f^n \otimes f^0 \otimes \ldots \otimes f^{n-1}\right)\\
&=\phi\big(m_{(0)(0)} \otimes \big(S^{-1}(m_{(0)(-1)})f^{n-1}\big)\big(S^{-1}(m_{(-1)})f^n\big) \otimes f^0 \otimes \ldots \otimes f^{n-2}\big)\\
&=\phi\big(m_{(0)} \otimes \big(S^{-1}(m_{(-1)2})f^{n-1}\big)\big(S^{-1}(m_{(-1)1})f^n\big) \otimes f^0 \otimes \ldots \otimes f^{n-2}\big)\\
&=\phi\big(m_{(0)} \otimes S^{-1}(m_{(-1)})(f^{n-1}f^n) \otimes f^0 \otimes \ldots \otimes f^{n-2}\big)\\
&=\left(\delta_{n-1}\tau_{n-1}(\phi)\right)(m \otimes f^0 \otimes \ldots \otimes f^{n})
\end{align*}
It follows immediately by definition that $\tau_n\delta_0=\delta_n$. Next, we verify the identities involving the degeneracies. For $i=n$, we have
\begin{equation*}
\begin{array}{ll}
\left(\sigma_{n-1}\tau_{n+1}(\psi)\right)(m \otimes f^0 \otimes \ldots \otimes f^{n})&=\tau_{n+1}(\psi)(m \otimes f^0 \otimes \ldots \otimes f^{n-1} \otimes id_{X_n} \otimes f^{n})\\
&=\psi\big(m_{(0)} \otimes S^{-1}(m_{(-1)})f^n \otimes f^0 \otimes \ldots \otimes f^{n-1} \otimes id_{X_n} \big)\\
&=\left(\tau_n\sigma_n(\psi)\right)(m \otimes f^0 \otimes \ldots \otimes f^{n})
\end{array}
\end{equation*}
The case $1 \leq i <n$ is easy to verify. Further, we have
\begin{equation*}
\begin{array}{ll}
\left(\sigma_n\tau^2_{n+1}(\varphi)\right)(m \otimes f^0 \otimes \ldots \otimes f^{n})&=\tau^2_{n+1}(\varphi)(m \otimes f^0 \otimes \ldots \otimes f^{n} \otimes id_{X_0})\\
&= \tau_{n+1}(\varphi)\big(m_{(0)} \otimes S^{-1}(m_{(-1)})id_{X_0} \otimes f^0 \otimes \ldots \otimes f^{n}\big)\\
&= \tau_{n+1}(\varphi)\big(m_{(0)} \otimes \varepsilon\big(S^{-1}(m_{(-1)})\big)id_{X_0} \otimes f^0 \otimes \ldots \otimes f^{n}\big)\\
&= \tau_{n+1}(\varphi)\big(m_{(0)} \otimes \varepsilon(m_{(-1)})id_{X_0} \otimes f^0 \otimes \ldots \otimes f^{n}\big)\\
&= \tau_{n+1}(\varphi)(m \otimes id_{X_0} \otimes f^0 \otimes \ldots \otimes f^{n})\\
&=\varphi\big(m_{(0)} \otimes S^{-1}(m_{(-1)})f^n \otimes id_{X_0} \otimes f^0 \otimes \ldots \otimes f^{n-1}\big)\\
&= \left(\tau_n\sigma_0 (\varphi)\right)(m \otimes f^0 \otimes \ldots \otimes f^{n})
\end{array}
\end{equation*}

\smallskip
(2) Using (1), it now remains to prove that the structure maps are well-defined and $\tau_n^{n+1}=id$. Let us first verify that the cyclic operator $\tau_n$ is well-defined, i.e., $\tau_n(\phi)$ is $H$-linear for each $\phi \in C^n_H(\mathcal{D}_H,M)$.
\begin{align*}
&\left(\tau_n(\phi)\right)\left(mh_1 \otimes S(h_2)(f^0 \otimes \ldots \otimes f^{n})\right)\\
&=\left(\tau_n(\phi)\right)\left(mh_1 \otimes S(h_{n+2})f^0 \otimes \ldots \otimes S(h_2)f^n\right)\\
&=\phi\left((mh_1)_{(0)} \otimes S^{-1}\left((mh_1)_{(-1)}\right)S(h_2)f^n \otimes S(h_{n+2})f^0 \otimes \ldots \otimes S(h_3)f^{n-1} \right)\\
&=\phi\left(m_{(0)}h_2 \otimes S^{-1}\left(S(h_3)m_{(-1)}h_1\right)S(h_4)f^n \otimes S(h_{n+4})f^0 \otimes \ldots \otimes S(h_5)f^{n-1} \right)\\
&=\phi\left(m_{(0)}h_2 \otimes S^{-1}(h_1)S^{-1}(m_{(-1)})h_3 S(h_4)f^n \otimes S(h_{n+4})f^0 \otimes \ldots \otimes S(h_5)f^{n-1} \right)\\
&=\phi\left(m_{(0)}h_2 \otimes S^{-1}(h_1)S^{-1}(m_{(-1)})f^n \otimes S(h_{n+2})f^0 \otimes \ldots \otimes S(h_3)f^{n-1} \right)\\
&=\phi\left(m_{(0)} \otimes h_2\big[S^{-1}(h_1)S^{-1}(m_{(-1)})f^n \otimes S(h_{n+2})f^0 \otimes \ldots \otimes S(h_3)f^{n-1} \big]\right) \quad (\text{by Lemma } \ref{lem2.3f})\\
&=\phi\left(m_{(0)} \otimes h_2S^{-1}(h_1)S^{-1}(m_{(-1)})f^n \otimes h_3S(h_{2n+2})f^0 \otimes \ldots \otimes h_{n+2}S(h_{n+3})f^{n-1}\right)\\
&=\varepsilon(h)\phi\left(m_{(0)} \otimes S^{-1}(m_{(-1)})f^n \otimes f^0 \otimes \ldots \otimes f^{n-1}\right)=\varepsilon(h)\left(\tau_n(\phi)\right)(m \otimes f^0 \otimes \ldots \otimes f^{n})
\end{align*}

Similarly, it may be verified that the degeneracies are also well-defined. Next, we verify that the face maps are well-defined. For $0 \leq i <n$, we have
\begin{equation*}
\begin{array}{ll}
&\left(\delta_i(\phi)\right)\left(mh_1 \otimes S(h_2)(f^0 \otimes \ldots \otimes f^{n})\right)\\
&\quad =\left(\delta_i(\phi)\right)\left(mh_1 \otimes S(h_{n+2})f^0 \otimes \ldots \otimes S(h_{n+2-i})f^i \otimes S(h_{n+2-(i+1)})f^{i+1} \otimes \ldots \otimes S(h_2)f^{n} \right)\\
&\quad =\phi\left(mh_1 \otimes S(h_{n+1})f^0 \otimes \ldots \otimes S(h_{n+2-(i+1)})(f^if^{i+1}) \otimes \ldots \otimes S(h_2)f^{n} \right)\\
&\quad =\phi\left(mh_1 \otimes S(h_2)(f^0 \otimes \ldots \otimes f^if^{i+1} \otimes \ldots \otimes f^n)\right)\\
& \quad = \varepsilon(h)\phi(m \otimes f^0 \otimes \ldots \otimes f^if^{i+1} \otimes \ldots \otimes f^n)=\varepsilon(h)\left(\delta_i(\phi)\right)\left(m\otimes f^0\otimes ...\otimes f^n\right)
\end{array}
\end{equation*}
Since $\delta_n=\tau_n\delta_0$, the preceeding computations show that $\delta_n$ is also well-defined. 

\smallskip
Further, using the stability of $M$, we have
\begin{equation*}
\begin{array}{ll}
\tau_n^{n+1}(\phi)(m \otimes f^0 \otimes \ldots \otimes f^{n})&=\phi(m_{(0)} \otimes S^{-1}((m_{(-1)})_{n+1})f^0 \otimes \ldots \otimes S^{-1}(m_{(-1)1})f^{n})\\
&= \phi\big(m_{(0)} \otimes S^{-1}(m_{(-1)})(f^0 \otimes \ldots \otimes f^{n})\big)\\
&= \phi\big(m_{(0)} S^{-1}(m_{(-1)})\otimes f^0 \otimes \ldots \otimes f^{n}\big)\quad (\text{by Lemma } \ref{lem2.3f})\\
&= \phi(m \otimes f^0 \otimes \ldots \otimes f^{n})\quad (\text{by } \eqref{2obsv})
\end{array}
\end{equation*}
This completes the proof.
\end{proof}

The cohomology of the cocyclic module $C^\bullet_H(\mathcal{D}_H,M)$ is referred to as the Hopf-cyclic cohomology of the $H$-category $\mathcal{D}_H$ with coefficients in the SAYD module $M$. The corresponding cohomology groups are denoted by $HC_H^\bullet(\mathcal{D}_H,M)$.

\begin{remark}
(1) As $k$ contains $\mathbb{Q}$, we recall that the cohomology of a cocyclic module $\mathscr{C}$ can be expresed alternatively as the cohomology of the following complex (see, for instance \cite[2.5.9]{Loday}):
\[
\begin{CD}
C^0_\lambda(\mathscr{C}) @> b>> \dots @ > b>> C^n_\lambda(\mathscr{C}) @ > b>> C^{n+1}_\lambda(\mathscr{C})@> b>>\dots
\end{CD}
\]
where $C^n_\lambda(\mathscr{C})=Ker(1-\lambda)\subseteq C^n(\mathscr{C})$, $b=\sum_{i=0}^{n+1} (-1)^i \delta_i$ and 
$\lambda=(-1)^n \tau_n$. In particular, an element $\phi \in C^n_H(\mathcal{D}_H,M)$ is a cyclic cocycle if and only if
\begin{equation}\label{3.3y}
b(\phi)=0 \quad  \text{and} \quad (1-\lambda)(\phi)=0
\end{equation}

\smallskip
(2) The field $k$ is trivially a Hopf algebra with $\Delta(1)=1 \otimes 1$, $S(1)=1=\varepsilon(1)$, and also a SAYD module over itself with coaction given by $\Delta$. Substituting $H=k=M$ in the construction of $C^\bullet_H(\mathcal{D}_H,M)$, we get back the ordinary cyclic cohomology $HC^\bullet(\mathcal{D}_H)$ of the $k$-linear category $\mathcal{D}_H$ as discussed in Section \ref{nerve}.
\end{remark}
We now let ${^\sigma}{k}_\delta$ denote the SAYD module structure on $k$ (see, for instance \cite{hkrs2}) over $H$ defined by setting
$$\alpha \cdot h:=\delta(h)\alpha, \quad \rho(\alpha):=\sigma \otimes \alpha \qquad \forall \alpha \in k, h \in H$$ 

\begin{prop}\label{bijection}
Let $\mathcal{D}_H$ be a left $H$-category and let $(\delta,\sigma)$ be a modular pair in involution for $H$. Then, $Z_H^0(\mathcal{D}_H,{^\sigma}{k}_\delta)=HC_H^0(\mathcal{D}_H,{^\sigma}{k}_\delta)$ is in bijection with the $k$-space consisting of $\delta$-invariant $\sigma$-traces on $\mathcal{D}_H$. 
\end{prop}
\begin{proof}
By definition, $HC_H^0(\mathcal{D}_H,{^\sigma}{k}_\delta)=Ker(b_0)$, where $b_0=\delta_0-\delta_1:C^0_H(\mathcal{D}_H,{^\sigma}{k}_\delta) \longrightarrow C^1_H(\mathcal{D}_H,{^\sigma}{k}_\delta)$. Let $\phi \in Ker(b_0)$. We now define a collection $T^H:=\{T_X^H:Hom_{\mathcal{D}_H}(X,X) \longrightarrow k\}_{X \in Ob(\mathcal{D}_H)}$ of $k$-linear maps given by
\begin{equation*}
T_X^H(f)=\phi(1 \otimes f) \qquad \forall f \in Hom_{\mathcal{D}_H}(X,X)
\end{equation*}
Using the fact that $\sigma$ is a group-like element in $H$ and $\delta(\sigma)=1$, we have
\begin{equation*}
\begin{array}{ll}
T^H_X(gf)=\phi(1 \otimes gf)=(\delta_0(\phi))(1 \otimes g \otimes f)&=(\delta_1(\phi))(1 \otimes g \otimes f)=\phi\left(1 \otimes (S^{-1}(\sigma)f)g\right)= \phi\big(1 \otimes (S(\sigma)f)(\varepsilon(\sigma)g)\big)\\
&= \phi\big(\delta(\sigma) \otimes (S(\sigma)f)(S(\sigma)\sigma g)\big)= \phi\big(1 \cdot \sigma \otimes S(\sigma)(f(\sigma g))\big)\\
&= \big(\phi(1 \otimes f(\sigma g)\big)\varepsilon(\sigma)=\phi(1 \otimes f(\sigma g))=T_Y^H(f(\sigma g))
\end{array}
\end{equation*}
for any $f \in Hom_{\mathcal{D}_H}(X,Y)$ and $g \in Hom_{\mathcal{D}_H}(Y,X)$. This shows that the collection $T^H$ is a $\sigma$-trace on the category $\mathcal{D}_H$. We now verify that $T^H$ is $\delta$-invariant. For any $f' \in Hom_{\mathcal{D}_H}(X,X)$, we have
\begin{align*}
\delta(h)T_X^H(f') = \phi(1 \otimes \delta(h)f')= \phi\left(\delta(h_1) \otimes \varepsilon(h_2)f'\right)=\phi\left(1 \cdot h_1 \otimes S(h_2)h_3f'\right)=  \varepsilon(h_1)\phi\left(1\otimes h_2f'\right)=T_X^H(hf')
\end{align*}

Conversely, suppose that $T^H$ is a $\delta$-invariant $\sigma$-trace on $\mathcal{D}_H$. We consider the $k$-linear map $\phi \in C^0(\mathcal{D}_H,{^\sigma}{k}_\delta)$   determined by $\phi(1 \otimes f')=T_X^H(f')$ for each $f' \in Hom_{\mathcal{D}_H}(X,X)$. Let us first verify that $\phi$ is $H$-linear. For any $h \in H$, we have
\begin{align*}
\phi\left((1 \otimes f')h\right)&= \phi\left(1 \cdot h_1 \otimes S(h_2)f'\right)=\phi\left(1  \otimes \delta(h_1)S(h_2)f'\right)=T_X^H\left(S_\delta(h)f'\right)\\
&=\delta\left(S_\delta(h)\right)T_X^H(f')=\varepsilon(h)T_X^H(f')=\varepsilon(h)\phi(1 \otimes f')
\end{align*}
Next, we verify that $\phi \in Ker(b_0)$. For any $f \in Hom_{\mathcal{D}_H}(X,Y)$ and $g \in Hom_{\mathcal{D}_H}(Y,X)$, we have
\begin{equation*}
\begin{array}{ll}
\left(b_0(\phi)\right)(1 \otimes g \otimes f)&= \left(\delta_0(\phi)\right)(1 \otimes g \otimes f)- \left(\delta_1(\phi)\right)(1 \otimes g \otimes f)\\
&=\phi(1 \otimes gf) - \phi\left(1 \otimes \left(S^{-1}(\sigma)f\right)g\right)\\
&=\phi(1 \otimes gf) - \phi\left(1 \cdot \sigma \otimes \left(S(\sigma)f\right)\left(S(\sigma)\sigma g\right)\right)\\
&=\phi(1 \otimes gf) - \phi\left(1 \otimes f(\sigma g)\right)= T_X^H(gf)-T_Y^H\left(f(\sigma g)\right)=0
\end{array}
\end{equation*}
This shows that $\phi \in HC_H^0(\mathcal{D}_H,{^\sigma}{k}_\delta)$.
\end{proof}

\section{Characteristic map with SAYD coefficients}
Let $\mathcal{D}_H$ be a left $H$-category and let $(\delta,\sigma)$ be a modular pair in involution for $H$. We have shown that the $\delta$-invariant $\sigma$-traces on $\mathcal{D}_H$ are in bijection with $HC_H^0(\mathcal{D}_H,{^\sigma}{k}_\delta)$ (Proposition \ref{bijection}). Moreover, a $\delta$-invariant $\sigma$-trace on $\mathcal{D}_H$ induces a homomorphism $HC^\bullet_{(\delta,\sigma)}(H) \longrightarrow HC^\bullet(\mathcal{D}_H)$ (Theorem \ref{charmap}). Thus, we obtain a pairing
\begin{equation}\label{p1}
HC^\bullet_{(\delta,\sigma)}(H) \otimes HC_H^0(\mathcal{D}_H,{^\sigma}{k}_\delta) \longrightarrow HC^\bullet(\mathcal{D}_H)
\end{equation}

The pairing in \eqref{p1} leads us to ask if there exists a similar pairing when the Hopf algebra $H$ is replaced by an $H$-module coalgebra $C$ and the trivial SAYD module ${^\sigma}{k}_\delta$ is replaced by a general SAYD module $M$.  In this section, we obtain the following pairing
\begin{equation*}
HC^\bullet_H(C,M)  \otimes HC^0_H(\mathcal{D}_H,M) \longrightarrow HC^\bullet(\mathcal{D}_H)\\
\end{equation*} 
A coalgebra $(C,\Delta_C,\varepsilon_C)$ which is also a left $H$-module such that
$$\Delta_C(hc)= h_1c_1 \otimes h_2c_2, \quad \varepsilon_C(hc)=\varepsilon(h)\varepsilon_C(c) \qquad \forall h \in H, c\in C$$
is said to be a left $H$-module coalgebra. 

\smallskip We now recall the Hopf-cyclic cohomology of a left $H$-module coalgebra $C$ with coefficients in a right-left SAYD module $M$ (see \cite{hkrs}). Let $C^n_H(C,M):=M \otimes_H C^{\otimes n+1}$ for $n \geq 0$. Then, $C^\bullet_H(C,M)=\{C^n_H(C,M)\}_{n \geq 0}$ is a $\Lambda$-module with the following structure maps:
\begin{align*}
\delta'_i(m \otimes_H c^0 \otimes \ldots \otimes c^{n-1})&= 
\begin{cases}  m \otimes_H c^0 \otimes \ldots \otimes c^i_1 \otimes c^{i}_2 \otimes \ldots \otimes c^{n-1} \quad~~~~~ 0 \leq i \leq n-1\\
 m_{(0)} \otimes_H c^0_2 \otimes c^1 \otimes \ldots \otimes c^{n-1} \otimes m_{(-1)}c^0_1 \quad~~~~~i=n
\end{cases}\\
\vspace{.2cm}
\sigma'_i(m \otimes_H c^0 \otimes \ldots \otimes c^{n+1})&= m \otimes_H c^0 \otimes \ldots \otimes \varepsilon_C(c^{i+1}) \otimes  \ldots \otimes c^{n+1} \quad~~~~~~~ 0 \leq i \leq n\\
\vspace{.2cm}
\tau'_n(m \otimes_H c^0 \otimes \ldots \otimes c^n)&=m_{(0)} \otimes_H c^1 \otimes \ldots \otimes c^n \otimes m_{(-1)}c^0
\end{align*}
The cohomology of the cocyclic module $C^\bullet_H(C,M)$ is said to be the Hopf-cyclic cohomology of the $H$-module coalgebra $C$ with coefficients in the SAYD module $M$. The corresponding cohomology groups will be denoted by $HC_H^\bullet(C,M)$.

\smallskip
In the construction of the pairing \eqref{p1}, the Hopf algebra $H$ acts on the category $\mathcal{D}_H$ in the sense of Definition \ref{defH-cat}. We will now define the action of an $H$-module coalgebra $C$  on a left $H$-category $\mathcal{D}_H$.

\begin{definition}\label{Def3.1s}
Let $\mathcal{D}_H$ be a left $H$-category and $C$ be a left $H$-module coalgebra. We say that $C$ acts on $\mathcal{D}_H$ if we have $k$-linear maps 
$\{C \otimes Hom_{\mathcal{D}_H}(X,Y) \longrightarrow Hom_{\mathcal{D}_H}(X,Y)\}_{(X,Y)\in Ob(\mathcal D_H)^2}$  satisfying
\begin{align}\label{Caction}
c(gf)=(c_1g)(c_2f), \quad c(id_X)=\varepsilon_C(c)id_X, \quad h(cf)=(hc)f
\end{align}
for any $f \in Hom_{\mathcal{D}_H}(X,Y),~ g \in Hom_{\mathcal{D}_H}(Y,Z),~ c \in C$ and $h \in H$.
\end{definition}

\smallskip
We now show that there is a pairing between the Hopf-cyclic cohomology of an $H$-module coalgebra $C$ and $HC^0_H(\mathcal D_H,M)$. This pairing takes values in the usual cyclic cohomology of the $k$-linear category  $\mathcal D_H$ (as described in Section \ref{nerve}).

\begin{theorem}\label{p2}
Let $\mathcal{D}_H$ be a left $H$-category and let $C$ be a left $H$-module coalgebra such that $C$ acts on $\mathcal{D}_H$.  Let $M$ be a right-left SAYD module over $H$. Then, for each $\phi \in HC^0_H(\mathcal{D}_H,M)$, we obtain a morphism $\gamma_M^\bullet: C^\bullet_H(C,M) \longrightarrow CN^\bullet(\mathcal{D}_H)$ of $\Lambda$-modules defined by
\begin{equation}
\big(\gamma_M^n(m \otimes _H c^0 \otimes \ldots \otimes c^n)\big)(f^0 \otimes \ldots \otimes f^{n}):=\phi\big(m \otimes (c^0f^0)\ldots (c^nf^n)\big) \label{gencharmap}
\end{equation}
for any $m \in M$, $c^0 \otimes \ldots \otimes c^n \in C^{\otimes n+1}$ and $f^0 \otimes \ldots \otimes f^{n} \in Hom_{\mathcal{D}_H}(X_1,X_0) \otimes Hom_{\mathcal{D}_H}(X_2,X_1) \otimes \ldots \otimes Hom_{\mathcal{D}_H}(X_0,X_n)$. Thus, we get the following pairing
\begin{equation*}
HC^\bullet_H(C,M) \otimes HC^0_H(\mathcal{D}_H,M)  \longrightarrow HC^\bullet(\mathcal{D}_H)
\end{equation*}
\end{theorem}
\begin{proof}
Let $\phi \in HC^0_H(\mathcal{D}_H,M)$. Then, by definition, we have
\begin{align}\label{3.3g}
\phi(m \otimes gf)=(\delta_0(\phi))(m \otimes g \otimes f)&=(\delta_1(\phi))(m \otimes g \otimes f)=\phi\left(m_{(0)} \otimes \left(S^{-1}(m_{(-1)})f\right)g\right)
\end{align}
for any $m \in M,$ $f \in Hom_{\mathcal{D}_H}(X,Y)$ and $g \in Hom_{\mathcal{D}_H}(Y,X)$.
In order to show that \eqref{gencharmap} defines a map of $\Lambda$-modules, we need to prove that the following identities hold:
\begin{align}
\gamma_M^{n}\delta'_i&=\delta_i\gamma_M^{n-1} \quad 0 \leq i \leq n \label{3.1}\\ 
\gamma_M^n \sigma'_i &=\sigma_i\gamma_M^{n+1} \quad 0 \leq i \leq n \label{3.2}\\ 
\gamma_M^n\tau'_n &= \tau_n\gamma_M^n \label{3.3}
\end{align}
For $0 \leq i \leq n-1$, we have
\begin{align*}
&\big((\gamma_M^{n}\delta'_i)(m \otimes_H c^0 \otimes \ldots \otimes c^{n-1})\big)(f^0 \otimes \ldots \otimes f^n)&\\
& \quad =\big(\gamma_M^n(m \otimes_H c^0 \otimes \ldots \otimes c^i_1 \otimes c^{i}_2 \otimes \ldots \otimes c^{n-1})\big)(f^0 \otimes \ldots \otimes f^n)&\\
& \quad =\phi\big(m \otimes (c^0f^0)\ldots (c^i_1f^i)(c^i_2f^{i+1}) \ldots (c^{n-1}f^n)\big)&\\
& \quad =\phi\big(m \otimes (c^0f^0)\ldots \left(c^i(f^if^{i+1})\right) \ldots (c^{n-1}f^n)\big)& (\text{by } \eqref{Caction})\\
& \quad =\big(\gamma_M^{n-1}(m \otimes_H c^0 \otimes \ldots \otimes c^{n-1})\big)(f^0 \otimes \ldots \otimes f^if^{i+1} \otimes \ldots \otimes f^n)&\\
& \quad =\big((\delta_i\gamma_M^{n-1})(m \otimes_H c^0 \otimes \ldots \otimes c^{n-1})\big)(f^0 \otimes \ldots \otimes f^n)&
\end{align*}
Moreover,
\begin{align*}
&\big((\gamma_M^{n}\delta'_n)(m \otimes_H c^0 \otimes \ldots \otimes c^{n-1})\big)(f^0 \otimes \ldots \otimes f^n)\\
&\quad =\big(\gamma_M^n(m_{(0)} \otimes_H c^0_2 \otimes c^1 \otimes \ldots \otimes c^{n-1} \otimes m_{(-1)}c^0_1)\big)(f^0 \otimes \ldots \otimes f^n)\\
&\quad= \phi\Big(m_{(0)} \otimes (c^0_2f^0)(c^1f^1) \ldots (c^{n-1}f^{n-1})\big((m_{(-1)}c^0_1)f^n\big)\Big)\\
&\quad= \phi\Big(m_{(0)(0)} \otimes [S^{-1}(m_{(0)(-1)})\big((m_{(-1)}c^0_1)f^n\big)] (c^0_2f^0)(c^1f^1) \ldots (c^{n-1}f^{n-1})\Big)\quad (\text{by } \eqref{3.3g})\\
&\quad = \phi\big(m \otimes (c^0_1f^n)(c^0_2f^0)(c^1f^1) \ldots (c^{n-1}f^{n-1})\big)\\
&\quad = \phi\big(m \otimes \left(c^0(f^nf^0)\right)(c^1f^1) \ldots (c^{n-1}f^{n-1})\big)\\
&\quad =\big(\gamma_M^{n-1}(m \otimes_H c^0 \otimes \ldots \otimes c^{n-1})\big)(f^nf^0 \otimes f^1 \otimes \ldots \otimes f^{n-1})\\
&\quad =\big((\delta_n\gamma_M^{n-1})(m \otimes_H c^0 \otimes \ldots \otimes c^{n-1})\big)(f^0 \otimes \ldots \otimes f^n)
\end{align*}
This proves \eqref{3.1}. Next, we verify the identity \eqref{3.2}. For $0 \leq i \leq n-1$, we have
\begin{align*}
&\big((\sigma_i\gamma_M^{n+1})(m \otimes_H c^0 \otimes \ldots \otimes c^{n+1})\big)(f^0 \otimes \ldots \otimes f^n)\\
&\quad =\big(\gamma_M^{n+1}(m \otimes_H c^0 \otimes \ldots \otimes c^{n+1})\big)(f^0 \otimes \ldots \otimes f^i \otimes id_{X_{i+1}} \otimes f^{i+1} \otimes \ldots \otimes f^n)\\
& \quad =\phi\big(m \otimes (c^0f^0)\ldots (c^if^i)(c^{i+1}id_{X_{i+1}}) \ldots (c^{n+1}f^n)\big)\\
& \quad =\phi\big(m \otimes (c^0f^0)\ldots (c^if^i)\left(\varepsilon_C(c^{i+1})id_{X_{i+1}}\right)(c^{i+2}f^{i+1}) \ldots (c^{n+1}f^n)\big)\\
&\quad =\big(\gamma_M^{n}(m \otimes_H c^0 \otimes \ldots \otimes \varepsilon_C(c^{i+1}) \otimes \ldots \otimes c^{n+1})\big)(f^0  \otimes  \ldots \otimes f^i \otimes  f^{i+1} \otimes \ldots \otimes f^n)\\
&\quad= \big((\gamma_M^{n}\sigma'_i)(m \otimes_H c^0 \otimes \ldots \otimes c^{n+1})\big)(f^0 \otimes \ldots \otimes f^n)
\end{align*}
It may be verified similarly that  $\sigma_n\gamma_M^{n+1}=\gamma_M^n \sigma_n'$. It remains to verify \eqref{3.3}. We have
\begin{align*}
&\big((\gamma_M^{n}\tau'_n)(m \otimes_H c^0 \otimes \ldots \otimes c^{n})\big)(f^0 \otimes \ldots \otimes f^n)\\
&\quad = \big(\gamma_M^n(m_{(0)} \otimes_H c^1 \otimes \ldots \otimes c^{n} \otimes m_{(-1)}c^0)\big)(f^0 \otimes \ldots \otimes f^n)\\
& \quad = \phi\big(m_{(0)} \otimes (c^1f^0)\ldots (c^{n}f^{n-1})\left((m_{(-1)}c^0)f^n\right)\big)\\
& \quad = \phi\big(m_{(0)(0)} \otimes [S^{-1}(m_{(0)(-1)})\left((m_{(-1)}c^0)f^n\right)] (c^1f^0) \ldots (c^{n}f^{n-1})\big)\\
& \quad = \phi\left(m \otimes (c^0f^n) (c^1f^0) \ldots (c^{n}f^{n-1})\right)\\
&\quad= \big(\gamma_M^{n}(m \otimes_H c^0 \otimes \ldots \otimes c^{n})\big)(f^n \otimes f^0 \otimes \ldots \otimes f^{n-1})\\
&\quad =\big((\tau_n\gamma_M^{n})(m \otimes_H c^0 \otimes \ldots \otimes c^{n})\big)(f^0 \otimes \ldots \otimes f^n)
\end{align*}
This completes the proof.
\end{proof}

We will now extend the result in  Theorem \ref{p2} to a general pairing
\begin{equation}\label{th4-1}
HC^q_H(C,M) \otimes HC^p_H(\mathcal{D}_H,M)  \longrightarrow HC^{q+p}(\mathcal{D}_H)
\end{equation}

Let $\mathcal{D}_H$ be a left $H$-category and $C$ be a left $H$-module coalgebra. Let $C^{n}_{(C,M,\mathcal{D}_H)}$ be the diagonal complex
\begin{equation*}
C^{n}_{(C,M,\mathcal{D}_H)}:=C^n_H(C,M) \otimes_k C^n_H(\mathcal{D}_H,M)=(M \otimes_H C^{n+1} ) \otimes_k Hom_H(M \otimes CN_n(\mathcal{D}_H),k) \qquad \forall \textrm{ }n\geq 0
\end{equation*}
which is a cocyclic module with structure maps $\{\delta_i' \otimes \delta_i, \sigma_i' \otimes \sigma_i, \tau_n' \otimes \tau_n\}_{0\leq i \leq n}$ (see \cite[$\S$ 2.5.1.2]{Loday}). 

\smallskip
We consider the $k$-linear category $(C,\mathcal{D}_H)$ defined as follows:
\begin{equation*}
\begin{array}{c}
Ob(C,\mathcal{D}_H)=Ob(\mathcal{D}_H)\\
Hom_{(C,\mathcal{D}_H)}(X,Y)=Hom_H\left(C,Hom_{{\mathcal{D}}_H}(X,Y)\right)
\end{array}
\end{equation*}
The composition in $(C,\mathcal{D}_H)$ is given by $(f \ast g)(c)=f(c_1) \circ g(c_2)$ for any $g\in Hom_{(C,\mathcal{D}_H)}(X,Y)$, $f\in Hom_{(C,\mathcal{D}_H)}(Y,Z)$ and $c \in C$.

\begin{prop}\label{Teorem4.3}
Let $\mathcal{D}_H$ be a left $H$-category and $C$ be a left $H$-module coalgebra. Then, the map
\begin{equation*}
\Psi:C^{n}_{(C,M,\mathcal{D}_H)}=(M \otimes_H C^{n+1} ) \otimes Hom_H(M \otimes CN_n(\mathcal{D}_H),k) \longrightarrow C^n(C,\mathcal{D}_H)=Hom_k(CN_n(C,\mathcal{D}_H),k)
\end{equation*}
given by
\begin{equation*}
\left(\Psi(m \otimes_H c^0 \otimes \ldots \otimes c^{n} \otimes \phi)\right)(g^0 \otimes \ldots \otimes g^{n}):=\phi(m \otimes g^0(c^0) \otimes \ldots \otimes g^n(c^{n}) )
\end{equation*}
determines a morphism of cocyclic modules.
\end{prop}
\begin{proof}
We first verify that $\Psi$ is well-defined. We have
\begin{equation*}
\begin{array}{lll}
&\left(\Psi(mh \otimes_H c^0 \otimes \ldots \otimes c^{n} \otimes \phi)\right)(g^0 \otimes \ldots \otimes g^{n})\\
& \quad =\phi(mh \otimes g^0(c^0) \otimes \ldots \otimes g^n(c^{n}) )\\
& \quad =\phi(m \otimes h_1g^0(c^0) \otimes \ldots \otimes h_{n+1}g^n(c^{n}) )& (\text{by Lemma } \ref{lem2.3f}) \\
& \quad =\phi(m \otimes g^0(h_1c^0) \otimes \ldots \otimes g^n(h_{n+1}c^{n}) )\\
& \quad= \left(\Psi(m \otimes_H h_1c^0 \otimes \ldots \otimes h_{n+1}c^{n} \otimes \phi)\right)(g^0 \otimes \ldots \otimes g^{n})\\
& \quad= \left(\Psi(m \otimes_H h(c^0 \otimes \ldots \otimes c^{n}) \otimes \phi)\right)(g^0 \otimes \ldots \otimes g^{n})\\
\end{array}
\end{equation*}
For $0 \leq i \leq n-1$, we have
\begin{equation*}
\begin{array}{lll}
&\left((\Psi \circ (\delta_i' \otimes \delta_i))(m \otimes_H c^0 \otimes \ldots \otimes c^{n} \otimes \phi)\right)(g^0 \otimes \ldots \otimes g^{n+1})\\
& \quad =\left(\Psi(m \otimes_H c^0 \otimes \ldots \otimes c^i_1 \otimes c^i_2 \otimes \ldots \otimes c^{n} \otimes \delta_i(\phi))\right)(g^0 \otimes \ldots \otimes g^{n+1})\\
&\quad= \delta_i(\phi)(m \otimes g^0(c^0) \otimes \ldots \otimes g^i(c^i_1) \otimes g^{i+1}(c^i_2) \otimes \ldots \otimes g^{n+1}(c^n))\\
& \quad =\phi(m \otimes g^0(c^0) \otimes \ldots \otimes g^i(c^i_1) \circ g^{i+1}(c^i_2) \otimes \ldots \otimes g^{n+1}(c^n))\\
& \quad =\phi(m \otimes g^0(c^0) \otimes \ldots \otimes (g^i \ast g^{i+1})(c^i) \otimes \ldots \otimes g^{n+1}(c^n))\\
&\quad=\left((\delta_i \circ \Psi)(m \otimes_H c^0 \otimes \ldots \otimes c^{n} \otimes \phi)\right)(g^0 \otimes \ldots \otimes g^{n+1})\\
\end{array}
\end{equation*}
The case $i=n$ can be verified similarly.
Further, for $0 \leq i \leq n-1$, we have
\begin{equation*}
\begin{array}{ll}
&\left((\Psi \circ (\sigma_i' \otimes \sigma_i))(m \otimes_H c^0 \otimes \ldots \otimes c^{n} \otimes \phi)\right)(g^0 \otimes \ldots \otimes g^{n-1})\\
&\quad=\Psi\left(m \otimes_H c^0 \otimes \ldots \otimes  \varepsilon(c^i) \otimes \ldots \otimes c^{n} \otimes \sigma_i(\phi)\right)\\
&\quad= \sigma_i(\phi)(m \otimes g^0(c^0) \otimes \ldots \otimes g^i(c^{i+1}) \otimes \ldots \otimes g^{n-1}(c^n)) \varepsilon(c^i)\\
&\quad=\phi\left(g^0(c^0) \otimes \ldots \otimes \varepsilon(c^i)id \otimes g^i(c^{i+1}) \otimes \ldots \otimes g^{n-1}(c^n)\right)\\
&\quad=\left((\sigma_i \circ \Psi)(m \otimes_H c^0 \otimes \ldots \otimes c^{n} \otimes \phi)\right)(g^0 \otimes \ldots \otimes g^{n-1})\\
\end{array}
\end{equation*}
The case $i=n$ can be verified similarly. We also have
\begin{equation*}
\begin{array}{ll}
&\left((\Psi \circ (\tau'_n \otimes \tau_n))(m \otimes_H c^0 \otimes \ldots \otimes c^{n} \otimes \phi)\right)(g^0 \otimes \ldots \otimes g^{n})\\
&\quad=\Psi(m_{(0)} \otimes  c^1 \otimes \ldots \otimes c^{n} \otimes m_{(-1)}c^0 \otimes \tau_n(\phi))(g^0 \otimes \ldots \otimes g^{n})\\
&\quad =\tau_n(\phi)(m_{(0)} \otimes_H  g^0(c^1) \otimes \ldots \otimes g^{n-1}(c^{n}) \otimes g^{n}(m_{(-1)}c^0))\\
&\quad=\phi\left(m_0 \otimes S^{-1}(m_{(-2)})g^{n}(m_{(-1)}c^0)\otimes  g^0(c^1) \otimes \ldots \otimes g^{n-1}(c^{n})\right)\\
&\quad=\phi\left(m \otimes g^{n}(c^0) \otimes  g^0(c^1) \otimes \ldots \otimes g^{n-1}(c^{n})\right)\\
&\quad=\left((\tau_n \circ \Psi)(m \otimes_H c^0 \otimes \ldots \otimes c^{n} \otimes \phi)\right)(g^0 \otimes \ldots \otimes g^{n})\\
\end{array}
\end{equation*}
This completes the proof.
\end{proof}

\begin{theorem}\label{thm8.15}
Let $M$ be a right-left SAYD module over $H$.  Let $C$ be a left $H$-module coalgebra and let $\mathcal{D}_H$ be a left $H$-category. Then, we have  a pairing
\begin{equation}\label{rpi1}
HC^q_H(C,M) \otimes HC^p_H(\mathcal{D}_H,M) \longrightarrow HC^{p+q}(C,\mathcal{D}_H)
\end{equation}
Additionally, suppose that   $C$ acts on $\mathcal{D}_H$ in the sense of Definition \ref{Def3.1s}. Then, we obtain a pairing:
\begin{equation}\label{rpi2}
HC^q_H(C,M) \otimes HC^p_H(\mathcal{D}_H,M) \longrightarrow HC^{p+q}(\mathcal{D}_H)
\end{equation}
\end{theorem}

\begin{proof} We let $\mathcal{B}(C^\bullet_H(C,M))$ and $\mathcal{B}(C^\bullet_H(\mathcal{D}_H,M))$ denote respectively the mixed complexes corresponding to the cocyclic modules $C^\bullet_H(C,M)$ and $C^\bullet_H(\mathcal{D}_H,M)$. By definition, $HC^q_H(C,M) =H^q(Tot(\mathcal{B}(C^\bullet_H(C,M))))$ and $HC^p_H(\mathcal{D}_H,M)=H^p(Tot(\mathcal{B}(C^\bullet_H(\mathcal{D}_H,M))))$. We have a canonical morphism
\begin{equation}\label{pomc1}
\begin{CD}
HC^q_H(C,M)\otimes HC^p_H(\mathcal{D}_H,M)=H^q(Tot(\mathcal{B}(C^\bullet_H(C,M))))\otimes H^p(Tot(\mathcal{B}(C^\bullet_H(\mathcal{D}_H,M))))\\
@VVV \\
H^{p+q}( Tot(\mathcal{B}(C^\bullet_H(C,M)))\otimes Tot(\mathcal{B}(C^\bullet_H(\mathcal{D}_H,M))))
\\ @V\cong VV \\
HC^{p+q}(C^\bullet_{(C,M,\mathcal D_H)})\\
\end{CD}
\end{equation}
where the vertical isomorphism follows from Eilenberg-Zilber Theorem \cite[$\S$ 4.3.8]{Loday}. The morphism of cocyclic modules in Proposition \ref{Teorem4.3} induces a morphism
$HC^{p+q}(C^\bullet_{(C,M,\mathcal D_H)})\longrightarrow HC^{p+q}(C,\mathcal D_H)$. Composing with the morphism in \eqref{pomc1} gives us the pairing in \eqref{rpi1}. 

\smallskip
Finally, when $C$ acts on $\mathcal D_H$, we have an inclusion $i:\mathcal D_H\hookrightarrow (C,\mathcal D_H)$ given by $i(f)(c):=cf$ for any morphism $f\in Hom_{\mathcal D_H}(X,Y)$
and $c\in C$. Then, $i$ induces a morphism $HC^{p+q}(C,\mathcal D_H)\longrightarrow HC^{p+q}(\mathcal D_H)$. Composing with the pairing in \eqref{rpi1} now gives us
the pairing in \eqref{rpi2}. 
\end{proof}
  
\section{Traces, cocycles and DGH-semicategories}\label{SectionDG}

Our purpose is to develop a formalism analogous to that of Connes \cite{C2} in order to interpret the cocycles $Z^\bullet_H(\mathcal D_H,M)$, $Z^\bullet(\mathcal{D}_H)$  and the coboundaries $B^\bullet_H(\mathcal D_H,M)$, $B^\bullet(\mathcal{D}_H)$ as characters of differential graded  semicategories. 
 In this section, we  will describe $Z^\bullet_H(\mathcal{D}_H,M)$ and  $Z^\bullet(\mathcal{D}_H)$, for which we will need  the framework of DG-semicategories. Let us first recall the notion of a semicategory introduced by Mitchell in \cite{Mit} (for more on semicategories, see, for instance, \cite{BoLS}). 

\begin{definition}(see \cite[Section 4]{Mit})
A semicategory $\mathcal C$ consists of a collection $Ob(\mathcal{C})$ of objects together with a set of morphisms $Hom_{\mathcal{C}}(X,Y)$ for each $X,Y \in Ob(\mathcal{C})$ and an associative composition. A semifunctor $F:\mathcal{C} \longrightarrow \mathcal{C}'$ between semicategories assigns an object $F(X) \in Ob(\mathcal{C}')$ to each $X \in Ob(\mathcal{C})$
and a morphism $F(f) \in Hom_{\mathcal C'}(F(X),F(Y))$ to each $f \in Hom_\mathcal{C}(X,Y)$ and preserves composition. 

\smallskip
A left $H$-semicategory is a small $k$-linear semicategory $\mathcal{S}_H$ such that
\begin{itemize}
\item[(i)] $Hom_{\mathcal{S}_H}(X,Y)$ is a left $H$-module for all $X,Y \in Ob(\mathcal{S}_H)$ 
\item[(ii)] $h(gf)=(h_1g)(h_2f)$
for any $h \in H$, $f \in Hom_{\mathcal{S}_H}(X,Y)$ and  $g \in Hom_{\mathcal{S}_H}(Y,Z)$. 
\end{itemize}
\end{definition}

It is clear that any ordinary category may be treated as a semicategory. Conversely, to any $k$-semicategory $\mathcal{C}$, we can associate an ordinary $k$-category
$\tilde{\mathcal{C}}$ by adjoining  unit morphisms as follows:
\begin{align*}
Ob(\tilde{\mathcal{C}}):&=Ob(\mathcal{C})\\
Hom_{\tilde{\mathcal{C}}}(X,Y):&=\left\{\begin{array}{ll} Hom_\mathcal{C}(X,X) \bigoplus k &  \mbox{if $X=Y$}\\
Hom_\mathcal{C}(X,Y) & \mbox{if $X \neq Y$} \\ \end{array}\right.
\end{align*}

A  morphism in $Hom_{\tilde{\mathcal{C}}}(X,Y)$ will be denoted by $\tilde{f}=f+\mu$, where $f \in Hom_{\mathcal{C}}(X,Y)$ and $\mu \in k$. It is understood that $\mu=0$ whenever $X \neq Y$. Any semifunctor $F:\mathcal{C} \longrightarrow \mathcal{D}$ where $\mathcal D$ is an ordinary category may  be extended to an ordinary functor $\tilde{F}:\tilde{\mathcal{C}} \longrightarrow \mathcal{D}$. If $\mathcal S_H$ is a left $H$-semicategory, we note that $\tilde{\mathcal S}_H$ is a left $H$-category in the sense of Definition \ref{defH-cat}. 

\smallskip
Next we  recall the notion of the tensor product of complexes. Let $(A^\bullet,\partial_A):=\ldots \longrightarrow A^n \stackrel{\partial^n_A}\longrightarrow A^{n+1} \longrightarrow \ldots  $ and $(B^\bullet,\partial_B):=\ldots \longrightarrow B^n \stackrel{\partial^n_B}\longrightarrow B^{n+1} \longrightarrow \ldots$ be two cochain complexes. Then, their tensor product $A^\bullet \otimes B^\bullet$ also forms a cochain complex which is defined as follows:
\begin{align*}
(A^\bullet \otimes B^\bullet)^n:&=\bigoplus_{i+j=n} A^i \otimes B^j\\
\partial^n_{A \otimes B}:&=\bigoplus_{i+j=n} \left(\partial^i_A \otimes 1_{B_j} + (-1)^i 1_{A_i} \otimes \partial^j_B\right)
\end{align*} 

\begin{definition}
A differential graded semicategory (DG-semicategory) $(\mathcal{S},\hat\partial)$ is a $k$-linear semicategory $\mathcal{S}$ such that
\begin{itemize}
\item[(i)] $Hom^\bullet_\mathcal{S}(X,Y)=\big(Hom^n_\mathcal{S}(X,Y),\hat\partial^n_{XY}\big)_{n \geq 0}$ is a cochain complex of $k$-spaces for each $X,Y \in Ob(\mathcal{S})$.

\item[(ii)] the composition map
$$Hom^\bullet_\mathcal{S}(Y,Z) \otimes Hom^\bullet_\mathcal{S}(X,Y) \longrightarrow Hom^\bullet_\mathcal{S}(X,Z)$$
is a morphism of complexes. Equivalently,
\begin{equation}
\hat\partial^n_{XZ}(gf)=\hat\partial^{n-r}_{YZ}(g)f+(-1)^{n-r}g\hat\partial^{r}_{XY}(f) \label{comp}
\end{equation}
for any $f \in Hom_\mathcal{S}(X,Y)^r$ and $g \in Hom_\mathcal{S}(Y,Z)^{n-r}$.
\end{itemize} Whenever the meaning is clear from context, we will drop the subscript and simply write $\hat\partial^\bullet$ for the differential
on any $Hom^\bullet_{\mathcal S}(X,Y)$. 
\end{definition}

A DG-semicategory with a single object is the same as a differential graded (but not necessarily unital) $k$-algebra. Accordingly, any small DG-semicategory may be treated as a differential 
graded (but  not necessarily unital) $k$-algebra with several objects. The DG-semicategories may be treated in a manner similar to DG-categories (see, for instance, \cite{Ke1}, \cite{Ke2}). 

\begin{definition}
A DG-semifunctor $\alpha:(\mathcal{S},\hat\partial)\longrightarrow (\mathcal{S}',\hat\partial')$ between two DG-semicategories  is a $k$-linear semifunctor $\alpha:\mathcal{S} \longrightarrow \mathcal{S}'$ such that the induced map $Hom^\bullet_\mathcal{S}(X,Y) \longrightarrow Hom^\bullet_{\mathcal{S}'}(\alpha X,\alpha Y)$, $f \mapsto \alpha (f)$, is a morphism of complexes for each $X,Y \in Ob(\mathcal{S})$. 
\end{definition}

\begin{remark}\label{0dg}
We observe that corresponding to any DG-semicategory $\mathcal{S}$, there is a semicategory $\mathcal{S}^0$ defined as:
\begin{align*}
Ob(\mathcal{S}^0):&=Ob(\mathcal{S})\\
Hom_{\mathcal{S}^0}(X,Y):&=Hom^0_\mathcal{S}(X,Y)
\end{align*}
The composition   in $\mathcal{S}$ induces a well-defined composition $Hom_{\mathcal{S}^0}(Y,Z) \otimes Hom_{\mathcal{S}^0}(X,Y) \longrightarrow Hom_{\mathcal{S}^0}(X,Z)$.
\end{remark}

We now construct a ``universal DG-semicategory'' associated to a given $k$-linear semicategory, similar to the construction of the universal differential graded  algebra
associated to a  (not necessarily unital) $k$-algebra (see, for instance, \cite[p. 315]{C2}). 

\smallskip
Let $\Omega\mathcal{C}$ be the semicategory with $Ob(\Omega\mathcal{C}):=Ob(\mathcal{C})$ and  $Hom_{\Omega \mathcal{C}}(X,Y)=\bigoplus\limits_{n \geq 0}Hom^n_{\Omega \mathcal{C}}(X,Y)$, where
\begin{equation}\label{xudga} Hom^n_{\Omega\mathcal{C}}(X,Y):=
\left\{
\begin{array}{ll}
Hom_{\mathcal C}(X,Y) & \mbox{if $n=0$} \\
\\ \underset{(X_1,...,X_n)\in Ob(\mathcal C)^n}{\mbox{\Large $\bigoplus$}}Hom_{\tilde{\mathcal C}}(X_1,Y)\otimes Hom_{\mathcal C}
(X_2,X_1)\otimes \dots \otimes Hom_{\mathcal C}(X,X_n) & \mbox{if $n\geq 1$} \\
\end{array} \right.
\end{equation}
Here the sum runs over the ordered tuples $(X_1,...,X_n)\in Ob(\mathcal C)^n$.
 In particular, $(\Omega\mathcal{C})^0={\mathcal{C}}$.  For $n\geq 1$, an element of the form $\tilde{f}^0\otimes f^1\otimes ...\otimes f^n$ in $Hom^n_{\Omega\mathcal{C}}(X,Y)$ will be denoted by  $\tilde{f}^0df^1 \ldots df^n=(f^0+\mu) df^1\dots df^n$   
and said to be homogeneous of degree $n$. By abuse of notation, we will continue to use $\tilde{f}^0df^1 \ldots df^n=(f^0+\mu) df^1\dots df^n$ to denote an element 
of $Hom^n_{\Omega\mathcal{C}}(X,Y)$ even when $n=0$. In that case, it will be understood that $\mu=0$. 

\smallskip
The composition in $\Omega\mathcal{C}$ is determined by
\begin{equation}\label{symb}
f^0\circ df^1\circ \dots \circ df^n= f^0df^1\dots df^n \qquad (df^0)\circ f^1=d(f^0f^1)-f^0(df^1) \qquad  df^1\circ \dots \circ df^n= df^1\dots df^n
\end{equation}
In particular, it follows that
\begin{equation}\label{compoDG}
 \begin{array}{l}
((f^0+\mu)df^1...df^i)\cdot ((g^0+\mu')dg^1...dg^j)\\
= (f^0+\mu)\left(df^1....df^{i-1}d(f^ig^0)dg^1...dg^j+\underset{l=1}{\overset{i-1}{\sum}}(-1)^{i-l}df^1...d(f^lf^{l+1})...df^idg^0dg^1...dg^j\right)\\
\quad + (-1)^i(f^0+\mu)f^1df^2...df^idg^0dg^1...dg^j+\mu' (f^0+\mu)df^1...df^idg^1...dg^j\\
\end{array}
\end{equation}
For each $X,Y \in Ob(\Omega\mathcal{C})$,  the differential $\partial^n_{XY}:Hom^n_{\Omega\mathcal{C}}(X,Y) \longrightarrow Hom^{n+1}_{\Omega\mathcal{C}}(X,Y)$ is determined by setting
$$\partial^n_{XY}((f^0+\mu)df^1 \ldots df^n):=df^0df^1 \ldots df^n$$ It follows from definition that $\partial^{n+1}_{XY}\circ \partial^n_{XY}=0$. Therefore, $Hom^\bullet_{\Omega\mathcal{C}}(X,Y):=\big(Hom^n_{\Omega\mathcal{C}}(X,Y),\partial^n_{XY}\big)_{n \geq 0}$ is a cochain complex for each $X, Y \in Ob(\Omega\mathcal{C})$. It may also be verified  that the composition in $\Omega\mathcal{C}$ is a morphism of complexes. Thus, $\Omega\mathcal{C}$ is a DG-semicategory.
 
\begin{prop}\label{construniv} Let $\mathcal C$ be a small $k$-linear semicategory. Then, the associated DG-semicategory $(\Omega\mathcal C,\partial)$  is universal in the following sense: given
\begin{itemize}
\item[(i)] any DG-semicategory $(\mathcal{S},\hat\partial)$ and 
\item[(ii)] a $k$-linear semifunctor $\rho:\mathcal{C} \longrightarrow \mathcal{S}^0$,
\end{itemize}
there exists a unique DG-semifunctor $\hat{\rho}:(\Omega\mathcal{C},\partial) \longrightarrow (\mathcal{S},\hat\partial)$ such that the restriction of $\hat{\rho}$ to  the semicategory $\mathcal{C}$ is identical to $\rho: \mathcal{C} \longrightarrow \mathcal{S}^0$.
\end{prop}
\begin{proof}
We extend $\rho$ to obtain a DG-semifunctor $\hat{\rho}:(\Omega\mathcal{C},\partial) \longrightarrow (\mathcal{S},\hat{\partial})$ as follows:
\begin{equation}\label{cat1}
\begin{array}{c}
\hat{\rho}(X):=\rho(X)\\
\hat{\rho}((f^0+\mu)df^1\ldots df^n):=\rho(f^0)\circ\hat{\partial}^0(\rho(f^1))\circ \ldots\circ \hat{\partial}^0(\rho(f^n))+\mu \hat{\partial}^0(\rho(f^1)) \circ \ldots \circ \hat{\partial}^0(\rho(f^n))
\end{array}
\end{equation}
for all $X \in Ob(\Omega\mathcal{C})=Ob(\mathcal{C})$ and $(f^0+\mu)df^1\ldots df^n \in Hom^n_{\Omega\mathcal{C}}(X,Y)$, $n\geq 1$. Since each $\rho(f^i)$ is a morphism of degree $0$ in $\mathcal{S}$, it follows from \eqref{comp} and \eqref{compoDG} that 
\begin{equation} \hat{\rho}(((f^0+\mu)df^1...df^n)\circ ((f^{n+1}+\mu')df^{n+2}... df^m))=\hat{\rho}((f^0+\mu)df^1...df^n)\circ \hat{\rho}((f^{n+1}+\mu')df^{n+2}... df^m)
\end{equation} It is also clear by construction that $\hat{\rho}|_{\mathcal{C}}=\rho$. Moreover, we have
\begin{equation*}
\begin{array}{ll}
\hat{\partial}^n\left(\hat{\rho}((f^0+\mu)df^1\ldots df^n)\right)&= \hat{\partial}^n\left(\rho(f^0)\hat{\partial}^0(\rho(f^1)) \ldots \hat{\partial}^0(\rho(f^n))\right)+
\mu \hat{\partial}^n\left( \hat{\partial}^0(\rho(f^1)) \ldots \hat{\partial}^0(\rho(f^n))\right)\\
&= {\hat{\partial}}^0(\rho(f^0)) {\hat{\partial}}^0(\rho(f^1)) \ldots  {\hat{\partial}}^0(\rho(f^n))+\rho(f^0) {\hat{\partial}}^n\left({ \hat{\partial}}^0(\rho(f^1)) \ldots  {\hat{\partial}}^0(\rho(f^n))\right)\\
&= \hat{\partial}^0(\rho(f^0)) \hat{\partial}^0(\rho(f^1)) \ldots  \hat{\partial}^0(\rho(f^n))=\hat{\rho}\left(\partial^n((f^0+\mu)df^1\ldots df^n)\right)
\end{array}
\end{equation*}
The uniqueness of $\hat\rho$ is also clear from \eqref{symb} and \eqref{compoDG}. 
\end{proof}

\begin{definition}\label{DGH}
A left DGH-semicategory   is a left $H$-semicategory $\mathcal S_H$  equipped with a DG-semicategory $(\mathcal S_H,\hat\partial_H)$ structure such that for all $n\geq 0$:

\smallskip
(a)  $Hom^n_{\mathcal S_H}(X,Y)$ is a left $H$-module for $X,Y \in Ob(\mathcal S_H)$. 

(b) $\hat\partial^n_{H}:Hom^n_{\mathcal S_H}(X,Y) \longrightarrow Hom^{n+1}_{\mathcal S_H}(X,Y)$ is $H$-linear for $X,Y \in Ob(\mathcal S_H)$. 
\end{definition}

We can similarly define the notion of a DGH-semifunctor between DGH-semicategories. If $(\mathcal S_H,\hat\partial_H)$  is a left DGH-semicategory, we note that  $\mathcal S_H^0$  is a left $H$-semicategory.

\begin{prop}
Let $\mathcal{D}_H$ be a left $H$-category. Then, the universal DG-semicategory $(\Omega(\mathcal{D}_H),\partial_H)$ associated to $\mathcal{D}_H$ is a left DGH-semicategory with the $H$-action determined by
\begin{equation}\label{comp4.8}
h \cdot \left((f^0+\mu)df^1 \ldots df^n\right):=(h_1f^0+\mu \varepsilon(h_1))d(h_2f^1) \ldots d(h_{n+1}f^n) 
\end{equation}
for all $h \in H$ and $(f^0+\mu)df^1 \ldots df^n\in Hom_{\Omega(\mathcal D_H)}(X,Y)$.
\end{prop}
\begin{proof} This  is immediate from the definitions in \eqref{compoDG} and \eqref{comp4.8}.
\end{proof} 

\begin{definition}\label{gradedtrace}
Let $(\mathcal S_H,\hat\partial_H)$ be a left DGH-semicategory and $M$ be a right-left SAYD module over $H$. A closed graded $(H,M)$-trace of dimension $n$ on $\mathcal S_H$ is a collection of $k$-linear maps $$\hat{\mathscr{T}}^H:=\{\hat{\mathscr{T}}_X^H:M \otimes Hom^n_{\mathcal S_H}(X,X) \longrightarrow k\}_{X \in Ob(\mathcal S_H)}$$ such that
\begin{align}
&\hat{\mathscr{T}}_X^H\big(mh_1 \otimes S(h_2)f\big)=\varepsilon(h)\hat{\mathscr{T}}_X^H(m \otimes f)\label{gt0}\\
&\hat{\mathscr{T}}_X^H\big(m \otimes \hat\partial_{H}^{n-1}(f')\big)=0 \label{gt1}\\
&\hat{\mathscr{T}}_X^H\big(m \otimes g'g)=(-1)^{ij}~\hat{\mathscr{T}}_Y^{H}\big(m_{(0)} \otimes \left(S^{-1}(m_{(-1)})g\right)g'\big)\label{gt2}
\end{align}
for all $h \in H$, $m \in M$, $f \in Hom^n_{\mathcal S_H}(X,X)$, $f' \in Hom^{n-1}_{\mathcal S_H}(X,X)$, $g \in Hom^i_{\mathcal S_H}(X,Y)$, $g' \in Hom^j_{\mathcal S_H}(Y,X)$ and $i+j=n$.

\smallskip

In particular, putting $H=k=M$, we get: 
a closed graded trace of dimension $n$ on  a DG-semicategory ($\mathcal S,\hat\partial)$ is a collection of $k$-linear maps $\hat T:=\{\hat T_X: Hom^n_\mathcal{S}(X,X) \longrightarrow k\}_{X \in Ob(\mathcal{S})}$ such that
\begin{align}
&\hat T_X\left(\hat\partial^{n-1}(f)\right)=0 \label{gh1}\\
&\hat T_X(g'g)=(-1)^{ij}~ \hat T_Y(gg') \label{gh2}
\end{align}
for all $f \in Hom^{n-1}_{\mathcal S}(X,X)$, $g \in Hom^i_{\mathcal S}(X,Y)$, $g' \in Hom^j_{\mathcal S}(Y,X)$ and $i+j=n$.
\end{definition}

\begin{definition}\label{cycle} 
An $n$-dimensional $\mathcal S_H$-cycle with coefficients in a SAYD module $M$ is a triple 
$(\mathcal{S}_H,\hat{\partial}_H,\hat{\mathscr T}^H)$ such that
\begin{itemize}
\item[(i)] $(\mathcal S_H,\hat\partial_H)$ is a left DGH-semicategory.
\item[(ii)]  $\hat{\mathscr T}^H$ is a closed graded $(H,M)$-trace of dimension $n$ on $\mathcal{S}_H$.
\end{itemize} Let $\mathcal{D}_H$ be a left $H$-category. By an $n$-dimensional cycle over $\mathcal D_H$, we mean a tuple $(\mathcal{S}_H,\hat{\partial}_H, \hat{\mathscr T}^H,\rho)$ such that 
\begin{itemize}
\item[(i)] $(\mathcal{S}_H,\hat{\partial}_H, \hat{\mathscr T}^H)$ is an $n$-dimensional $\mathcal S_H$-cycle with coefficients in a SAYD module $M$.
\item[(ii)] $\rho:\mathcal{D}_H \longrightarrow \mathcal{S}_H^0$ is an $H$-linear semifunctor.
\end{itemize}
\end{definition}

In particular, putting $H=k=M$, we get:
\begin{definition}
An $n$-dimensional $\mathcal S$-cycle is a triple 
$(\mathcal{S},\hat{\partial},\hat{T})$ such that
\begin{itemize}
\item[(i)] $(\mathcal S,\hat\partial)$ is a DG-semicategory.
\item[(ii)]  $\hat{T}$ is a closed graded trace of dimension $n$ on $\mathcal{S}$.
\end{itemize} Let $\mathcal{C}$ be a small $k$-linear category. By an $n$-dimensional cycle over $\mathcal C$, we mean a tuple $(\mathcal{S},\hat{\partial},\hat{T},\rho)$ such that 
\begin{itemize}
\item[(i)] $(\mathcal{S},\hat{\partial},\hat{T})$ is an $n$-dimensional $\mathcal S$-cycle.
\item[(ii)] $\rho:\mathcal C \longrightarrow \mathcal{S}^0$ is a $k$-linear   semifunctor.
\end{itemize}
\end{definition}

We will denote by $Cat_H$ the category of all left $H$-categories with $H$-linear functors between them. We fix a left $H$-category $\mathcal D_H$. Given an $n$-dimensional cycle $(\mathcal{S}_H,\hat{\partial}_H, \hat{\mathscr T}^H,\rho)$  over $\mathcal{D}_H$, we define its character $\phi\in C^n_H(\mathcal D_H,M)$ by setting
\begin{equation*}
\phi:M\otimes CN_n(\mathcal D_H)\longrightarrow k\qquad \phi(m \otimes f^0 \otimes \ldots \otimes f^n):=\hat{\mathscr T}^H_{X_0}\big(m \otimes \rho(f^0)\hat{\partial}^0_H\left(\rho(f^1)\right) \ldots \hat{\partial}^0_H\left(\rho(f^n)\right)\big)
\end{equation*} 
for $m \in M$ and $f^0 \otimes \ldots \otimes f^n \in Hom_{\mathcal{D}_H}(X_1,X_0) \otimes Hom_{\mathcal{D}_H}(X_2,X_1) \otimes \ldots \otimes Hom_{\mathcal{D}_H}(X_0,X_n)$. We will often suppress the semifunctor $\rho$ and refer to
$\phi$ simply as the character of the $n$-dimensional cycle   $(\mathcal{S}_H,\hat{\partial}_H, \hat{\mathscr T}^H)$. 

\smallskip
The next result provides a characterization of the space $Z^n_H(\mathcal{D}_H,M)$ of $n$-cocycles in the Hopf-cyclic cohomology of the category $\mathcal{D}_H$ with coefficients in the SAYD module $M$.

\begin{theorem}\label{charcycl}
Let $\mathcal{D}_H$ be a left $H$-category and $M$ be a right-left SAYD module over $H$. Let $\phi \in C^n_H(\mathcal{D}_H,M)$. Then, the following conditions are equivalent:
\begin{itemize}
\item[(1)] $\phi$ is the character of an $n$-dimensional cycle over $\mathcal D_H$, i.e., there is an $n$-dimensional cycle $(\mathcal{S}_H,\hat{\partial}_H, \hat{\mathscr T}^H)$ with coefficients in $M$ and an $H$-linear semifunctor $\rho:\mathcal{D}_H \longrightarrow \mathcal{S}_H^0$ such that
\begin{equation}\label{eq1}
\begin{array}{ll}
\phi(m \otimes f^0 \otimes \ldots \otimes f^n)
&=\hat{\mathscr T}^H_{X_0}((id_M \otimes \hat{\rho})(m \otimes f^0df^1\ldots df^n))\vspace{0.02in}\\ 
&=\hat{\mathscr T}^H_{X_0}\big(m \otimes \rho(f^0)\hat{\partial}_H^0\left(\rho(f^1)\right) \ldots \hat{\partial}_H^0\left(\rho(f^n)\right)\big)\\
\end{array}
\end{equation}
for any $m \in M$ and $f^0 \otimes \ldots \otimes f^n \in Hom_{\mathcal{D}_H}(X_1,X_0) \otimes Hom_{\mathcal{D}_H}(X_2,X_1) \otimes \ldots \otimes Hom_{\mathcal{D}_H}(X_0,X_n)$.
\item[(2)] There exists a closed graded $(H,M)$-trace $\mathscr{T}^H$ of dimension $n$ on $\left(\Omega(\mathcal{D}_H),\partial_H\right)$ such that
\begin{equation}\label{eq2}
\phi(m \otimes f^0 \otimes \ldots \otimes f^n)=\mathscr{T}^H_{X_0}(m \otimes f^0df^1 \ldots df^n)
\end{equation}
for any $m \in M$ and $f^0 \otimes \ldots \otimes f^n \in Hom_{\mathcal{D}_H}(X_1,X_0) \otimes Hom_{\mathcal{D}_H}(X_2,X_1) \otimes \ldots \otimes Hom_{\mathcal{D}_H}(X_0,X_n)$.
\item[(3)] $\phi \in Z^n_H(\mathcal{D}_H,M)$. 
\end{itemize}
\end{theorem}
\begin{proof}
(1) $\Rightarrow$ (2). By the universal property of $\Omega(\mathcal{D}_H)$, the $H$-linear semifunctor $\rho:\mathcal{D}_H \longrightarrow \mathcal{S}_H^0$ can be extended to a DGH-semifunctor $\hat{\rho}:\Omega(\mathcal{D}_H) \longrightarrow \mathcal{S}_H$ as in \eqref{cat1}.  We define a collection $\mathscr{T}^H:=\{\mathscr{T}^H_X:M \otimes Hom^n_{\Omega(\mathcal{D}_H)}(X,X) \longrightarrow k\}_{X \in Ob(\Omega(\mathcal{D}_H))}$  of $k$-linear maps  given by
\begin{equation}\label{ver1x}\mathscr{T}_X^H(m \otimes (f^0+\mu)df^1 \ldots df^n):=\hat{\mathscr T}^H_{X}\big(m \otimes \hat\rho((f^0+\mu)df^1 \ldots df^n)\big)
\end{equation} for any $m \in M$ and $f^0 \otimes \ldots \otimes f^n \in Hom_{\mathcal{D}_H}(X_1,X) \otimes Hom_{\mathcal{D}_H}(X_2,X_1) \otimes \ldots \otimes Hom_{\mathcal{D}_H}(X,X_n)$. In particular, it follows from \eqref{ver1x} that
\begin{equation}
\phi(m \otimes f^0 \otimes \ldots \otimes f^n)=\hat{\mathscr T}^H_{X}\big(m \otimes \rho(f^0)\hat{\partial}_H^0\left(\rho(f^1)\right) \ldots \hat{\partial}_H^0\left(\rho(f^n)\right)\big)=\mathscr{T}_X^H(m \otimes f^0df^1 \ldots df^n)
\end{equation} 
 We now verify that the collection $\mathscr{T}^H$ is an $n$-dimensional closed graded $(H,M)$-trace on $\Omega(\mathcal{D}_H)$. For any $\alpha=(f^0+\mu)df^1 \ldots df^n \in Hom^n_{\Omega(\mathcal{D}_H)}(X,X)$ and $h \in H$, we have
 \begin{equation*}
 \mathscr{T}_X^H\big(mh_1 \otimes S(h_2)\alpha\big)=\hat{\mathscr T}^H_{X}(mh_1\otimes \hat\rho(S(h_2)\alpha))=\hat{\mathscr T}^H_{X}(mh_1\otimes (S(h_2)(\hat\rho(\alpha)))=\varepsilon(h)\mathscr{T}_X^H(m\otimes \alpha)
 \end{equation*}
Hence, $\mathscr T^H$ satisfies the condition \eqref{gt0}.
Further, for any $\beta=(p^0+\mu)dp^1 \ldots dp^{n-1} \in Hom^{n-1}_{\Omega(\mathcal{D}_H)}(X,X)$, we have
\begin{align*}
\mathscr{T}_X^H\big(m \otimes \partial^{n-1}_{H}(\beta))&=\hat{\mathscr{T}}_X^H(m \otimes \hat\rho(\partial^{n-1}_{H}(\beta)))=\hat{\mathscr{T}}_X^H(m \otimes \hat{\partial}^{n-1}_{H}(\hat\rho(\beta)))=0
\end{align*}
Hence,  $\mathscr T^H$ satisfies the condition \eqref{gt1}. Finally, we see that
\begin{align*}
\mathscr{T}_X^{H}(m \otimes g'g)&= \hat{\mathscr T}^H_X\left(m \otimes \hat{\rho}(g')\hat{\rho}(g)\right)=(-1)^{ij} \hat{\mathscr T}_Y^{H}\big(m_{(0)} \otimes \left(S^{-1}(m_{(-1)})\hat{\rho}(g)\right)\hat{\rho}(g')\big)\\
&=(-1)^{ij}~\mathscr{T}_Y^{H}\big(m_{(0)} \otimes \left(S^{-1}(m_{(-1)})g\right)g'\big)
\end{align*}
for any $g \in Hom^i_{\Omega(\mathcal{D}_H)}(X,Y)$, $g' \in Hom^j_{\Omega(\mathcal{D}_H)}(Y,X)$ with $i+j=n$. This proves the condition in \eqref{gt2}.

\smallskip
(2) $\Rightarrow$ (1). Suppose that we have a closed graded $(H,M)$-trace $\mathscr{T}^H$ of dimension $n$ on $\Omega(\mathcal{D}_H)$ satisfying \eqref{eq2}.  Then, the triple $(\Omega(\mathcal D_H), \partial_H, \mathscr{T}^H)$ forms an $n$-dimensional cycle over $\mathcal D_H$ with coefficients in $M$. Further, by observing that $\partial^0_{H}(f)=df$ for any $f \in Hom_{\mathcal{D}_H}(X,Y)$, we get \eqref{eq1}.

\smallskip
(1) $\Rightarrow$ (3). Let $(\mathcal{S}_H,\hat{\partial}_H, \hat{\mathscr T}^H)$ be an $n$-dimensional cycle over $\mathcal D_H$ with coefficients in $M$ and $\rho:\mathcal{D}_H \longrightarrow \mathcal{S}^0_H$ be an $H$-linear semifunctor satisfying 
$$\phi(m \otimes f^0 \otimes \ldots \otimes f^n)=\hat{\mathscr T}^H_{X_0}\big(m \otimes \rho(f^0)\hat{\partial}_H^0\left(\rho(f^1)\right) \ldots \hat{\partial}_H^0\left(\rho(f^n)\right)\big)$$
for any $m \in M$ and $f^0 \otimes \ldots \otimes f^n \in Hom_{\mathcal{D}_H}(X_1,X_0) \otimes Hom_{\mathcal{D}_H}(X_2,X_1) \otimes \ldots \otimes Hom_{\mathcal{D}_H}(X_0,X_n)$.
For simplicity of  notation, we will drop the functor $\rho$.
To show that $\phi$ is an $n$-cocycle, it suffices to check that (see \eqref{3.3y})
$$b(\phi)=0\quad \text{and} \quad (1-\lambda)(\phi)=0$$ where $b=\sum\limits_{i=0}^{n+1}(-1)^i\delta_i$ and $\lambda=(-1)^n\tau_n$. For any $p^0 \otimes \ldots \otimes p^{n+1} \in Hom_{\mathcal{D}_H}(X_1,X_0) \otimes Hom_{\mathcal{D}_H}(X_2,X_1) \otimes \ldots \otimes Hom_{\mathcal{D}_H}(X_0,X_{n+1})$, we have

\begin{align*}
&\sum\limits_{i=0}^{n+1}(-1)^i\delta_i(\phi)(m \otimes p^0 \otimes \ldots \otimes p^{n+1})\\
&= \sum\limits_{i=0}^{n}(-1)^i\phi(m \otimes p^0 \otimes \ldots \otimes p^ip^{i+1} \otimes \ldots \otimes p^{n+1})~ + (-1)^{n+1}\phi\big(m_{(0)} \otimes \big(S^{-1}(m_{(-1)})p^{n+1}\big)p^0 \otimes p^1 \otimes \ldots \otimes p^{n}\big)\\
&= \hat{\mathscr T}^H_{X_0}\big(m \otimes p^0p^1\hat{\partial}_H^0(p^2) \ldots \hat{\partial}_H^0(p^{n+1})\big) ~+ \sum\limits_{i=1}^{n}(-1)^i \hat{\mathscr T}^H_{X_0}\big(m \otimes p^0\hat{\partial}_H^0(p^1) \ldots \hat{\partial}_H^0(p^ip^{i+1})\ldots \hat{\partial}_H^0(p^{n+1})\big)~ +\\
& \quad (-1)^{n+1}\hat{\mathscr T}^H_{X_{n+1}}\big(m_{(0)} \otimes \big(S^{-1}(m_{(-1)})p^{n+1}\big)p^0 \hat{\partial}_H^0(p^1)\ldots \otimes \hat{\partial}_H^0(p^n)\big)
\end{align*}
Now using the equality $\hat{\partial}_H^0(fg)=\hat{\partial}_H^0(f)g+f\hat{\partial}_H^0(g)$ for any $f$ and $g$ of degree $0$, we have
 \begin{align*}
&\big(p^0\hat{\partial}_H^0(p^1) \ldots \hat{\partial}_H^0(p^n)\big)p^{n+1}\\
&=\sum\limits_{i=1}^n (-1)^{n-i} p^0\hat{\partial}_H^0(p^1) \ldots \hat{\partial}_H^0(p^ip^{i+1}) \ldots \hat{\partial}_H^0(p^{n+1}) + (-1)^n p^0p^1\hat{\partial}_H^0(p^2) \ldots \hat{\partial}_H^0(p^{n+1})
\end{align*}
Thus, using the condition in \eqref{gt2},  we obtain
\begin{align*}
&\sum\limits_{i=0}^{n+1}(-1)^i\delta_i(\phi)(m \otimes p^0 \otimes \ldots \otimes p^{n+1})\\
&= (-1)^n \hat{\mathscr T}^H_{X_0}\big(m \otimes \big( p^0\hat{\partial}_H^0(p^1) \ldots \hat{\partial}_H^0(p^n)\big)p^{n+1}\big) + (-1)^{n+1}\hat{\mathscr T}^H_{X_{n+1}}\big(m_{(0)} \otimes \big(S^{-1}(m_{(-1)})p^{n+1}\big)p^0 \hat{\partial}_H^0(p^1)\ldots  \hat{\partial}_H^0(p^n)\big)=0
\end{align*}

Next, using \eqref{gt1}, \eqref{gt2}, and the $H$-linearity of $\hat{\partial}_H$,  we have
\begin{align*}
&\big(\left(1-(-1)^n\tau_n\right)\phi\big)(m \otimes f^0 \otimes \ldots \otimes f^{n})\\
&=\phi(m \otimes f^0 \otimes \ldots \otimes f^{n})- (-1)^n\phi\big(m_{(0)} \otimes S^{-1}(m_{(-1)})f^n \otimes f^0 \otimes \ldots \otimes f^{n-1}\big)\\
&=\hat{\mathscr T}^H_{X_0}(m \otimes f^0\hat{\partial}_H^0(f^1) \ldots \hat{\partial}_H^0(f^n))-(-1)^n \hat{\mathscr T}^H_{X_n}\big(m_{(0)} \otimes \big(S^{-1}(m_{(-1)})f^n\big)\hat{\partial}_H^0(f^0)\hat{\partial}_H^0(f^1) \ldots \hat{\partial}_H^0(f^{n-1})\big)\\
&= (-1)^{n-1} \hat{\mathscr T}^H_{X_n}\big (m_{(0)} \otimes \big(S^{-1}(m_{(-1)})\hat{\partial}_H^0(f^n)\big)f^0\hat{\partial}_H^0(f^1) \ldots \hat{\partial}^0_H(f^{n-1}) \big)+\\
& \quad (-1)^{n-1} \hat{\mathscr T}^H_{X_n}\big(m_{(0)} \otimes \big(S^{-1}(m_{(-1)})f^n\big)\hat{\partial}^0_H(f^0) \hat{\partial}_H^0(f^1)\ldots \hat{\partial}^0_H(f^{n-1})\big)\\
&= (-1)^{n-1} \hat{\mathscr T}^H_{X_n}\big(m_{(0)} \otimes \hat{\partial}_H^{n-1}\big((S^{-1}(m_{(-1)})f^n)f^0\hat{\partial}_H^0(f^1) \ldots \hat{\partial}_H^0(f^{n-1})\big)\big)=0
\end{align*}

(3) $\Rightarrow$ (2). Let $\phi \in Z^n_H(\mathcal{D}_H,M)$. For each $X \in Ob(\Omega(\mathcal{D}_H))$, we define an $H$-linear map
$\mathscr{T}^H_X:M \otimes Hom^n_{\Omega(\mathcal{D}_H)}(X,X) \longrightarrow k$ given by
$$\mathscr{T}^H_X(m \otimes (f^0+\mu)df^1\ldots df^{n}):= \phi(m \otimes f^0 \otimes \ldots \otimes f^{n})$$
for $f^0 \otimes \ldots \otimes f^{n} \in Hom_{\mathcal{D}_H}(X_1,X) \otimes Hom_{\mathcal{D}_H}(X_2,X_1) \otimes \ldots \otimes Hom_{\mathcal{D}_H}(X,X_{n})$. We now verify that the collection $\{\mathscr{T}^n_X:M \otimes Hom^n_{\Omega(\mathcal{D}_H)}(X,X) \longrightarrow k\}_{X \in Ob(\Omega(\mathcal{D}_H))}$ is a closed graded $(H,M)$-trace on $(\Omega(\mathcal{D}_H),\partial_H)$. For any $(p^0+\mu)dp^1\ldots dp^{n-1} \in Hom^{n-1}_{\Omega(\mathcal{D}_H)}(X,X)$, we have
\begin{align*}
\mathscr{T}_X^{H}\big(m \otimes \partial_H^{n-1}((p^0+\mu)dp^1\ldots dp^{n-1})\big)&=\mathscr{T}_X^{H}\big(m \otimes 1dp^0dp^1\ldots dp^{n-1}\big)=\phi(m \otimes 0 \otimes p^0 \otimes \ldots \otimes p^{n-1})=0
\end{align*}
This proves the condition in \eqref{gt1}. Using \eqref{new3.2}, it is also clear that $\{\mathscr{T}^n_X:M \otimes Hom^n_{\Omega(\mathcal{D}_H)}(X,X) \longrightarrow k\}_{X \in Ob(\Omega(\mathcal{D}_H))}$ satisfies condition \eqref{gt0}. Finally, for any $g'=(g^0+\mu')dg^1\ldots dg^{r} \in Hom^r_{\Omega(\mathcal{D}_H)}(Y,X)$ and $g=(g^{r+1}+\mu)dg^{r+2}\ldots dg^{n+1} \in Hom^{n-r}_{\Omega(\mathcal{D}_H)}(X,Y)$, we have
\begin{align*}
&\mathscr{T}_X^{H}\big(m \otimes g'g\big)\\&=\sum\limits_{j=1}^r (-1)^{r-j}~ \mathscr{T}_X^{H}\big(m \otimes (g^0+\mu')dg^1 \ldots d(g^jg^{j+1}) \ldots dg^{n+1}\big) + (-1)^r~ \mathscr{T}_X^{H}\big(m \otimes (g^0+\mu')g^1dg^2 \ldots dg^{n+1}\big)\\
&\textrm{ }+\mathscr{T}_X^{H}\big(m \otimes \mu(g^0+\mu')dg^1 \ldots dg^rdg^{r+2}\ldots dg^{n+1}\big) \\
&= \sum\limits_{j=1}^r (-1)^{r-j} \phi(m \otimes g^0 \otimes \ldots \otimes g^jg^{j+1} \otimes \ldots \otimes g^{n+1}) + (-1)^r~ \phi(m \otimes g^0g^1 \otimes g^2 \otimes \ldots \otimes g^{n+1})\\
& \textrm{ }+ (-1)^r~ \mu'\phi(m \otimes g^1 \otimes g^2 \otimes \ldots \otimes g^{n+1})+\mu \phi(m \otimes g^0\otimes g^1 \otimes...\otimes g^r\otimes g^{r+2} \otimes \ldots \otimes g^{n+1})\\
&= \sum\limits_{j=0}^r (-1)^{r+j} \phi(m \otimes g^0 \otimes \ldots \otimes g^jg^{j+1} \otimes \ldots \otimes g^{n+1})+ (-1)^r~ \mu'\phi(m \otimes g^1 \otimes g^2 \otimes \ldots \otimes g^{n+1})\\ &\textrm{ }+\mu \phi(m \otimes g^0\otimes g^1 \otimes...\otimes g^r\otimes g^{r+2} \otimes \ldots \otimes g^{n+1})
\end{align*}
On the other hand, we have

\begin{equation*}
\begin{array}{ll}
&(-1)^{r(n-r)}~\mathscr{T}_Y^{H}\Big(m_{(0)} \otimes \big(S^{-1}(m_{(-1)})g\big)g'\Big)\\
&=(-1)^{r(n-r)}~\mathscr{T}_Y^{H}\Big(m_{(0)} \otimes \left([S^{-1}\left((m_{(-1)})_{n-r+1}\right)(g^{r+1}+\mu)][d\left(S^{-1}\left((m_{(-1)})_{n-r}\right)g^{r+2}\right)] \ldots [d\left(S^{-1}\left((m_{(-1)})_{1}\right)g^{n+1}\right)]\right)\circ\\
&\qquad((g^0+\mu')dg^1\ldots dg^{r}) \Big)\\
&=(-1)^{r(n-r)} \sum\limits_{j=r+2}^{n} (-1)^{n-j+1} ~\mathscr{T}_Y^{H}\Big(m_{(0)} \otimes [S^{-1}\left((m_{(-1)})_{n-r}\right)(g^{r+1}+\mu)]\ldots d\big[\big(S^{-1}((m_{(-1)})_{n-j+1})(g^{j}g^{j+1})\big]\ldots dg^r)+\\
&\quad (-1)^{r(n-r)}~\mathscr{T}_Y^{H}\Big(m_{(0)} \otimes [S^{-1}\left((m_{(-1)})_{n-r+1}\right)(g^{r+1}+\mu)]\ldots d[\big(S^{-1}((m_{(-1)})_1)g^{n+1}\big)g^{0}] \ldots dg^r\Big)+\\
& \quad (-1)^{r(n-r)} (-1)^{n-r} \mathscr{T}_Y^{H}\Big(m_{(0)} \otimes \left([S^{-1}\left((m_{(-1)})_{n-r}\right)((g^{r+1}+\mu)g^{r+2})]\right) \ldots [d\left(S^{-1}\left((m_{(-1)})_{1}\right)g^{n+1}\right)]
(dg^0dg^1\ldots dg^{r}) \Big)\\
&\textrm{ }+(-1)^{r(n-r)}\mu'\mathscr{T}_Y^{H}\Big(m_{(0)} \otimes \left. [S^{-1}\left((m_{(-1)})_{n-r+1}\right)(g^{r+1}+\mu)][d\left(S^{-1}\left((m_{(-1)})_{n-r}\right)g^{r+2}\right)] \ldots [d\left(S^{-1}\left((m_{(-1)})_{1}\right)g^{n+1}\right)]\right.\\
&\qquad dg^1\ldots dg^{r} \Big)\\
&=(-1)^{r(n-r)} \sum\limits_{j=r+2}^{n}(-1)^{n-j+1} ~\phi\Big(m_{(0)} \otimes S^{-1}\left((m_{(-1)})_{n-r}\right)g^{r+1} \otimes \ldots \otimes \big(S^{-1}((m_{(-1)})_{n-j+1})(g^{j}g^{j+1}) \otimes \ldots \otimes g^r)+\\
&\quad (-1)^{r(n-r)}~\phi\Big(m_{(0)} \otimes S^{-1}\left((m_{(-1)})_{n-r+1}\right)g^{r+1} \otimes \ldots \otimes \big(S^{-1}((m_{(-1)})_1)g^{n+1}\big)g^{0} \otimes \ldots \otimes g^r\Big)+\\
& \quad (-1)^{r(n-r)} (-1)^{n-r} \phi\Big(m_{(0)} \otimes S^{-1}\left((m_{(-1)})_{n-r}\right)(g^{r+1}g^{r+2}) \otimes \ldots \otimes \left(S^{-1}\left((m_{(-1)})_{1}\right)g^{n+1}\right) \otimes  g^0 \otimes g^1 \otimes \ldots \otimes g^{r} \Big)\\
& \quad (-1)^{r(n-r)} (-1)^{n-r} \mu \phi\Big(m_{(0)} \otimes  S^{-1}\left((m_{(-1)})_{n-r}\right)g^{r+2} \otimes \ldots \otimes \left(S^{-1}\left((m_{(-1)})_{1}\right)g^{n+1}\right) \otimes  g^0 \otimes g^1 \otimes \ldots \otimes g^{r} \Big)\\
&\textrm{ }+(-1)^{r(n-r)}\mu'\phi\Big(m_{(0)} \otimes \left. S^{-1}\left((m_{(-1)})_{n-r+1}\right)g^{r+1}\otimes S^{-1}\left((m_{(-1)})_{n-r}\right)g^{r+2}\otimes \ldots \otimes (S^{-1}\left((m_{(-1)})_{1}\right)g^{n+1}\otimes \right. g^1\otimes \ldots \otimes g^{r} \Big) \\
\end{array}
\end{equation*}

Using repeatedly the fact that  $\phi=(-1)^n\tau_n\phi$, we get

\begin{align*}
&(-1)^{r(n-r)}~\mathscr T_Y^{n}\Big(m_{(0)} \otimes \big(S^{-1}(m_{(-1)})g\big)g'\Big)\\
&=-\sum\limits_{j=r+1}^{n} (-1)^{r+j} \phi(m \otimes g^0 \otimes \ldots \otimes g^jg^{j+1} \otimes \ldots \otimes g^{n+1})-(-1)^{n+r+1}\phi\Big(m_{(0)}  \otimes \big(S^{-1}(m_{(-1)})g^{n+1}\big)g^0 \otimes g^1 \otimes \ldots \otimes g^n\Big)\\
&\textrm{ } + (-1)^r~ \mu'\phi(m \otimes g^1 \otimes g^2 \otimes \ldots \otimes g^{n+1}) +\mu \phi(m \otimes g^0\otimes g^1 \otimes...\otimes g^r\otimes g^{r+2} \otimes \ldots \otimes g^{n+1})
\end{align*}

\smallskip
The condition \eqref{gt2} now follows using the fact that $b(\phi)=0$. This proves the result.
\end{proof}

\begin{remark}\label{rem4.11}
From the statement and proof of Theorem \ref{charcycl}, it is clear that there is a one to one correspondence between $n$-dimensional closed graded $(H,M)$-traces on $\Omega(\mathcal{D}_H)$ and $Z^n_H(\mathcal{D}_H,M)$.
\end{remark}

Substituting $H=k=M$ in Theorem \ref{charcycl}, we obtain the following categorical version of \cite[Proposition 1, p. 98]{C2}. This gives a characterization of $Z^n(\mathcal{C})$, the $n$-cocyles in the ordinary cyclic cohomology of a small $k$-linear category $\mathcal{C}$.

\begin{prop}\label{ordcharcycl}
Let $\mathcal{C}$ be a small $k$-linear category and $\phi \in CN^n(\mathcal{C})$. Then, the following conditions are equivalent:

\begin{itemize}
\item[(1)] $\phi$ is the character of an $n$-dimensional cycle over $\mathcal C$, i.e., there exists an $n$-dimensional cycle $(\mathcal{S},\hat{\partial}, \hat{T})$  and a $k$-linear semifunctor $\rho:\mathcal{C} \longrightarrow \mathcal{S}^0$ such that
\begin{equation}\label{eq1xx}
\begin{array}{ll}
\phi(f^0 \otimes \ldots \otimes f^n)&
=\hat{T}_{X_0}(\hat{\rho}(f^0df^1\ldots df^n))\vspace{0.02in}\\ 
&=\hat{T}_{X_0}\big(\rho(f^0)\hat{\partial}^0\left(\rho(f^1)\right) \ldots \hat{\partial}^0\left(\rho(f^n)\right)\big)\\
\end{array}
\end{equation}
for any $f^0 \otimes \ldots \otimes f^n \in Hom_{\mathcal{C}}(X_1,X_0) \otimes Hom_{\mathcal{C}}(X_2,X_1) \otimes \ldots \otimes Hom_{\mathcal{C}}(X_0,X_n)$.
\item[(2)] There exists a closed graded trace ${T}$ of dimension $n$ on $\left(\Omega \mathcal{C},\partial\right)$ such that
\begin{equation}\label{eq2xx}
\phi( f^0 \otimes \ldots \otimes f^n)={T}_{X_0}( f^0df^1 \ldots df^n)
\end{equation}
for any  $f^0 \otimes \ldots \otimes f^n \in Hom_{\mathcal{C}}(X_1,X_0) \otimes Hom_{\mathcal{C}}(X_2,X_1) \otimes \ldots \otimes Hom_{\mathcal{C}}(X_0,X_n)$.
\item[(3)] $\phi \in Z^n(\mathcal{C})$. 
\end{itemize}
\end{prop}

\begin{remark}\label{rem4.12}
By Proposition \ref{ordcharcycl}, it is clear that there is a one to one correspondence between $n$-dimensional closed graded traces on $\Omega\mathcal{C}$ and $Z^n(\mathcal{C})$.
\end{remark}

\section{Linearization by matrices and Hopf-cyclic cohomology}\label{Morita}
We continue with $M$ being a right-left SAYD module over $H$. In Section \ref{SectionDG}, we described the spaces $Z^\bullet_H(\mathcal{D}_H,M)$ and $Z^\bullet(\mathcal{D}_H)$. The next aim is to find a characterization of  $B^\bullet_H(\mathcal{D}_H,M)$ and $B^\bullet(\mathcal{D}_H)$ which will be done in several steps. For this, we will show in this section that the Hopf-cyclic cohomology of an $H$-category  $\mathcal D_H$ is the same as that of its linearization $\mathcal{D}_H \otimes M_r(k)$. We denote by $\overline{Cat}_H$ the category whose objects are left $H$-categories and whose morphisms are $H$-linear semifunctors.

\smallskip
Let $\mathcal{C}$ be a $k$-linear category and $r \in \mathbb{N}$. Then, its linearization $\mathcal{C} \otimes M_r(k)$ is the $k$-linear category defined as follows:
\begin{equation*}
\begin{array}{l}
Ob\left(\mathcal{C} \otimes M_r(k)\right):=Ob(\mathcal{C})\\
Hom_{\mathcal{C} \otimes M_r(k)}(X,Y):=Hom_\mathcal{C}(X,Y) \otimes M_r(k)
\end{array}
\end{equation*}
for any $X,Y \in Ob(\mathcal{C} \otimes M_r(k))$.

\begin{lemma}
Let $\mathcal{D}_H$ be a left $H$-category and let $r \in \mathbb{N}$. Then, $\mathcal{D}_H \otimes M_r(k)$ is also a left $H$-category.
\end{lemma}
\begin{proof}
This follows easily by defining the $H$-action on each $Hom_{\mathcal{D}_H \otimes M_r(k)}(X,Y)$ by setting $h(f \otimes B):=hf \otimes B$ for a morphism
$f\otimes B$ in $\mathcal D_H\otimes M_r(k)$.
\end{proof}

\begin{prop}\label{P6.2wp} Let $\mathcal{D}_H$ be a left $H$-category and let $M$ be a right-left SAYD module. Then:

\smallskip
(1) We obtain a para-cyclic module $C_\bullet(\mathcal{D}_H,M):=\{C_n(\mathcal{D}_H,M):=M \otimes CN_n(\mathcal{D}_H)\}_{n \geq 0}$ with the following structure maps  
\begin{align*}
d_i(m \otimes f^0 \otimes \ldots \otimes f^{n})&= \begin{cases}
m \otimes f^0 \otimes f^1 \otimes \ldots \otimes f^if^{i+1} \otimes \ldots \otimes f^{n} \quad~~~~~~~~~~~~~~~~ 0 \leq i \leq n-1\\
m_{(0)} \otimes \big(S^{-1}(m_{(-1)})f^n\big)f^0 \otimes f^1 \otimes \ldots \otimes f^{n-1} \quad~~~~~~~~~~~~i=n\\
\end{cases}\\
\vspace{.2cm}
s_i(m \otimes f^0 \otimes \ldots \otimes f^{n})&= \begin{cases} m \otimes f^0 \otimes f^1 \otimes \ldots f^i \otimes id_{X_{i+1}}\otimes f^{i+1} \otimes \ldots \otimes f^{n} \quad~~~~~ 0 \leq i \leq n-1\\
m \otimes f^0 \otimes f^1 \otimes \ldots \otimes f^{n} \otimes id_{X_{0}} \quad~~~~~~~~~~~~~~~~~~~~~~~~~~ i=n\\
\end{cases}\\ 
\vspace{.2cm}
t_n(m \otimes f^0 \otimes \ldots \otimes f^{n})&=m_{(0)} \otimes S^{-1}(m_{(-1)})f^n \otimes f^0 \otimes \ldots \otimes f^{n-1}
\end{align*}
for any $m \in M$ and $ f^0 \otimes f^1 \otimes \ldots \otimes f^n \in Hom_{\mathcal{D}_H}(X_1,X_0) \otimes Hom_{\mathcal{D}_H}(X_2,X_1) \otimes \ldots \otimes Hom_{\mathcal{D}_H}(X_0,X_n)$.

\smallskip
(2) By passing to the tensor product over $H$,   we obtain a cyclic module $C_\bullet^H(\mathcal{D}_H,M):=\{C_n^H(\mathcal{D}_H,M)=M \otimes_H CN_n(\mathcal{D}_H)\}_{n \geq 0}$. 
\end{prop}
\begin{proof}
The proof of (1) follows as in Proposition \ref{prop2.3} (1). To prove (2), we first verify that the cyclic operator $t_n$ is well-defined. For any $m \otimes_H f^0 \otimes \ldots \otimes f^n \in C_n^H(\mathcal{D}_H,M)$, we have
\begin{equation*}
\begin{array}{lll}
t_n\left(mh \otimes_H f^0 \otimes \ldots \otimes f^{n}\right)&=(mh)_{(0)} \otimes_H S^{-1}((mh)_{(-1)})f^n \otimes f^0 \otimes \ldots \otimes f^{n-1}&\\
&=m_{(0)}h_2 \otimes_H S^{-1}\left(S(h_3)m_{(-1)}h_1\right) f^n \otimes f^0 \otimes \ldots \otimes f^{n-1}& (\text{by } \eqref{SAYDcondi})\\
&=m_{(0)}h_2 \otimes_H S^{-1}(h_1)S^{-1}(m_{(-1)})h_3f^n \otimes f^0 \otimes \ldots \otimes f^{n-1}&\\
&=m_{(0)}\otimes_H h_2 S^{-1}(h_1)S^{-1}(m_{(-1)})h_{n+3}f^n \otimes h_3f^0 \otimes \ldots \otimes h_{n+2}f^{n-1}&\\
&=m_{(0)}\otimes_H \varepsilon(h_1)S^{-1}(m_{(-1)})h_{n+2}f^n \otimes h_2f^0 \otimes \ldots \otimes h_{n+1}f^{n-1}&\\
&=m_{(0)}\otimes_H S^{-1}(m_{(-1)})h_{n+1}f^n \otimes h_1f^0 \otimes \ldots \otimes h_{n}f^{n-1}&\\
&=t_n\left(m \otimes_H h(f^0 \otimes \ldots \otimes f^{n})\right)
\end{array}
\end{equation*}
It may be verified easily that the face maps and the degeneracies are also well-defined. Moreover,
\begin{equation*}
\begin{array}{lll}
t_n^{n+1}(m \otimes_H f^0 \otimes \ldots \otimes f^{n})&=m_{(0)}\otimes_H S^{-1}(m_{(-1)})(f^0 \otimes \ldots \otimes f^{n})=m_{(0)}S^{-1}(m_{(-1)}) \otimes_H f^0 \otimes \ldots \otimes f^{n}\\
&=m \otimes_H f^0 \otimes \ldots \otimes f^{n} \quad (\text{by } \eqref{2obsv})
\end{array}
\end{equation*}
\end{proof}

The cyclic homology groups corresponding to the cyclic module $C_\bullet^H(\mathcal{D}_H,M)$ will be denoted by $HC^H_\bullet(\mathcal{D}_H,M)$. 

\smallskip
We fix $r\geq 1$. For $1\leq i,j\leq r$ and $\alpha\in k$, we let $E_{ij}(\alpha)$ denote the elementary matrix in $M_r(k)$ having $\alpha$
at $(i,j)$-th position and $0$ everywhere else. We will often use $E_{ij}$ for $E_{ij}(1)$. Let $\mathcal{D}_H$ be a left $H$-category. For each $1\leq p\leq r$, we have an inclusion $inc_p:\mathcal{D}_H \longrightarrow \mathcal{D}_H \otimes M_r(k)$ in $\overline{Cat}_H$ given by
\begin{equation*}
inc_p(X)=X \qquad inc_p(f)=f\otimes E_{pp}=f \otimes E_{pp}(1) 
\end{equation*}

For any right-left SAYD-module $M$, the inclusion $inc_p:\mathcal{D}_H \longrightarrow \mathcal{D}_H  \otimes M_r(k)$ induces an inclusion map 
\begin{equation*}
({inc}_p,M):M \otimes  CN_n(\mathcal{D}_H)  \longrightarrow M \otimes CN_n\left(\mathcal{D}_H  \otimes M_r(k)\right) 
\end{equation*}
\begin{equation*}
m \otimes f^0 \otimes \ldots \otimes f^n \mapsto m \otimes (f^0 \otimes E_{pp}) \otimes \ldots \otimes (f^n \otimes E_{pp})
\end{equation*} This induces a morphism of Hochschild complexes
\begin{equation}\label{hocrq}
C_\bullet(inc_p,M)^{hoc}:C_\bullet(\mathcal{D}_H,M)^{hoc} \longrightarrow C_\bullet\left(\mathcal{D}_H \otimes M_r(k),M\right)^{hoc}
\end{equation} as well as a morphism of double complexes computing cyclic homology
\begin{equation}\label{cyrq}
C_{\bullet\bullet}(inc_p,M)^{cy}:C_{\bullet\bullet}(\mathcal{D}_H,M)^{cy} \longrightarrow C_{\bullet\bullet}\left(\mathcal{D}_H \otimes M_r(k),M\right)^{cy}
\end{equation}

\begin{definition}
Let $\mathcal{D}_H$ be a left $H$-category, $M$ be a right-left SAYD module over $H$ and let $r \in \mathbb{N}$. Then, for each $n \geq 0$, we define an $H$-linear trace map 
\begin{equation*}
tr^M:M \otimes CN_n\left(\mathcal{D}_H \otimes M_r(k)\right) \longrightarrow M \otimes CN_n(\mathcal{D}_H)  
\end{equation*}
\begin{equation}\label{tracemap}
tr^M\left(m \otimes (f^0 \otimes B^0) \otimes \ldots \otimes (f^n \otimes B^n)\right):=(m \otimes f^0 \otimes \ldots \otimes f^n)\text{trace}(B^0\ldots B^n)
\end{equation}
for any $m \in M$ and $(f^0 \otimes B^0) \otimes \ldots \otimes (f^n \otimes B^n) \in CN_n\left(\mathcal{D}_H \otimes M_r(k)\right)$.
\end{definition}

\begin{lemma}
Let $\mathcal{D}_H$ be a left $H$-category, $M$ be a right-left SAYD module and let $r \in \mathbb{N}$. Then, the trace map as in \eqref{tracemap} defines a morphism $C_\bullet(tr^M):C_\bullet\left(\mathcal{D}_H \otimes M_r(k),M\right) \longrightarrow C_\bullet(\mathcal{D}_H,M)$ of para-cyclic modules.
In particular, we have an induced morphism between underlying Hochschild complexes
\begin{equation*}
C_\bullet({tr^M})^{hoc}:C_\bullet\left(\mathcal{D}_H \otimes M_r(k),M\right)^{hoc} \longrightarrow  C_\bullet(\mathcal{D}_H,M)^{hoc}
\end{equation*}
\end{lemma}
\begin{proof}
It may be verified easily  that the trace map $tr^M$ commutes with the face maps, the cyclic operator and the degeneracies. 
\end{proof}

\begin{prop}\label{homotopy}
The maps $C_\bullet({inc}_1,M)^{hoc}$ and $C_\bullet(tr^M)^{hoc}$ are homotopy inverses of each other.
\end{prop}
\begin{proof}
We first verify that $C_\bullet(tr^M)^{hoc} \circ C_\bullet(inc_1,M)^{hoc}=id$. For any $m \otimes f^0 \otimes \ldots \otimes f^{n} \in M \otimes Hom_{\mathcal{D}_H}(X_1,X_0) \otimes Hom_{\mathcal{D}_H}(X_2,X_1) \otimes \ldots \otimes Hom_{\mathcal{D}_H}(X_0,X_n)$, we have
\begin{equation*}
\begin{array}{ll}
(C_\bullet(tr^M)^{hoc} \circ C_\bullet(inc_1,M)^{hoc})(m \otimes f^0 \otimes \ldots \otimes f^{n})&=tr^M\left(m \otimes (f^0 \otimes E_{11}) \otimes \ldots \otimes (f^n \otimes E_{11})\right)\\
&=(m \otimes f^0 \otimes \ldots \otimes f^{n})trace(E_{11}\ldots E_{11})\\
&=m \otimes f^0 \otimes \ldots \otimes f^{n}
\end{array}
\end{equation*}
Therefore, it remains to show that $C_\bullet(inc_1,M)^{hoc} \circ C_\bullet(tr^M)^{hoc} \sim id$. 
We define  $k$-linear maps
$\{\hbar_i: C_{n}\left(\mathcal{D}_H \otimes M_r(k),M\right) \longrightarrow C_{n+1}\left(\mathcal{D}_H \otimes M_r(k),M\right)\}_{0\leq i\leq n}$  by setting:
\begin{equation*}
\begin{array}{ll}
\hbar_i\left(m \otimes (f^0 \otimes B^0) \otimes \ldots \otimes (f^n \otimes B^n)\right):=& m \otimes \sum_{1 \leq j,k,l,\ldots,p,q\leq r}  (f^0 \otimes E_{j1}(B^0_{jk})) \otimes (f^1 \otimes E_{11}(B^1_{kl})) \otimes \ldots\\ & \quad  \otimes (f^i \otimes E_{11}(B^i_{pq})) \otimes (id_{X_{i+1}} \otimes E_{1q}(1)) \otimes (f^{i+1} \otimes B^{i+1}) \otimes \ldots\\
& \quad \ldots \otimes (f^n \otimes B^n)
\end{array}
\end{equation*}
for $0\leq i<n$ and 
\begin{equation*}
\begin{array}{ll}
\hbar_n\left(m \otimes (f^0 \otimes B^0) \otimes \ldots \otimes (f^n \otimes B^n)\right):= & m \otimes \sum_{1 \leq j,k,m,\ldots,p,q \leq r}  (f^0 \otimes E_{j1}(B^0_{jk})) \otimes (f^1 \otimes E_{11}(B^1_{km})) \otimes \ldots\\ & \quad \ldots \otimes (f^n \otimes E_{11}(B^n_{pq})) \otimes (id_{X_0} \otimes E_{1q}(1))
\end{array}
\end{equation*}
We now verify that $\hbar^n:=\sum_{i=0}^n (-1)^i\hbar_i$ is a pre-simplicial homotopy (see, for instance, \cite[$\S$ 1.0.8]{Loday}) between $C_\bullet(inc_1,M)^{hoc} \circ C_\bullet(tr^M)^{hoc}$ and $id_{C_\bullet\left(\mathcal{D}_H \otimes M_r(k),M\right)}$. For this, we need to verify the following identities:
\begin{equation}\label{relations}
\begin{array}{lll}
d_i\hbar_{i'}=\hbar_{i'-1}d_i & \text{for}~ i<i'\\
d_i\hbar_i=d_i\hbar_{i-1} & \text{for}~ 0< i \leq n\\
d_i\hbar_{i'}= \hbar_{i'}d_{i-1} & \text{for}~ i>i'+1\\
d_0\hbar_0= id_{C_\bullet\left(\mathcal{D}_H \otimes M_r(k),M\right)^{hoc}}& \text{and}~ d_{n+1}\hbar_n=C_\bullet(inc_1,M)^{hoc} \circ C_\bullet(tr^M)^{hoc}
\end{array}
\end{equation}
where $d_i:C_{n+1}\left(\mathcal{D}_H \otimes M_r(k),M\right) \longrightarrow C_{n}\left(\mathcal{D}_H \otimes M_r(k),M\right)$, $0 \leq i \leq n+1$ are the face maps.  

\smallskip
For $0<i <i'$, we have
\begin{equation*}
\begin{array}{ll}
&d_i\hbar_{i'}\left(m \otimes (f^0 \otimes B^0) \otimes \ldots \otimes (f^n \otimes B^n)\right)\\
& \quad =d_i\big(m \otimes \sum_{1 \leq j,k,l,\ldots,p,q \leq r}  (f^0 \otimes E_{j1}(B^0_{jk})) \otimes (f^1 \otimes E_{11}(B^1_{kl})) \otimes \ldots \otimes (f^{i'} \otimes E_{11}(B^{i'}_{pq})) \otimes \\
& \qquad (id_{X_{i'+1}} \otimes E_{1q}(1)) \otimes (f^{i'+1} \otimes B^{i'+1}) \otimes \ldots  \otimes (f^n \otimes B^n)\big)\\
& \quad = m \otimes \sum_{1 \leq l,k,\ldots,u,v_i,w,\ldots,p,q\leq r}  (f^0 \otimes E_{l1}(B^0_{lk})) \otimes \ldots \otimes (f^i \otimes E_{11}(B^i_{uv_i}))(f^{i+1} \otimes E_{11}(B^{i+1}_{v_iw})) \otimes \ldots\\
& \qquad \otimes (f^{i'} \otimes E_{11}(B^{i'}_{pq})) \otimes (id_{X_{i'+1}} \otimes E_{1q}(1)) \otimes (f^{i'+1} \otimes B^{i'+1}) \otimes \ldots  \otimes (f^n \otimes B^n)\\
& \quad = m \otimes \sum_{1 \leq l,k,\ldots,u,w, \ldots,p,q \leq r}  (f^0 \otimes E_{l1}(B^0_{lk})) \otimes \ldots \otimes (f^if^{i+1} \otimes E_{11}\left((B^iB^{i+1})_{uw}\right)) \otimes \ldots \\
&  \qquad \otimes (f^{i'} \otimes E_{11}(B^{i'}_{pq})) \otimes (id_{X_{i'+1}} \otimes E_{1q}(1)) \otimes (f^{i'+1} \otimes B^{i'+1}) \otimes \ldots  \otimes (f^n \otimes B^n)\\
& \quad= \hbar_{i'-1}d_i\left(m \otimes (f^0 \otimes B^0) \otimes \ldots \otimes (f^n \otimes B^n)\right)
\end{array}
\end{equation*}
Moreover,
\begin{equation*}
\begin{array}{ll}
&d_0\hbar_{i'}\left(m \otimes (f^0 \otimes B^0) \otimes \ldots \otimes (f^n \otimes B^n)\right)\\
& \quad =d_0\big(m \otimes \sum_{1 \leq j,k,l,\ldots,p,q \leq r}  (f^0 \otimes E_{j1}(B^0_{jk})) \otimes (f^1 \otimes E_{11}(B^1_{kl})) \otimes \ldots \otimes (f^{i'} \otimes E_{11}(B^{i'}_{pq})) \otimes \\
& \qquad (id_{X_{i'+1}} \otimes E_{1q}(1)) \otimes (f^{i'+1} \otimes B^{i'+1}) \otimes \ldots  \otimes (f^n \otimes B^n)\big)\\
& \quad =m \otimes \sum_{1 \leq j,k,l,\ldots,p,q \leq r}  \left(f^0f^1 \otimes E_{j1}(B^0_{jk})E_{11}(B^1_{kl})\right) \otimes  \ldots \otimes (f^{i'} \otimes E_{11}(B^{i'}_{pq})) \otimes \\
& \qquad (id_{X_{i'+1}} \otimes E_{1q}(1)) \otimes (f^{i'+1} \otimes B^{i'+1}) \otimes \ldots  \otimes (f^n \otimes B^n)\\
& \quad =m \otimes \sum_{1 \leq j,k,l,\ldots,p,q \leq r}  \left(f^0f^1 \otimes E_{j1}(B^0_{jk}B^1_{kl})\right) \otimes  \ldots \otimes (f^{i'} \otimes E_{11}(B^{i'}_{pq})) \otimes \\
& \qquad (id_{X_{i'+1}} \otimes E_{1q}(1)) \otimes (f^{i'+1} \otimes B^{i'+1}) \otimes \ldots  \otimes (f^n \otimes B^n)\\
& \quad =m \otimes \sum_{1 \leq j,k,\ldots,p,q \leq r}  \left(f^0f^1 \otimes E_{j1}((B^0B^1)_{jk})\right) \otimes  \ldots \otimes (f^{i'}\otimes E_{11}(B^{i'}_{pq})) \otimes \\
& \qquad (id_{X_{i'+1}} \otimes E_{1q}(1)) \otimes (f^{i'+1} \otimes B^{i'+1}) \otimes \ldots  \otimes (f^n \otimes B^n)\big)\\
& \quad= \hbar_{i'-1}d_0\left(m \otimes (f^0 \otimes B^0) \otimes \ldots \otimes (f^n \otimes B^n)\right)
\end{array}
\end{equation*}
Next, using the equality $\sum_{l=1}^rE_{11}(B_{kl})E_{1l}(1)=E_{1k}(1)B$ for all $B\in M_r(k)$, $1 \leq k \leq r$, we have for $0 <i <n$:
\begin{equation*}
\begin{array}{ll}
&d_i\hbar_i\left(m \otimes (f^0 \otimes B^0) \otimes \ldots \otimes (f^n \otimes B^n)\right)\\
&\quad =d_i\big(m \otimes \sum_{1 \leq j,k,l,\ldots,p,q \leq r}  (f^0 \otimes E_{j1}(B^0_{jk})) \otimes (f^1 \otimes E_{11}(B^1_{kl})) \otimes \ldots \otimes (f^i \otimes E_{11}(B^i_{pq})) \otimes\\
& \qquad  (id_{X_{i+1}} \otimes E_{1q}(1)) \otimes (f^{i+1} \otimes B^{i+1}) \otimes \ldots  \otimes (f^n \otimes B^n)\big)\\
& \quad =m \otimes \sum_{1 \leq j,k,l,\ldots,p,q \leq r}  (f^0 \otimes E_{j1}(B^0_{jk})) \otimes (f^1 \otimes E_{11}(B^1_{kl})) \otimes \ldots \otimes (f^i \otimes E_{11}(B^i_{pq})E_{1q}(1)) \otimes\\
& \qquad \ldots \otimes (f^n \otimes B^n)\\
& \quad =m \otimes \sum_{1 \leq j,k,l,\ldots,p \leq r}  (f^0 \otimes E_{j1}(B^0_{jk})) \otimes (f^1 \otimes E_{11}(B^1_{kl})) \otimes \ldots \otimes (f^i \otimes \sum_{q=1}^r E_{11}(B^i_{pq})E_{1q}(1)) \otimes\\
& \qquad \ldots \otimes (f^n \otimes B^n)\\
& \quad =m \otimes \sum_{1 \leq j,k,l,\ldots,u,p \leq r}  (f^0 \otimes E_{j1}(B^0_{jk})) \otimes (f^1 \otimes E_{11}(B^1_{kl})) \otimes \ldots \otimes (f^{i-1} \otimes E_{11}(B^{i-1}_{up})) \otimes \\
& \qquad (f^i \otimes E_{1p}(1)B^i) \otimes \ldots \otimes (f^n \otimes B^n)\\
& \quad =  d_i\big(m \otimes \sum_{1 \leq j,k,l,\ldots,u,p \leq r}  (f^0 \otimes E_{j1}(B^0_{jk})) \otimes (f^1 \otimes E_{11}(B^1_{kl})) \otimes \ldots \otimes (f^{i-1} \otimes E_{11}(B^{i-1}_{up})) \otimes \\
& \qquad (id_{X_i} \otimes E_{1p}(1)) \otimes (f^i \otimes B^i) \otimes \ldots \otimes (f^n \otimes B^n)\big)\\
& \quad = d_i\hbar_{i-1}\left(m \otimes (f^0 \otimes B^0) \otimes \ldots \otimes (f^n \otimes B^n)\right)
\end{array}
\end{equation*}
The case $i=n$ follows similarly. For $i>i'+1$, we have
\begin{equation*}
\begin{array}{ll}
&d_i\hbar_{i'}\left(m \otimes (f^0 \otimes B^0) \otimes \ldots \otimes (f^n \otimes B^n)\right)\\
& \quad =d_i\big(m \otimes \sum_{1 \leq j,k\ldots,p,q \leq r}  (f^0 \otimes E_{j1}(B^0_{jk})) \otimes  \ldots \otimes (f^{i'} \otimes E_{11}(B^{i'}_{pq})) \otimes (id_{X_{i'+1}} \otimes E_{1q}(1))\\
& \qquad  \otimes (f^{i'+1} \otimes B^{i'+1}) \otimes \ldots  \otimes (f^n \otimes B^n)\big)\\
& \quad = m \otimes \sum_{1 \leq j,k\ldots,p,q \leq r}  (f^0 \otimes E_{j1}(B^0_{jk})) \otimes  \ldots \otimes (f^{i'} \otimes E_{11}(B^{i'}_{pq})) \otimes (id_{X_{i'+1}} \otimes E_{1q}(1))\\
& \qquad  \otimes (f^{i'+1} \otimes B^{i'+1}) \otimes \ldots \otimes (f^{i-1}f^i \otimes B^{i-1}B^i) \otimes \ldots  \otimes (f^n \otimes B^n)\\
&\quad = \hbar_{i'}\left(m \otimes (f^0 \otimes B^0) \otimes \ldots \otimes (f^{i-1}f^i \otimes B^{i-1}B^i) \otimes \ldots \otimes (f^n \otimes B^n)\right)\\
& \quad = \hbar_{i'}d_{i-1}\left(m \otimes (f^0 \otimes B^0) \otimes \ldots \otimes (f^n \otimes B^n)\right)
\end{array}
\end{equation*}
Further, using the equality $\sum_{1 \leq j,k \leq r}  E_{j1}(B_{jk})E_{1k}(1)=B$, we have
\begin{equation*}
\begin{array}{ll}
&d_0\hbar_0\left(m \otimes (f^0 \otimes B^0) \otimes \ldots \otimes (f^n \otimes B^n)\right)\\
& \quad =d_0\big(m \otimes  \sum_{1 \leq j,k \leq r}  (f^0 \otimes E_{j1}(B^0_{jk})) \otimes (id_{X_1} \otimes E_{1k}(1)) \otimes (f^1 \otimes B^1) \otimes \ldots \otimes (f^n \otimes B^n)\big)\\
& \quad =m \otimes  \sum_{1 \leq j,k \leq r}  (f^0 \otimes E_{j1}(B^0_{jk})E_{1k}(1)) \otimes (f^1 \otimes B^1) \otimes \ldots \otimes (f^n \otimes B^n)\\
&\quad= m \otimes (f^0 \otimes B^0) \otimes \ldots \otimes (f^n \otimes B^n)
\end{array}
\end{equation*}
Finally, using the fact that $E_{1q}(1)E_{j1}(B_{jk})=0$ unless $q=j$, we have 
\begin{equation*}
\begin{array}{ll}
&d_{n+1}\hbar_n\left(m \otimes (f^0 \otimes B^0) \otimes \ldots \otimes (f^n \otimes B^n)\right)\\
& \quad =d_{n+1}\big(m \otimes \sum_{1 \leq j,k,l,\ldots,p,q \leq r}  (f^0 \otimes E_{j1}(B^0_{jk})) \otimes (f^1 \otimes E_{11}(B^1_{kl})) \otimes \ldots \otimes (f^n \otimes E_{11}(B^n_{pq})) \otimes 
 (id_{X_{0}} \otimes E_{1q}(1))\big)\\
& \quad =m_{(0)} \otimes \sum_{1 \leq j,k,l,\ldots,p,q \leq r} \left(S^{-1}(m_{(-1)})(id_{X_{0}} \otimes E_{1q}(1))\right)(f^0 \otimes E_{j1}(B^0_{jk}))\otimes (f^1 \otimes E_{11}(B^1_{kl})) \otimes \ldots \\
&\qquad \ldots \otimes (f^n \otimes E_{11}(B^n_{pq}))\\
& \quad= m \otimes \sum_{1 \leq j,k,l,\ldots,p,q \leq r}\left(f^0 \otimes E_{1q}(1)E_{j1}(B^0_{jk})\right) \otimes (f^1 \otimes E_{11}(B^1_{kl})) \otimes \ldots \otimes (f^n \otimes E_{11}(B^n_{pq}))\\
& \quad= m \otimes \sum_{1 \leq j,k,l,\ldots,p\leq r}\left(f^0 \otimes E_{1j}(1)E_{j1}(B^0_{jk})\right) \otimes (f^1 \otimes E_{11}(B^1_{kl})) \otimes \ldots \otimes (f^n \otimes E_{11}(B^n_{pj}))\\
& \quad= m \otimes \sum_{1 \leq j,k,l,\ldots,p\leq r}(f^0 \otimes E_{11}(B^0_{jk})) \otimes (f^1 \otimes E_{11}(B^1_{kl})) \otimes \ldots \otimes (f^n \otimes E_{11}(B^n_{pj}))\\
& \quad = \left(m \otimes (f^0 \otimes E_{11}) \otimes \ldots \otimes (f^n \otimes E_{11})\right) \sum_{1 \leq j,k,l,\ldots,p\leq r}(B^0_{jk}B^1_{kl}\ldots B^n_{pj})\\
& \quad = \left(m \otimes (f^0 \otimes E_{11}) \otimes \ldots \otimes (f^n \otimes E_{11})\right) \sum_{1 \leq j \leq r} 
(B^0B^1 \ldots B^n)_{jj}\\
& \quad = \left(m \otimes (f^0 \otimes E_{11}) \otimes \ldots \otimes (f^n \otimes E_{11})\right)trace(B^0B^1 \ldots B^n)\\
& \quad = \left(C_\bullet(inc_1,M)^{hoc} \circ C_\bullet(tr^M)^{hoc}\right)\left(m \otimes (f^0 \otimes B^0) \otimes \ldots \otimes (f^n \otimes B^n)\right)
\end{array}
\end{equation*} This proves the result.
\end{proof}

We denote by $Vect_k$ the category of all $k$-vector spaces and by $H\text{-}Mod$ the category of all left $H$-modules. Let
$Hom_H(-,k):H\text{-}Mod \longrightarrow Vect_k$ be the functor that takes $N \mapsto Hom_H(N,k)$. 

\begin{prop}\label{moritainvcyc} (1)  Applying the functor $Hom_H(-,k)$, we obtain morphisms of Hochschild complexes
\begin{equation*}
C^{\bullet}_H(inc_1,M)^{hoc}:C^{\bullet}_H\left(\mathcal{D}_H \otimes M_r(k),M\right)^{hoc} \longrightarrow C^{\bullet}_H(\mathcal{D}_H,M)^{hoc}
\end{equation*} as well as a morphism of double complexes computing cyclic cohomology
\begin{equation*}
C^{\bullet\bullet}_H(inc_1,M)^{cy}:C^{\bullet\bullet}_H(\mathcal{D}_H\otimes M_r(k),M)^{cy} \longrightarrow C_H^{\bullet\bullet}\left(\mathcal{D}_H ,M\right)^{cy}
\end{equation*}

(2) Applying the functor $Hom_H(-,k)$, we obtain morphisms of cocyclic modules:
\begin{equation*}
\begin{array}{ll}
C^\bullet_H(tr^M):=Hom_H(C_\bullet(tr^M),k): C^\bullet_H(\mathcal{D}_H,M) \longrightarrow C^\bullet_H\left(\mathcal{D}_H \otimes M_r(k),M\right) 
\end{array}
\end{equation*}
(3) The morphisms
\begin{equation*}
\begin{array}{ll}
HC^{\bullet}_H(inc_1,M)^{hoc}:HC^{\bullet}_H\left(\mathcal{D}_H \otimes M_r(k),M\right)^{hoc} \longrightarrow HC^{\bullet}_H(\mathcal{D}_H,M)^{hoc}\\
HC^{\bullet}_H(tr^M)^{hoc}:HC^{\bullet}_H(\mathcal{D}_H,M)^{hoc} \longrightarrow HC^{\bullet}_H\left(\mathcal{D}_H \otimes M_r(k),M\right)^{hoc}
\end{array}
\end{equation*}
induced by $C^{\bullet}_H(inc_1,M)^{hoc}$ and $C^{\bullet}_H(tr^M)^{hoc}$ are mutually inverse isomorphisms of Hochschild cohomologies.

\smallskip
(4) Applying the functor $Hom_H(-,k)$ gives isomorphisms 
 \begin{equation*} \xy (10,0)*{HC^\bullet_H(\mathcal{D}_H,M)}; (40,3)*{HC^{\bullet}_H(tr^M)}; (40,-3.3)*{HC^{\bullet}_H(inc_1,M)};  {\ar
(25,1)*{}; (60,1)*{}}; {\ar (60,-1)*{};(25,-1)*{}};
\endxy \quad HC^\bullet_H\left(\mathcal{D}_H \otimes M_r(k),M\right)
\end{equation*} of Hopf-cyclic cohomologies.
\end{prop}
\begin{proof}
(1) This follows by applying the functor $Hom_H(-,k)$ to the morphisms in \eqref{hocrq} and \eqref{cyrq}. 

\smallskip
(2) This  follows from Proposition \ref{P6.2wp}(1),  Proposition \ref{prop2.3} and the fact that   $C_\bullet(tr^M):C_\bullet\left(\mathcal{D}_H \otimes M_r(k),M\right) \longrightarrow C_\bullet(\mathcal{D}_H,M)$ is a morphism of para-cyclic modules.

\smallskip
(3) By Proposition \ref{homotopy}, we know that $C_\bullet(tr^M)^{hoc} \circ C_\bullet(inc_1,M)^{hoc}=id_{C_\bullet(\mathcal{D}_H,M)^{hoc}}$ and
$C_\bullet(inc_1,M)^{hoc} \circ C_\bullet(tr^M)^{hoc} \sim id_{C_\bullet\left(\mathcal{D}_H \otimes M_r(k),M\right)^{hoc}}$. Thus, applying the functor $Hom_H(-,k)$, we obtain 
\begin{equation*}
\begin{array}{ll}
C^{\bullet}_H(inc_1,M)^{hoc} \circ C^{\bullet}_H(tr^M)^{hoc}=id_{C^{\bullet}_H(\mathcal{D}_H,M)^{hoc}}\\
C^{\bullet }_H(tr^M)^{hoc} \circ C^{\bullet}_H(inc_1,M)^{hoc} \sim id_{C^{\bullet}_H(\mathcal{D}_H \otimes M_r(k),M)^{hoc}}
\end{array}
\end{equation*}
Therefore, $C^{\bullet}_H(inc_1,M)^{hoc}$ and $C^{\bullet}_H(tr^M)^{hoc}$ are homotopy inverses of each other. 

\smallskip
(4) This follows immediately from (3) and Hochschild to cyclic spectral sequence.
\end{proof}

\begin{corollary}
Given a small $k$-linear category $\mathcal C$, there is an isomorphism $HC^{\bullet}(\mathcal{C})
\overset{\simeq}{\longrightarrow} HC^{\bullet}\left(\mathcal C\otimes M_r(k)\right)$ of cyclic cohomology groups.
\end{corollary}
\begin{proof}
This follows by taking $H=k=M$ in Proposition \ref{moritainvcyc} (4).
\end{proof}

\begin{corollary}\label{5.8a}
For an  $n$-cocycle $\phi \in Z^n_H(\mathcal{D}_H,M)$, the $n$-cocycle $\tilde{\phi}=Hom_H(tr^M,k)(\phi) =\phi\circ tr^M\in Z^n_H(\mathcal{D}_H \otimes M_r(k),M)$  may be described as follows
\begin{equation*}
\tilde{\phi}\left(m \otimes (f^0 \otimes B^0) \otimes \ldots \otimes (f^n \otimes B^n)\right)=\phi(m \otimes f^0 \otimes \ldots \otimes f^n)trace(B^0\ldots B^n)
\end{equation*}
\end{corollary}

\section{Vanishing cycles on an $H$-category and coboundaries}
From now onwards, we will always assume that $k=\mathbb C$. In this section, we will describe the spaces $B^\bullet_H(\mathcal{D}_H,M)$ and $B^\bullet(\mathcal{D}_H)$.  This will be done using the isomorphism $HC^\bullet_H(\mathcal{D}_H,M)\overset{\simeq}{\longrightarrow}HC^\bullet_H\left(\mathcal{D}_H \otimes M_r(k),M\right)$ proved in Section \ref{Morita}. 
Finally, we will use the formalism of categorified cycles and vanishing cycles developed in this paper to obtain a pairing on the cyclic cohomology of a $k$-linear category.

\smallskip We begin by recalling the notion of an inner automorphism of a category.

\begin{definition}\label{indefr} [see \cite{Sch}, p 24]
Let $\mathcal{D}_H$ be a left $H$-category. An automorphism $\Phi \in Hom_{Cat_H}(\mathcal{D}_H,\mathcal{D}_H)$ is said to be inner if $\Phi$ is isomorphic to the identity functor $id_{\mathcal{D}_H}$. In particular, there exist isomorphisms $\{\eta(X):X\longrightarrow \Phi(X)\}_{X\in Ob(\mathcal 
D_H)}$ such that   $\Phi(f)=\eta(Y)\circ f \circ (\eta(X))^{-1}$ for any $f \in Hom_{\mathcal{D}_H}(X,Y)$.
\end{definition}

We now set
\begin{equation}
\mathbb G(\mathcal D_H):=\underset{X\in Ob(\mathcal D_H)}{\prod}\textrm{ }Aut_{\mathcal D_H}(X)
\end{equation}
By definition, an element $\eta\in \mathbb G(\mathcal D_H)$ corresponds to a family of automorphisms $\{\eta(X):X\longrightarrow X\}_{X\in Ob(\mathcal 
D_H)}$. We now set
\begin{equation}
\mathbb U_H(\mathcal D_H):=\{\mbox{$\eta\in \mathbb G(\mathcal D_H)$ $\vert$ $h(\eta(X))=\varepsilon(h)\eta(X)$ for every $h\in H$ and $X\in Ob(\mathcal D_H)$}\}
\end{equation}

\begin{lemma} $\mathbb U_H(\mathcal D_H)$ is a subgroup of $\mathbb G(\mathcal D_H)$. 
\end{lemma}
\begin{proof}
The element $\mathbf{e}=\underset{X\in Ob(\mathcal D_H)}{\prod} id_X$ is the identity of the group $\mathbb G(\mathcal D_H)$.
By definition of an $H$-category, we know that $h\cdot id_X=\varepsilon(h)\cdot id_X$ for each $X \in Ob(\mathcal{D}_H)$ and $h \in H$. Thus, $\mathbf{e} \in \mathbb U_H(\mathcal D_H)$. Now, suppose that $\eta, \eta' \in \mathbb U_H(\mathcal D_H)$. Then, for each $X \in Ob(\mathcal{D}_H)$ and $h \in H$,
\begin{equation*}
h\left((\eta \circ \eta')(X)\right)=h(\eta(X) \circ \eta'(X))=(h_1\eta(X)) \circ (h_2\eta'(X))=(\varepsilon(h_1)\eta(X)) \circ (\varepsilon(h_2)\eta'(X))=\varepsilon(h)(\eta(X) \circ \eta'(X))
\end{equation*}
Hence, $\eta \circ \eta' \in \mathbb U_H(\mathcal D_H)$. 

\smallskip
Now, let $\eta \in \mathbb U_H(\mathcal D_H)$. Then, $\eta^{-1} \in \mathbb G(\mathcal D_H)$ corresponds to a family of morphisms $\{\eta^{-1}(X):=\eta(X)^{-1}:X \longrightarrow X\}_{X \in Ob(\mathcal{D}_H)}$. Then, for each $h \in H$ and $X \in Ob(\mathcal{D}_H)$,
\begin{equation*}
\varepsilon(h)id_X=h(\eta(X) \circ \eta^{-1}(X))=(\varepsilon(h_1)\eta(X)) \circ (h_2 \eta^{-1}(X))=\eta(X) \circ (h \eta^{-1}(X))
\end{equation*}
which gives $\varepsilon(h)\eta^{-1}(X)=h \eta^{-1}(X)$. Therefore, $\eta^{-1} \in  \mathbb U_H(\mathcal D_H)$.
\end{proof}

\begin{lemma}\label{innerautonew}
 Let $\mathcal{D}_H$ be a left $H$-category and let $\eta\in \mathbb U_H(\mathcal D_H)$.

\smallskip
(1) Consider $\Phi_\eta:\mathcal D_H\longrightarrow \mathcal D_H$ defined by 
\begin{equation*}\Phi_\eta(X)=X\qquad \Phi_\eta(f):=\eta(Y)\circ f\circ\eta(X)^{-1}
\end{equation*} for every $X\in Ob(\mathcal D_H)$ and $f\in Hom_{\mathcal D_H}(X,Y)$. Then, $\Phi_\eta: \mathcal D_H\longrightarrow \mathcal D_H$
is an  inner automorphism of $\mathcal{D}_H$.

\smallskip
(2) Consider $\tilde\Phi_\eta: \mathcal D_H\otimes M_2(k)\longrightarrow \mathcal D_H\otimes M_2(k)$ defined by
\begin{equation*}
\tilde\Phi_\eta(X)=X \qquad \tilde\Phi_\eta(f\otimes B)=(id_Y\otimes E_{11}+  \eta(Y)\otimes E_{22})\circ (f\otimes B)\circ (id_X\otimes E_{11}+  \eta(X)^{-1}\otimes E_{22})
\end{equation*}  for every $X\in Ob(\mathcal D_H\otimes M_2(k))=Ob(\mathcal D_H)$ and $f\otimes B \in Hom_{\mathcal D_H\otimes M_2(k)}(X,Y)$. Then,
$\tilde\Phi_\eta: \mathcal D_H\otimes M_2(k)\longrightarrow \mathcal D_H\otimes M_2(k)$ is an  inner automorphism.
\end{lemma}
\begin{proof}
{\it (1)}  Using the fact that $\eta, \eta^{-1} \in \mathbb U_H(\mathcal D_H)$, we have
\begin{equation*}
\begin{array}{ll}
h(\Phi_\eta(f))=(h_1\eta(Y)) \circ (h_2f) \circ (h_3\eta(X)^{-1})&=(\varepsilon(h_1)\eta(Y)) \circ (h_2f) \circ (\varepsilon(h_3)\eta(X)^{-1})\\
&=\eta(Y)  \circ (h_1f) \circ (\varepsilon(h_2)\eta(X)^{-1})\\
&=\eta(Y)\circ (hf)\circ\eta(X)^{-1}
\end{array}
\end{equation*}
for any $h \in H$ and $f\in Hom_{\mathcal D_H}(X,Y)$. By Definition \ref{indefr}, we now see that $\Phi_\eta$ is an inner automorphism.

\smallskip
{\it (2)}  Setting $\tilde{\eta}(X):X \longrightarrow X$ in $\mathcal D_H\otimes M_2(k)$ as $\tilde{\eta}(X)=id_X\otimes E_{11}+  \eta(X)\otimes E_{22}$, we see that
\begin{equation*}
\begin{array}{ll}
\tilde\Phi_\eta(f\otimes B)&=(id_Y\otimes E_{11}+  \eta(Y)\otimes E_{22})\circ (f\otimes B)\circ (id_X\otimes E_{11}+  \eta(X)^{-1}\otimes E_{22})\\
&=\tilde{\eta}(Y) \circ (f\otimes B)\circ \tilde{\eta}(X)^{-1}
\end{array}
\end{equation*}
for any  $f\otimes B \in Hom_{\mathcal D_H\otimes M_2(k)}(X,Y)$.
Considering the $H$-action on the category $\mathcal D_H\otimes M_2(k)$, we have
\begin{equation*}
\begin{array}{ll}
h\left(\tilde{\Phi}_\eta((f \otimes B)\right)&=h_1(id_Y\otimes E_{11}+  \eta(Y)\otimes E_{22})\circ h_2(f\otimes B)\circ h_3(id_X\otimes E_{11}+  \eta(X)^{-1}\otimes E_{22})\\
&=(h_1id_Y\otimes E_{11}+  h_1\eta(Y)\otimes E_{22})\circ h_2(f\otimes B)\circ (h_3id_X\otimes E_{11}+  h_3\eta(X)^{-1}\otimes E_{22})\\
&=\varepsilon(h_1)(id_Y\otimes E_{11}+  \eta(Y)\otimes E_{22})\circ h_2(f\otimes B)\circ \varepsilon(h_3)(id_X\otimes E_{11}+  \eta(X)^{-1}\otimes E_{22})\\
&=(id_Y\otimes E_{11}+  \eta(Y)\otimes E_{22})\circ h(f\otimes B)\circ (id_X\otimes E_{11}+  \eta(X)^{-1}\otimes E_{22})\\
&=\tilde{\Phi}_\eta(h(f \otimes B))
\end{array}
\end{equation*}
for any $h \in H$ and $f\otimes B \in Hom_{\mathcal D_H\otimes M_2(k)}(X,Y)$. By Definition \ref{indefr}, we now see that $\Phi_{\tilde\eta}$ is an inner automorphism.
\end{proof}

For $\eta\in \mathbb U_H(\mathcal D_H)$, we will always denote by $\Phi_\eta$ and $\tilde{\Phi}_\eta$ the inner automorphisms defined in Lemma \ref{innerautonew}.

\begin{lemma}\label{6.3a} Let $M$ be a right-left SAYD module over $H$. Then, 

\smallskip
(1) A semifunctor $\alpha \in Hom_{\overline{Cat}_H}(\mathcal{D}_H,\mathcal{D}_H')$ induces a morphism (for all $n\geq 0$)
\begin{equation*}
C^n_H(\alpha,M):C^n_H(\mathcal{D}_H',M)=Hom_H(M \otimes CN_n(\mathcal{D}_H'),k)\longrightarrow C^n_H(\mathcal{D}_H,M)=Hom_H(M \otimes CN_n(\mathcal{D}_H),k)
\end{equation*}  determined by
\begin{equation*}
C^{n}_H(\alpha,M)(\phi)(m \otimes f^0 \otimes \ldots \otimes f^n)=\phi\left(m \otimes \alpha(f^0) \otimes \ldots \otimes \alpha(f^n)\right)
\end{equation*}
for any $\phi \in C^n_H(\mathcal{D}_H',M)$, $m \in M$ and $f^0 \otimes \ldots \otimes f^n \in CN_n(\mathcal{D}_H)$. This leads to a morphism  $C^{\bullet\bullet}_H(\alpha,M)^{cy}:C^{\bullet\bullet}_H(\mathcal{D}_H',M)^{cy} \longrightarrow C^{\bullet\bullet}_H(\mathcal{D}_H,M)^{cy}$ of double complexes and induces  a functor $HC^\bullet_H(-,M):\overline{Cat}_H^{op} \longrightarrow Vect_k$.

\smallskip
(2) Let $\eta\in \mathbb U_H(\mathcal D_H)$. Then, $\Phi_\eta$  induces the identity map on $HC^\bullet_H(\mathcal{D}_H,M)$.
\end{lemma}
\begin{proof}
(1) Since $\phi$ and $\alpha$ are $H$-linear, the morphisms $C^n_H(\alpha,M)$ are well-defined and well behaved with respect to the maps appearing in the Hochschild
and cyclic complexes. The result follows.

\smallskip
(2) Let $\eta\in \mathbb U_H(\mathcal D_H)$ and $\Phi_\eta \in Hom_{Cat_H}(\mathcal{D}_H,\mathcal{D}_H)$ be the corresponding inner automorphism. By Proposition \ref{moritainvcyc}, the maps $HC^\bullet_H(inc_1,M)$ and $HC^\bullet_H(tr^M)$ are mutually inverse isomorphisms of Hopf-cyclic cohomology groups. Thus, we have
\begin{equation}\label{dd6.1}
HC^\bullet_H(inc_2,M) \circ \left(HC^\bullet_H(inc_1,M)\right)^{-1} =HC^\bullet_H(inc_2,M)  \circ HC^\bullet_H(tr^M)= HC^\bullet_H\left(tr^M \circ (inc_2,M)\right)=id
\end{equation} 
Further, we have the following commutative diagram in the category $\overline{Cat}_H$:
\begin{equation}\label{d6.1}
\begin{CD}
\mathcal{D}_H @>inc_1>>  \mathcal{D}_H \otimes M_2(k) @<inc_2<< \mathcal{D}_H\\
@Vid_{\mathcal{D}_H}VV        @VV \tilde\Phi_\eta V @VV \Phi_\eta V\\
\mathcal{D}_H    @> inc_1>>\mathcal{D}_H \otimes M_2(k)@<inc_2<< \mathcal{D}_H\\
\end{CD}
\end{equation}
Thus, by applying the functor $HC^\bullet_H(-,M)$ to the commutative diagram \eqref{d6.1} and using \eqref{dd6.1}, we obtain
\begin{align*}
HC^\bullet_H(\Phi_\eta,M)&=  \left(HC^\bullet_H(inc_2,M)\right)\circ HC^\bullet_H(inc_1,M)^{-1} \circ HC^\bullet_H(id_{\mathcal{D}_H},M)\circ \left(HC^\bullet_H(inc_1,M)\right) \circ HC^\bullet_H(inc_2,M)^{-1}\\
&=id_{HC^\bullet_H(\mathcal{D}_H,M)}
\end{align*} 
\end{proof}

\begin{prop}\label{vanishing}
Let $\mathcal{D}_H$ be a left $H$-category. Suppose that there is a semifunctor $\upsilon \in Hom_{\overline{Cat}_H}(\mathcal{D}_H,\mathcal{D}_H)$ and an $\eta \in \mathbb U_H\left(\mathcal{D}_H \otimes M_2(k)\right)$ such that 
\begin{itemize}
\item[(1)] $\upsilon(X)=X \qquad \forall X \in Ob(\mathcal{D}_H)$
\item[(2)] $\Phi_\eta(f \otimes E_{11}+\upsilon(f) \otimes E_{22})=\upsilon(f) \otimes E_{22}$
\end{itemize}
for all $f \in Hom_{\mathcal{D}_H}(X,Y)$ and  $X,Y \in Ob(\mathcal{D}_H)$. Then, $HC^\bullet_H(\mathcal{D}_H,M)=0$.
\end{prop}
\begin{proof}
Let $\alpha, \alpha' \in Hom_{\overline{Cat}_H}\left(\mathcal{D}_H, \mathcal{D}_H \otimes M_2(k)\right)$ be the semifunctors defined by
\begin{equation*}
\begin{array}{c}
\alpha(X):=X \qquad \alpha(f):=f \otimes E_{11} + \upsilon(f) \otimes E_{22}\\
\alpha'(X):=X \qquad \alpha'(f):=\upsilon(f) \otimes E_{22}
\end{array}
\end{equation*}
for all $X \in Ob(\mathcal{D}_H)$ and $f \in Hom_ {\mathcal{D}_H}(X,Y)$. Then, by assumption, $\alpha'=\Phi_\eta \circ \alpha$. Therefore, applying the functor $HC^\bullet_H(-,M)$ and using Lemma \ref{6.3a} (2), we get
\begin{equation}\label{cohomclass}
HC^\bullet_H(\alpha',M)=HC^\bullet_H(\alpha,M) \circ HC^\bullet_H(\Phi_\eta,M)=HC^\bullet_H(\alpha,M):HC^\bullet_H(\mathcal D_H\otimes M_2(k),M)\longrightarrow HC^\bullet_H(\mathcal D_H,M)
\end{equation}
Let $\phi \in Z^n_H(\mathcal{D}_H,M)$ and $\tilde{\phi}=Hom_H(tr^M,k)(\phi) =\phi\circ tr^M\in Z^n_H(\mathcal{D}_H \otimes M_2(k),M)$ as in Corollary \ref{5.8a}. Let $[\tilde{\phi}]$ denote the cohomology class of $\tilde{\phi}$. Then, by \eqref{cohomclass}, we have $HC^\bullet_H(\alpha,M)([\tilde{\phi}])=HC^\bullet_H(\alpha',M)([\tilde{\phi}])$, i.e.,
\begin{equation}
\tilde{\phi}\circ (id_M\otimes CN_n(\alpha)) + B^n_H(\mathcal{D}_H,M)=\tilde{\phi}\circ  (id_M\otimes CN_n(\alpha'))  + B^n_H(\mathcal{D}_H,M)
\end{equation}
so that $\tilde{\phi}\circ (id_M\otimes CN_n(\alpha)) - \tilde{\phi}\circ (id_M\otimes CN_n(\alpha')) \in B^n_H(\mathcal{D}_H,M)$.
Applying the definition of $\tilde{\phi}$, we now have
\begin{equation*}
\begin{array}{ll}
&(\tilde{\phi}\circ (id_M\otimes CN_n(\alpha))) (m \otimes f^0 \otimes \ldots \otimes f^n)\\
&\quad =\tilde{\phi}\left(m \otimes \alpha(f^0) \otimes \ldots \otimes \alpha(f^n)\right)\\
& \quad = \tilde{\phi}\left(m \otimes (f^0 \otimes E_{11} + \upsilon(f^0) \otimes E_{22}) \otimes \ldots \otimes (f^n \otimes E_{11} + \upsilon(f^n) \otimes E_{22})\right)\\
& \quad = \phi(m \otimes f^0 \otimes \ldots \otimes f^n)+\phi(m \otimes \upsilon(f^0) \otimes \ldots \otimes \upsilon(f^n))
\end{array}
\end{equation*}
Similarly, $(\tilde{\phi}\circ (id_M\otimes CN_n(\alpha'))) (m \otimes f^0 \otimes \ldots \otimes f^n)=\phi(m \otimes \upsilon(f^0) \otimes \ldots \otimes \upsilon(f^n))$. Thus, $\phi=\tilde{\phi}\circ (id_M\otimes CN_n(\alpha)) - \tilde{\phi}\circ (id_M\otimes CN_n(\alpha'))\in B^n_H(\mathcal{D}_H,M)$. This proves the result.
\end{proof}

In particular, substituting $M=k=H$ in Proposition \ref{vanishing}, we obtain the following result:

\begin{corollary}\label{cor6.4x}
Let $\mathcal{C}$ be a small $k$-linear category. Suppose that there is a $k$-linear semifunctor $\nu :\mathcal{C} \longrightarrow \mathcal{C}$ and an $\eta \in \mathbb{U}(\mathcal{C}\otimes M_2(k))$ such that 
\begin{itemize}
\item[(1)] $\nu(X)=X \qquad \forall X \in Ob(\mathcal{C})$
\item[(2)] $\Phi_\eta(f \otimes E_{11}+\nu(f) \otimes E_{22})=\nu(f) \otimes E_{22}$
\end{itemize}
for all $f \in Hom_{\mathcal{C}}(X,Y)$ and $X,Y \in Ob(\mathcal{C})$. Then, $HC^\bullet(\mathcal{C})=0$.
\end{corollary}

\begin{definition}\label{def6.5}
Let $(\mathcal{S}_H,\hat{\partial}_H, \hat{\mathscr{T}}^H)$ be an $n$-dimensional $\mathcal{S}_H$-cycle with coefficients in a SAYD module  $M$ over $H$ (see, Definition \ref{cycle}). Then, we say that the cycle $(\mathcal{S}_H,\hat{\partial}_H, \hat{\mathscr{T}}^H)$ is vanishing if  $\mathcal{S}^0_H$ is a left $H$-category and $\mathcal S^0_H$ satisfies the assumptions in Proposition \ref{vanishing}.
\end{definition}

By taking $H=k=M$ in Definition \ref{def6.5}, we obtain the notion of a vanishing $\mathcal S$-cycle $(\mathcal{S},\hat\partial,\hat T)$ associated to a $k$-linear DG-semicategory $\mathcal S$. 

\smallskip
We now recall from \cite[p103]{C2}  the algebra $\mathbf C$ of infinite matrices $(a_{ij})_{i,j\in \mathbb N}$ with entries from $\mathbb C$ satisfying the following conditions
 (see also \cite{kv})
\begin{itemize}
\item[(i)] the set $\{\mbox{$a_{ij}$ $\vert$ $i,j\in \mathbb N$}\}$ is finite,
\item[(ii)] the number of non-zero entries in each row or each column is bounded.
\end{itemize}

Identifying $M_2(\mathbf{C})=\mathbf{C} \otimes M_2(\mathbb C)$, we recall the following result from \cite[p104]{C2}:
\begin{lemma}\label{6.6r}
There exists an algebra homomorphism $\omega:\mathbf{C} \longrightarrow \mathbf{C}$ and an invertible element $\tilde U \in M_2(\mathbf{C})$ such that the corresponding inner automorphism $\Xi:M_2(\mathbf{C}) \longrightarrow M_2(\mathbf{C})$ satisfies
\begin{equation}\label{6.6pqe}
\Xi(B \otimes E_{11}+\omega(B) \otimes E_{22})=\omega(B) \otimes E_{22} \qquad \forall B \in \mathbf{C}
\end{equation}
Then, $HC^\bullet(\mathbf{C})=0$.
\end{lemma}

\begin{remark}
We note that the condition in \eqref{6.6pqe} ensures that $\omega(\mathbf 1)\neq \mathbf 1$, where $\mathbf 1$ is the unit element of $\mathbf C$.
\end{remark}

For any $k$-algebra $\mathcal A$, we  may define a $k$-linear category $\mathcal{A} \otimes \mathcal{D}_H$ as follows:
\begin{equation*}
Ob(\mathcal{A} \otimes \mathcal{D}_H)=Ob(\mathcal{D}_H) \qquad
Hom_{\mathcal{A} \otimes \mathcal{D}_H}(X,Y)=\mathcal{A} \otimes Hom_{\mathcal{D}_H}(X,Y)
\end{equation*}

The category $\mathcal{A} \otimes \mathcal{D}_H$ is a left $H$-category via the action $h(a \otimes f):=a \otimes hf$ for any $h \in H$, $a \otimes f \in \mathcal{A} \otimes Hom_{\mathcal{D}_H}(X,Y)$.

\begin{lemma}\label{lem6.7x}
We have  $HC^\bullet_H(\mathbf{C} \otimes \mathcal{D}_H,M)=0$.
\end{lemma}
\begin{proof}
We will verify that the category $\mathbf{C} \otimes \mathcal{D}_H$ satisfies the assumptions of Proposition \ref{vanishing}. Let $\omega$ and $\tilde{U}$ be as in Lemma \ref{6.6r}. We now define $\upsilon:\mathbf{C} \otimes \mathcal{D}_H \longrightarrow \mathbf{C} \otimes \mathcal{D}_H$ given by
\begin{equation*}
\upsilon(X):=X \qquad \upsilon(B \otimes f):=\omega(B) \otimes f
\end{equation*}
for any $X \in Ob(\mathbf{C} \otimes \mathcal{D}_H)$ and $B \otimes f \in Hom_{\mathbf{C} \otimes \mathcal{D}_H}(X,Y)$. Since $\omega:\mathbf{C} \longrightarrow \mathbf{C}$ is an algebra homomorphism, it follows that $\upsilon$ is a semifunctor. By the definition of the $H$-action on $\mathbf{C} \otimes \mathcal{D}_H$, it is also clear that $\upsilon$ is $H$-linear.

\smallskip
Using the identification $\mathbf{C} \otimes \mathcal{D}_H \otimes M_2(\mathbb C) = M_2(\mathbf{C}) \otimes \mathcal{D}_H$, we now define an element $\eta \in \mathbb{G}(\mathbf{C} \otimes \mathcal{D}_H \otimes M_2(\mathbb C))=\mathbb{G}( M_2(\mathbf{C}) \otimes \mathcal{D}_H)$ given by the family of morphims
\begin{equation}\label{tqp}
\{\eta(X):=\tilde U \otimes id_X \in Hom_{M_2(\mathbf{C}) \otimes \mathcal{D}_H}(X,X)=M_2(\mathbf C) \otimes Hom_{\mathcal{D}_H}(X,X)\}_{X \in Ob(\mathcal{D}_H)}
\end{equation}
Since $\tilde U$ is a unit in $M_2(\mathbf C)$, it follows that each $\eta(X)$ in \eqref{tqp} is an automorphism. 
Since $H$ acts trivially on $M_2(\mathbf{C})$, we see that $\eta \in \mathbb{U}_H(\mathbf{C} \otimes \mathcal{D}_H \otimes M_2(\mathbb C))$. Moreover, for any $\tilde{B} \otimes f \in Hom_{M_2(\mathbf{C}) \otimes \mathcal{D}_H}(X,Y)=M_2(\mathbf{C}) \otimes Hom_{\mathcal{D}_H}(X,Y)$, we have
\begin{equation*}
\Phi_\eta(\tilde{B} \otimes f)=\eta(Y) \circ (\tilde{B} \otimes f) \circ \eta(X)^{-1}=(\tilde{U} \otimes id_Y) \circ (\tilde{B} \otimes f) \circ (\tilde{U}^{-1} \otimes id_X)=\tilde{U}\tilde{B}\tilde{U}^{-1} \otimes f=\Xi(\tilde{B}) \otimes f
\end{equation*}

\smallskip
Therefore,  for any $B \otimes f \in \mathbf{C} \otimes Hom_{\mathcal{D}_H}(X,Y)$, we have
\begin{equation*}
\begin{array}{ll}
\Phi_\eta((B \otimes f) \otimes E_{11} + \upsilon(B \otimes f) \otimes E_{22})&=\Phi_\eta(B \otimes f \otimes E_{11} + \omega(B) \otimes f \otimes E_{22})\\
&= \Phi_\eta(B \otimes E_{11} \otimes f + \omega(B) \otimes E_{22} \otimes f)\\
&= \Xi(B \otimes E_{11} + \omega(B) \otimes E_{22}) \otimes f\\
&=\omega(B) \otimes E_{22} \otimes f=\upsilon(B \otimes f) \otimes E_{22}
\end{array}
\end{equation*}
This proves the result.
\end{proof}

We now provide a useful interpretation of the space $B^n_H(\mathcal{D}_H,M)$.
\begin{prop}\label{cob1}
An element $\phi \in C^n_H(\mathcal{D}_H,M)$ is a coboundary iff $\phi$ is the character of an $n$-dimensional vanishing $\mathcal{S}_H$-cycle $(\mathcal{S}_H,\hat{\partial}_H, \hat{\mathscr{T}}^H,\rho)$ over $\mathcal{D}_H$. 
\end{prop}
\begin{proof}
Let $\phi$ be the character of an $n$-dimensional vanishing $\mathcal{S}_H$-cycle $(\mathcal{S}_H,\hat{\partial}_H, \hat{\mathscr{T}}^H, \rho)$. By definition, $\hat{\mathscr{T}}^H$ is an $n$-dimensional closed graded $(H,M)$-trace on the $H$-semicategory $\mathcal{S}_H$ and that
$\mathcal S_H^0$ is an ordinary $H$-category. We now define $\psi \in C^n_H(\mathcal{S}_H^0,M)$ by setting
\begin{equation*}
\psi(m \otimes g^0 \otimes \ldots \otimes g^n):=\hat{\mathscr{T}}^H_{X_0}\big(m \otimes g^0\hat{\partial}_H^0(g^1) \ldots \hat{\partial}_H^0(g^n)\big)
\end{equation*}
for $m \in M$ and $g^0 \otimes \ldots \otimes g^n \in Hom_{\mathcal{S}^0_H}(X_1,X_0) \otimes Hom_{\mathcal{S}^0_H}(X_2,X_1) \otimes \ldots \otimes Hom_{\mathcal{S}^0_H}(X_0,X_n)$. Then, by the implication (1) $\Rightarrow$ (3) in Theorem \ref{charcycl}, we have that $\psi \in Z^n_H(\mathcal{S}_H^0,M)$. Since $HC^n_H(\mathcal{S}_H^0,M)=0$, we have that $\psi =b\psi'$ for some $\psi' \in C^{n-1}_H(\mathcal{S}_H^0,M)$.

\smallskip
By Lemma \ref{6.3a}, the semifunctor $\rho \in Hom_{\overline{Cat}_H}(\mathcal{D}_H,\mathcal{S}_H^0)$ induces a map $C^{n-1}_H(\rho,M):C^{n-1}_H(\mathcal{S}_H^0,M)
\longrightarrow C_H^{n-1}(\mathcal D_H,M)$. Setting $\psi'':=C^{n-1}_H(\rho,M)(\psi')$,  we have
 \begin{equation*}
\left(\psi''\right)(m \otimes p^0 \otimes \ldots \otimes p^{n-1})=\psi'\left(m \otimes \rho(p^0) \otimes \ldots \otimes \rho(p^{n-1})\right)
\end{equation*}
for any $m \in M$ and $p^0 \otimes \ldots \otimes p^{n-1} \in CN_{n-1}(\mathcal{D}_H)$. Therefore,
\begin{equation*}
\begin{array}{ll}
\phi(m \otimes f^0 \otimes \ldots \otimes f^n)&=\hat{\mathscr{T}}^H_{X_0}\big(m \otimes \rho(f^0)\hat{\partial}_H^0\left(\rho(f^1)\right) \ldots \hat{\partial}_H^0\left(\rho(f^n)\right)\big)=\psi\left(m \otimes \rho(f^0) \otimes \ldots \otimes \rho(f^n)\right)\\
&= (b\psi')\left(m \otimes \rho(f^0) \otimes \ldots \otimes \rho(f^n)\right)=(b\psi'')(m \otimes f^0 \otimes \ldots \otimes f^n)\\
\end{array}
\end{equation*}
for any $m \in M$ and $f^0 \otimes \ldots \otimes f^n \in Hom_{\mathcal{D}_H}(X_1,X_0) \otimes Hom_{\mathcal{D}_H}(X_2,X_1) \otimes \ldots \otimes Hom_{\mathcal{D}_H}(X_0,X_n)$. Thus, $\phi \in B^n_H(\mathcal{D}_H,M)$.

\smallskip
Conversely, suppose that $\phi \in B^n_H(\mathcal{D}_H,M)$. Then, $\phi=b\psi$ for some $\psi \in C^{n-1}_H(\mathcal{D}_H,M)$. We now extend $\psi$ to get an element $\psi' \in C^{n-1}_H(\mathbf{C} \otimes \mathcal{D}_H,M)$ as follows:
\begin{equation*}
\psi'\left(m \otimes (B^0 \otimes f^0) \otimes \ldots \otimes (B^{n-1} \otimes f^{n-1})\right)=\psi(m \otimes B^0_{11}f^0 \otimes \ldots \otimes B^{n-1}_{11}f^{n-1})
\end{equation*}
We now set $\phi'=b\psi' \in Z^n_H(\mathbf{C} \otimes \mathcal{D}_H,M)$. We now consider the $H$-linear semifunctor  $
\rho:\mathcal{D}_H \longrightarrow \mathbf{C} \otimes \mathcal{D}_H$ which fixes objects and takes any morphism $
 f$ to $\mathbf 1 \otimes f$. Then, we have
\begin{equation*}
\begin{array}{ll}
\left(C^n_H(\rho,M)(\phi')\right)(m \otimes f^0 \otimes \ldots \otimes f^n)&=\phi'\left(m \otimes \rho(f^0) \otimes \ldots \otimes \rho(f^n)\right)=(b\psi')\left(m \otimes \rho(f^0) \otimes \ldots \otimes \rho(f^n)\right)\\
&=(b\psi)(m \otimes f^0 \otimes \ldots \otimes f^n)=\phi(m \otimes f^0 \otimes \ldots \otimes f^n)\\
\end{array}
\end{equation*}
Since $\phi'\in Z^n_H(\mathbf{C} \otimes \mathcal{D}_H,M)$, the implication (3) $\Rightarrow$ (2) in Theorem \ref{charcycl} gives us a  closed graded $(H,M)$-trace $\mathscr{T}^H$ of dimension $n$ on the DGH-semicategory $\left(\Omega(\mathbf{C} \otimes \mathcal{D}_H),\partial_H\right)$ such that
\begin{equation}\label{7.6tre}
\mathscr{T}^H_{X_0}\left(m \otimes \rho(f^0)\partial_H^0\left(\rho(f^1)\right) \ldots \partial_H^0\left(\rho(f^n)\right)\right)=
\phi'\left(m \otimes \rho(f^0) \otimes \ldots \otimes \rho(f^n)\right)
=\phi(m \otimes f^0 \otimes \ldots \otimes f^n)
\end{equation}
Since $\left(\Omega\left(\mathbf{C} \otimes \mathcal{D}_H\right)\right)^0=\mathbf C\otimes \mathcal D_H$ is a left $H$-category, we see that $\phi$ is the character associated to the cycle $\left(\Omega\left(\mathbf{C} \otimes \mathcal{D}_H\right),\partial_H,\mathscr{T}^H, {\rho} \right)$ over $\mathcal D_H$.

\smallskip 
 From the proof of Lemma \ref{lem6.7x}, we know that  $\mathbf{C} \otimes \mathcal{D}_H$ satisfies the assumptions in Proposition \ref{vanishing}. Hence, $\left(\Omega\left(\mathbf{C} \otimes \mathcal{D}_H\right),\partial_H,\mathscr{T}^H, \rho \right)$ is a vanishing cycle over $\mathcal D_H$. From this, the result follows.
\end{proof}

Substituting $H=k=M$, we have
\begin{corollary}\label{cor6.9r}
An element $\phi \in C^n(\mathcal{C})$ is a coboundary iff $\phi$ is the character of an $n$-dimensional vanishing $\mathcal{S}$-cycle $(\mathcal{S},\hat{\partial}, \hat{{T}},\rho)$ over $\mathcal{C}$. 
\end{corollary}

Our final aim is to use the method of categorified cycles and categorified vanishing cycles to obtain a pairing \begin{equation*}
HC^p(\mathcal{C}) \otimes HC^q(\mathcal{C}') \longrightarrow HC^{p+q}(\mathcal{C} \otimes \mathcal{C}')
\end{equation*} for $k$-linear categories $\mathcal C$ and $\mathcal C'$. 
Let $(\mathcal S,\hat\partial_\mathcal{S})$ and $(\mathcal S',\hat\partial_{\mathcal{S}'})$ be DG-semicategories. Then, their tensor product $\mathcal S \otimes \mathcal S'$ is the DG-semicategory defined as follows:
$$\begin{array}{ll}
Ob( \mathcal S \otimes \mathcal S')=Ob(\mathcal{S}) \times Ob(\mathcal{S}')\\
Hom_{\mathcal S \otimes \mathcal S'}^n\left((X,X'),(Y,Y')\right)=\bigoplus\limits_{i+j=n}Hom^i_\mathcal{S}(X,Y) \otimes_k Hom^j_{\mathcal{S}'}(X',Y')\\
\end{array}$$
The composition in $ \mathcal S \otimes  \mathcal S'$ is given by the rule:
\begin{equation*}
(g \otimes g')\circ (f \otimes f')=(-1)^{deg(g') deg(f)} (gf \otimes g'f')
\end{equation*}
for homogeneous $f:X \longrightarrow Y$, $g: Y \longrightarrow Z$ in $\mathcal{S}$ and $f':X' \longrightarrow Y'$, $g': Y' \longrightarrow Z'$ in $\mathcal{S}'$. The differential $\hat\partial^n_{\mathcal{S} \otimes \mathcal{S}'} :Hom_{\mathcal S \otimes \mathcal S'}^n\left((X,X'),(Y,Y')\right) \longrightarrow Hom_{\mathcal S \otimes  \mathcal S'}^{n+1}\left((X,X'),(Y,Y')\right)$ is determined by
\begin{equation*}
\hat\partial^n_{\mathcal{S} \otimes \mathcal{S}'}(f_i \otimes g_j)=\hat\partial^i_{\mathcal{S}}(f_i) \otimes g_j + (-1)^i f_i \otimes \hat\partial^j_{\mathcal{S}'}(g_j)
\end{equation*}
for any $f_i \in Hom^i_\mathcal{S}(X,Y)$ and $g_j \in Hom^j_{\mathcal{S}'}(X',Y')$ such that $i+j=n$. Clearly, $(\mathcal{S} \otimes \mathcal{S}')^0=\mathcal{S}^0 \otimes \mathcal{S}'^0$.

\begin{theorem}\label{Thmfin}
Let $\mathcal{C}$ and $\mathcal{C'}$ be small $k$-linear categories. Then, we have a pairing
\begin{equation*}
HC^p(\mathcal{C}) \otimes HC^q(\mathcal{C}') \longrightarrow HC^{p+q}(\mathcal{C} \otimes \mathcal{C}')
\end{equation*} for $p$, $q\geq 0$. 
\end{theorem}
\begin{proof}
Let $\phi \in Z^p(\mathcal{C})$ and $\phi' \in Z^q(\mathcal{C}')$. We may express $\phi$ and $\phi'$ respectively as the characters of $p$ and $q$-dimensional cycles $(\mathcal{S},\hat\partial, \hat{T},\rho)$ and $(\mathcal{S}',\hat\partial',\hat{T}',\rho')$ over $\mathcal{C}$ and $\mathcal{C}'$. We consider the collection $\hat{T} \# \hat{T}':=\{ {(\hat{T} \# \hat{T}')}_{(X,X')}:Hom^{p+q}_{\mathcal{S}\otimes \mathcal{S}'}\left((X,X'),(X,X')\right) \longrightarrow \mathbb{C}\}_{(X,X') \in Ob\left(\mathcal{S}\otimes  \mathcal{S}'\right)}$ of $\mathbb{C}$-linear maps defined by
\begin{equation*}
{(\hat{T} \# \hat{T}')}_{(X,X')}(f \otimes f'):=\hat{T}_X(f_p)\hat{T}'_{X'}(f'_q)
\end{equation*}
for any $f \otimes f'= (f_i \otimes f'_j)_{i+j=p+q} \in Hom^{p+q}_{\mathcal{S} \otimes \mathcal{S}'}\left((X,X'),(X,X')\right)$.
We will now prove that $\hat{T} \# \hat{T}'$ is a $p+q$-dimensional closed graded trace on the DG-semicategory $\mathcal{S}\otimes \mathcal{S}'$. For any $g \otimes g'= (g_i \otimes g'_j)_{i+j=p+q-1} \in Hom^{p+q-1}_{\mathcal{S} \otimes \mathcal{S}'}\left((X,X'),(X,Y)\right)$, we have
$$\begin{array}{ll}
(\hat{T} \# \hat{T}')_{(X,X')}\left(\hat\partial^{p+q-1}(g \otimes g')\right)&=\underset{i+j=p+q-1}{\sum}(\hat{T} \# \hat{T}')_{(X,X')}\left(\hat\partial^i_{\mathcal{S}}(g_i) \otimes g'_j + (-1)^{i} g_i \otimes \hat\partial^j_{\mathcal{S}'}(g'_j) \right)\\
&=\hat{T}_X(\hat\partial^{p-1}_\mathcal S(g_{p-1}))\hat{T}'_{X'}(g'_q)+(-1)^{p}\hat{T}_X(g_p)\hat{T}'_{X'}(\hat\partial^{q-1}_{\mathcal S'}(g'_{q-1}))=0
\end{array}$$

This proves the condition in \eqref{gh1}. Next for any homogeneous $f:X \longrightarrow Y$, $g: Y \longrightarrow X$ in $\mathcal{S}$ and $f':X' \longrightarrow Y'$, $g': Y' \longrightarrow X'$ in $\mathcal{S}'$, we have
$$\begin{array}{ll}
&(\hat{T} \# \hat{T}')_{(X,X')}\left((g \otimes g')(f \otimes f')\right)\\
&=(-1)^{deg(g')deg(f)}(\hat{T} \# \hat{T}')_{(X,X')}(gf \otimes g'f')\\
&=(-1)^{deg(g')deg(f)}\hat{T}_X((gf)_p)\hat{T}'_{X'}((g'f')_q)\\
&=(-1)^{deg(g')deg(f)}(-1)^{deg(g)deg(f)}(-1)^{deg(g')deg(f')}\hat{T}_Y((fg)_p)\hat{T}'_{Y'}((f'g')_q)\\
&=(-1)^{deg(g')deg(f)}(-1)^{deg(g)deg(f)}(-1)^{deg(g')deg(f')}(-1)^{deg(g)deg(f')}(\hat{T} \#\hat{T}')_{(Y,Y')}\left((f \otimes f')(g \otimes g')\right)\\
&=(-1)^{deg(g \otimes g')deg(f \otimes f')}(\hat{T} \# \hat{T}')_{(Y,Y')}\left((f \otimes f')(g \otimes g')\right)
\end{array}$$
This proves the condition in \eqref{gh2}. Thus, we obtain a $p+q$ dimensional cycle $\left(\mathcal{S} \otimes \mathcal{S}', \hat\partial_{\mathcal{S} \otimes \mathcal{S}'}, \hat{T}\#\hat{T}', \rho \otimes \rho'\right)$ over the category $\mathcal C \otimes \mathcal{C}'$. Then,  the character of this cycle, denoted by $\phi \# \phi' \in Z^{p+q}(\mathcal{C} \otimes \mathcal{C}')$,
gives  a well defined map $\gamma:Z^p(\mathcal{C}) \otimes Z^q(\mathcal{C}') \longrightarrow Z^{p+q}(\mathcal{C} \otimes \mathcal{C}')$. 

\smallskip
We now verify that the map $\gamma$ restricts to a pairing
$$\begin{array}{ll}
B^p(\mathcal{C}) \otimes Z^q(\mathcal{C}') \longrightarrow B^{p+q}(\mathcal{C} \otimes \mathcal{C}')
\end{array}$$
For this, we let $\phi \in Z^p(\mathcal{C})$ be the character of a $p$-dimensional vanishing cycle $\left(\mathcal{S},\hat{\partial},\hat{T}, \rho \right)$ over $\mathcal C$. In particular, it follows from Definition \ref{def6.5} that $\mathcal S^0$ is an ordinary category. From the implication (1) $\Rightarrow$ (2) in Theorem \ref{charcycl}, it follows that we might as well take $\mathcal{S'}^0$ to be an ordinary category. In fact, we could assume that $\mathcal S'=\Omega\mathcal C'$. Then, $\mathcal S^0\otimes \mathcal S'^0$ is an ordinary category. It suffices to show that the tuple $\left(\mathcal{S} \otimes \mathcal{S}', \hat\partial_{\mathcal{S} \otimes \mathcal{S}'}, \hat{T}\#\hat{T}', \rho \otimes \rho'\right)$ is a vanishing cycle.

\smallskip
Since $\left(\mathcal{S},{\hat\partial},\hat{T}\right)$ is a vanishing cycle, we have a $\mathbb{C}$-linear semifunctor $\upsilon: \mathcal{S}^0 \longrightarrow \mathcal{S}^0$ and an $\eta \in \mathbb{U}(\mathcal S^0 \otimes M_2(\mathbb{C}))$ satisfying the conditions in Corollary \ref{cor6.4x}. Extending  $\upsilon$, we get the the semifunctor $\upsilon \otimes id: \mathcal{S}^0 \otimes \mathcal{S}'^0 \longrightarrow \mathcal{S}^0 \otimes \mathcal{S}'^0$. Identifying, $\mathcal{S}^0 \otimes \mathcal S'^0 \otimes M_2(\mathbb{C})  \cong \mathcal{S}^0  \otimes M_2(\mathbb{C}) \otimes \mathcal S'^0$, we obtain $\tilde{\eta} \in \mathbb{U}(\mathcal{S}^0  \otimes M_2(\mathbb{C}) \otimes \mathcal S'^0)$ given by
\begin{equation*}
\{\tilde{\eta}(X,X')=\eta(X) \otimes id_{X'} \in Hom_{\mathcal{S}^0  \otimes M_2(\mathbb{C}) \otimes \mathcal S'^0}((X,X'),(X,X'))=Hom_{\mathcal{S}^0  \otimes M_2(\mathbb{C})}(X,X) \otimes Hom_{\mathcal{S}'^0}(X',X')\}
\end{equation*}
It may also be easily verified that
\begin{equation*}
\Phi_{\tilde{\eta}}(f \otimes f' \otimes E_{11} + (\upsilon \otimes id)(f \otimes f') \otimes E_{22})=(\upsilon \otimes id)(f \otimes f') \otimes E_{22}
\end{equation*}
Thus, we see that the category $(\mathcal S \otimes \mathcal S')^0=\mathcal{S}^0 \otimes \mathcal{S}'^0$   satisfies the conditions in  Corollary \ref{cor6.4x}. Therefore, the tuple $\left(\mathcal{S} \otimes \mathcal{S}', \hat\partial_{\mathcal{S} \otimes \mathcal{S}'}, \hat{T}\#\hat{T}', \rho \otimes \rho'\right)$ is a vanishing cycle.
This proves the result.
\end{proof}

\end{document}